\documentclass[11pt,a4paper]{amsart}
\usepackage[utf8]{inputenc}
\usepackage[a4paper, left=30mm, right=30mm, top=25mm, bottom=25mm]{geometry}
\usepackage[T1]{fontenc}
\usepackage[french,english]{babel}
\usepackage{amsmath}
\usepackage{amsfonts}
\usepackage{amssymb}
\usepackage{graphicx}
\graphicspath{{../Figures/}}
\usepackage{lmodern}
\usepackage{kpfonts}
\usepackage{xcolor}
\usepackage{url}
\usepackage[pdftex]{hyperref}
\usepackage{tikz}
\usetikzlibrary{decorations.pathreplacing}

\newtheorem{theorem}{Theorem}[subsection]

\newtheorem*{theorem*}{Theorem}
\newtheorem{corollary}[theorem]{Corollary}
\newtheorem*{corollary*}{Corollary}
\newtheorem{lemma}[theorem]{Lemma}
\newtheorem{lemmadefn}[theorem]{Lemma and Definition}

\theoremstyle{definition}
\newtheorem{definition}[theorem]{Definition}

\theoremstyle{remark}

\theoremstyle{remark}
\newtheorem{remark}[theorem]{Remark}

\numberwithin{equation}{section}

\catcode`\@=11
\def\eqalign#1{\null\,\vcenter{\openup1\jot \m@th %
\ialign{\strut\hfil$\displaystyle{##}\quad$&$\displaystyle{{}##}$\hfil\crcr#1\crcr}}\,}
\def\triplealign#1{\null\,\vcenter{\openup1\jot \m@th %
\ialign{\strut\hfil$\displaystyle{##}\quad$&$\displaystyle{{}##}$\hfil&$\displaystyle{{}##}$\hfil\crcr#1\crcr}}\,}
\def\multiline#1{\null\,\vcenter{\openup1\jot \m@th %
\ialign{\strut$\displaystyle{##}$\hfil&$\displaystyle{{}##}$\hfil\crcr#1\crcr}}\,}
\catcode`\@=12

\title[On $(2,1)$-dimensional anti-de Sitter spacetimes...]{On $(2,1)$-dimensional anti-de Sitter spacetimes with convex boundary}

\author[R. Prosanov, J.-M. Schlenker, G.\,A. Smith]{Roman Prosanov, Jean-Marc Schlenker, Graham Andrew Smith}
\address{Roman Prosanov: Departament de Matem\`atiques, Universitat Aut\`onoma de Barcelona, Edifici C, Facultat de Ci\`encies, 080193 Bellaterra, Spain}
\email{roman.prosanov@univie.ac.at}
\address{Jean-Marc Schlenker: Department of Mathematics, University of Luxembourg, Maison du nombre, 6 avenue de la Fonte, L-4364 Esch-sur-Alzette, Luxembourg}
\email{jean-marc.schlenker@uni.lu}
\address{Graham Andrew Smith: Departamento de Matem\'atica, Pontif\'\i cia Universidade Cat\'olica do Rio de Janeiro (PUC-Rio), Rio de Janeiro, Brasil} \email{grahamandrewsmith@gmail.com}

\newcommand{\opO}{\mathrm{O}}
\newcommand{\R}{\mathbb{R}}
\newcommand{\Mat}{\mathcal{M}(2,\R)}
\newcommand{\opmat}{{\mathrm{mat}}}
\newcommand{\opTr}{{\mathrm{Tr}}}
\newcommand{\slalg}{\mathfrak{sl}(2,\R)}
\newcommand{\SL}{\mathrm{SL}(2,\R)}
\newcommand{\TSL}{\mathrm{TSL}(2,\R)}
\newcommand{\USL}{\mathrm{U}^+\mathrm{SL}(2,\R)}
\newcommand{\PSL}{\mathrm{P}\mathrm{SL}(2,\R)}
\newcommand{\hyp}{\mathbb{H}}
\newcommand{\AdS}{\mathrm{AdS}^{2,1}}
\newcommand{\dS}{\mathrm{dS}^{1,1}}
\newcommand{\opT}{\mathrm{T}}

\newcommand{\geodesics}{\mathcal{G}^+}
\newcommand{\ideal}{\partial_\infty}
\newcommand{\rank}{\mathrm{Rank}}
\newcommand{\RP}{\mathrm{P}\R^1}
\newcommand{\Ein}{\mathrm{Ein}^{1,1}}
\newcommand{\opGL}{\mathrm{GL}}
\newcommand{\opdVol}{\mathrm{dVol}}
\newcommand{\Inv}{\mathrm{Inv}}
\newcommand{\opDet}{\mathrm{Det}}
\newcommand{\Stab}{\mathrm{Stab}}
\newcommand{\stab}{\mathrm{stab}}
\newcommand{\Conv}{\mathrm{Conv}}
\newcommand{\Length}{\mathrm{Length}}
\newcommand{\opU}{\mathrm{U}}
\newcommand{\minn}{{m\in\mathbb{N}}}
\newcommand{\opId}{\mathrm{Id}}

\newcommand{\teich}{\mathcal{T}}
\newcommand{\oprep}{\mathrm{rep}}
\newcommand{\Lip}{{\mathrm{Lip}}}

\newcommand{\cyl}{{\mathrm{cyl}}}
\newcommand{\pa}{{\mathrm{pa}}}

\newcommand{\mc}[1]{\mathcal{#1}}

\newcommand{\mb}[1]{\mathbb{#1}}
\newcommand{\mr}[1]{\mathrm{#1}}

\renewcommand{\hat}{\widehat}
\renewcommand{\tilde}{\widetilde}
\newcommand{\fff}{\mathrm{I}}
\newcommand{\sff}{\mathrm{I\!I}}
\newcommand{\tff}{\mathrm{I\!I\!I}}

\newcommand{\Rho}{\mathrm{R}}

\newcommand{\cML}{\mathcal{ML}}

\newcounter{notes}%

\newif\ifshowcomments
\showcommentsfalse

\begin{document}

\begin{abstract}
Let $S$ be a closed oriented surface of genus at least $2$. Let $X$ be a $(2,1)$-dimen\-sional anti-de Sitter spacetime homeomorphic to $S\times[-1,1]$ with smooth, spacelike, strong\-ly convex boundary. By Gauss' formula, the induced metric on each of the two components of $\partial X$ has curvature $K<-1$. In this paper we address the inverse problem. We show that, for any two smooth metrics $h_-, h_+$ on $S$ with curvature $K<-1$, there exists a unique, smooth anti-de Sitter metric on $X:=S\times [-1,1]$ with respect to which the boundary is locally strongly convex and has induced metric isotopic to $h_-$ on its past component, and isotopic to $h_+$ on its future component. Existence has already been proven by Tamburelli.

We use this result to solve affirmatively a conjecture of Mess: given any two measured laminations $l_-, l_+$ that together fill $S$, there exists a unique $(2,1)$-dimensional, globally hyperbolic, maximal, Cauchy compact (GHMC) anti-de Sitter spacetime such that the measured bending laminations of the past and future boundary components of its convex core are respectively $l_-$ and $l_+$. Here, existence has already been proven by Bonsante--Schlenker.
\end{abstract}

\maketitle
\section{Introduction}

\subsection{Main results}
\label{sec:results}
A $(2,1)$-dimensional \emph{anti-de Sitter spacetime} is a smooth Lorentzian spacetime $X$ of constant curvature equal to $-1$, that is, one which is locally modelled on $(2,1)$-dimen\-sional anti-de Sitter space $\AdS$.

Such a spacetime is said to be \emph{globally hyperbolic and Cauchy compact (GHC)} whenever it is orientable and time-orientable, and contains a closed, spacelike surface $\Sigma$ which intersects every inextendible causal curve exactly once.

\begin{definition}
  A GHC anti-de Sitter spacetime is
  \begin{itemize}
  \item {\em globally hyperbolic, maximal, and Cauchy compact} (GHMC) whenever it is open and maximal under inclusion, that is, whenever it cannot be isometrically embedded into a strictly larger GHC anti-de Sitter spacetime;
  \item {\em globally hyperbolic, convex, and Cauchy compact} (GHCC) whenever it is compact, with boundary the disjoint union of two smooth, strongly convex, spacelike surfaces.
  \end{itemize}
\end{definition}
\noindent Here, {\em strongly convex} means that the second fundamental form is positive definite (with respect to the appropriate orientation).

Any GHCC anti-de Sitter spacetime embeds isometrically, with geodesically convex image, into a unique GHMC anti-de Sitter spacetime (see \cite{Mes,BS2}). Conversely, any closed, geodesically convex subset with smooth, strongly convex boundary of a GHMC anti-de Sitter spacetime is GHCC.

By Gauss' formula in mixed signature, the induced metric on the boundary of a GHCC anti-de Sitter spacetime has curvature $K<-1$. Our first main result is a complete solution of the inverse problem.

\begin{theorem}
\label{rigiditytheorem}
Let $(S_1,h_1)$ and $(S_2,h_2)$ be marked, closed, oriented Riemannian surfaces of genus $g$, and of curvature strictly less than $-1$. There exists a $(2,1)$-dimensional marked $GHCC$ anti-de Sitter spacetime $X$ into which $(S_1,h_1)$ and $(S_2,h_2)$ embed respectively as the past and future boundaries. Furthermore $X$ is unique up to isometries preserving the marking.
\end{theorem}
\noindent Theorem \ref{rigiditytheorem} is proven in Section \ref{ss:GlobalRigidity}.
\begin{remark}Existence has already been proven by Tamburelli in \cite{Tam}.
\end{remark}

This statement has a simple ``physical'' interpretation: one can uniquely prescribe the induced metrics on two Cauchy surfaces in a $(2,1)$-dimensional GHMC AdS spacetime, one in the past and one in the future, under the condition that both metrics have curvature $K<-1$. In this realization, the past surface is future-convex, while the future surface is past-convex, and the entire ambient spacetime is prescribed by the two metrics. We note the analogy to the construction of solutions to the linear wave equation having prescribed initial and final configurations. From this point of view, Theorem \ref{rigiditytheorem} states that, subject to the above geometric restrictions, any final metric can be arrived at from any initial metric via an anti-de Sitter spacetime.

Our main result may be given globally in terms of moduli spaces as follows. For $g \geq 2$, let $\mathcal{S}:=\mathcal{S}_{g,-1}$ denote the space of closed, oriented, marked Riemannian surfaces of genus $g$ and of curvature strictly less than $-1$, up to isometries preserving the marking. Likewise, let $\mathcal{X}:=\mathcal{X}_g$ denote the space of oriented, marked $(2,1)$-dimensional GHCC anti-de Sitter spacetimes of genus $g$, up to isometries preserving the time orientation and the marking.
We furnish both these spaces with the quotient of $C^\infty$ topology. The restriction of the metric to the boundary yields a continuous map $\mathcal{I}:\mathcal{X}\rightarrow\mathcal{S}\times\mathcal{S}$.
Theorem \ref{rigiditytheorem} may now be given in the following, slightly stronger terms.
\begin{theorem} \label{diffeotheorem}
The map $\mathcal{I}$ is a homeomorphism.
\end{theorem}

This theorem answers a question that has been studied since the early 2010's, when the analogy between anti-de Sitter GHMC spacetimes and quasi-Fuchsian hyperbolic manifolds became an active area of research. It appears for instance as \cite[Question 3.5]{BBD+}. The main novelty of Theorems \ref{rigiditytheorem} and \ref{diffeotheorem} is injectivity, surjectivity having already been established by Tamburelli in~\cite{Tam} using a different approach. Attempts at proving injectivity, in particular using the approach developed for the hyperbolic case by the second author in \cite{Sch}, have to date proven unsuccessful because key estimates and regularity properties used in the hyperbolic case fail in the anti-de Sitter setting. In the present paper, we adopt a different, novel approach, which cleanly bypasses these difficulties, and hence applies in the anti-de Sitter case (more details are given in Section \ref{sec:rigidity}).

The key new element in the proofs of Theorems \ref{rigiditytheorem} and \ref{diffeotheorem} is the following infinitesimal rigidity result. Note first that $\mathcal{I}$ is trivially smooth in the Frechet sense.
\begin{theorem}
\label{infinitesimaltheorem1}
Every marked $(2,1)$-dimensional GHCC anti-de Sitter spacetime $X$ is infinitesimally rigid with respect to the induced metric of its boundary. That is, $d_X\mc I$ is injective.
\end{theorem}
\noindent A precise version of Theorem \ref{infinitesimaltheorem1} is stated in Theorem \ref{thm:InfinitesimalRigidity} and proven in Section \ref{ss:InfinitesimalRigidity}.

Theorem \ref{diffeotheorem} has the following direct corollary. Let $\mc J: \mc X\to \mc S\times \mc S$ denote the map which sends each GHCC anti-de Sitter structure to the third fundamental forms of its boundary components.
\begin{corollary} \label{cr:III}
The map $\mc J$ is a homeomorphism.
\end{corollary}
\noindent This result follows from Theorem \ref{diffeotheorem} through a well-known duality for strongly convex surfaces in GHMC anti-de Sitter spacetimes, which we now briefly recall (see, for example, \cite{BS2} and \cite{FS4} for more detailed discussions). We view $\AdS$ as the projectivized quadric
\begin{equation*}
\AdS := \big\{\mathbf{x}\in\mathbb{R}^{2,2}\ \big|\ \langle\mathbf{x},\mathbf{x}\rangle=-1\big\}/\{\pm 1\}\ .
\end{equation*}
Every spacelike, totally geodesic plane $P$ in $\AdS$ is the intersection of $\AdS$ with a hyperplane $\bar P$ of signature $(2,1)$ in $\R^{2,2}$. The unit normal of $\bar P$ in $\R^{2,2}$ is a point of $\AdS$, which we denote by $P^*$ and which we call the \emph{dual} of $P$. In a similar manner, we associate to every point $p\in\AdS$, a spacelike plane $p^*$ in $\AdS$, and we readily show that $(p^*)^*=p$.

This duality commutes with the action of isometries on $\AdS$ and thus also applies within every GHMC anti-de Sitter spacetime $X$. Given a smooth, strongly future-convex, spacelike surface $S\subset X$, its {\em dual}, defined to be the set of points dual to its support planes, is a smooth, strongly past-convex, spacelike surface $S^*\subset X$. Furthermore, this duality maps the induced metric on $S$ to the third fundamental form of $S^*$ and vice-versa. Corollary \ref{cr:III} now follows directly from Theorem \ref{diffeotheorem} using this duality.

Mess \cite{Mes} observed that GHMC anti-de Sitter spacetimes have remarkable similarities to quasi-fuchsian hyperbolic manifolds. In particular, every such spacetime has a well-defined convex core, given by the intersection of all geodesically convex subsets that are deformation retracts of the ambient space. Furthermore, in the non-Fuchsian case, the boundary of this convex core is the disjoint union of two pleated surfaces, each of which is equipped with an induced hyperbolic metric $m_\pm\in \mc T$ and is bent along a measured geodesic lamination $l_\pm\in \mc{ML}$. Here $\mc T:=\mathcal{T}_g$ denotes the Teichm\"uller space of surfaces of genus $g$ and $\mc{ML}:=\mathcal{ML}_g$ denotes the space of measured geodesic laminations over such surfaces. Mess conjectured that every pair of hyperbolic metrics $(m_-, m_+)\in \mc T \times \mc T $ may be uniquely obtained in this manner, and likewise that any pair $(l_-, l_+)\in \mc{ML} \times \mc{ML} $ of measured geodesic laminations that together fill the surface may be uniquely obtained in this manner \cite{Mes}. These are anti-de Sitter versions of Thurston's conjectures concerning quasi-Fuchsian hyperbolic 3-manifolds. We solve affirmatively the second conjecture of Mess, which may be viewed as a limiting case of Corollary \ref{cr:III}.

\begin{theorem} \label{tm:mess}
Let $(l_-, l_+)\in \mc{ML} \times \mc{ML} $ be measured geodesic laminations that together fill $S$. There exists a unique (non-Fuchsian) $(2,1)$-dimensional GHMC anti-de Sitter spacetime $X$ such that the measured bending laminations of the boundary components of its convex core are $l_-$ and $l_+$.
\end{theorem}
\noindent Existence has already been proven in \cite{BS3} by Bonsante--Schlenker.
We prove uniqueness in Section \ref{sc:mess}, following the strategy of the work \cite{DS} of Dular--Schlenker, with the key new ingredient being Corollary \ref{cr:III}.

Here too, we encounter a key difference with the hyperbolic setting, where Bonahon--Otal \cite{BO} were able to prove infinitesimal rigidity of quasi-Fuchsian hyperbolic 3-manifolds, with respect to the bending angles of the convex core, in the ``rational'' case, that is, in the case where the bending lamination is supported along the union of a finite family of disjoint simple closed curves (so-called ``multicurves''). There, using a doubling argument, they reduced the problem of rigidity of hyperbolic convex cores to that of certain hyperbolic cone-manifolds, as studied by Hodgson--Kerckhoff in  \cite{HK}. This technique does not apply in the present case, as there is no corresponding rigidity result for anti-de Sitter cone manifolds that we are aware of.

Theorem \ref{tm:mess} may also be stated in terms of boundary data for GHMC anti-de Sitter spacetimes in their entirety, rather than just for their convex cores. Indeed, GHMC anti-de Sitter spacetimes are non-complete, being bounded by a ``past singularity'' and a ``future singularity'', physically interpreted respectively as a ``big bang'' and a ``big crunch''. Each singularity is the quotient of a real tree by an action of the fundamental group, and is thus dual to a measured geodesic lamination. With the above described duality, the past singularity of a GHMC anti-de Sitter spacetime is dual to the measured bending lamination of the future boundary component of its convex core, whilst its future singularity is dual to the measured bending lamination of the past boundary component of its convex core. Denoting by $l^*$ the real tree with group action dual to the measured geodesic lamination $l$, we obtain the following reformulation of Theorem \ref{tm:mess}.

\begin{corollary} \label{cr:mess*}
Let $(l_-, l_+)\in \mc{ML} \times \mc{ML} $ be measured geodesic laminations that together fill $S$. There exists a unique $(2,1)$-dimensional GHMC anti-de Sitter spacetime $X$ with past singularity $l_+^*$ and future singularity $l_-^*$.
\end{corollary}

Theorem \ref{tm:mess} can also be stated purely in terms of earthquakes.
\begin{corollary} \label{cr:mess}
Let $l_-, l_+\in \mc{ML} $ be measured geodesic laminations that together fill $S$. The composition $E(l_-)\circ E(l_+): \mc T \to \mc T $ of the left earthquake along $l_+$ with the left earthquake along $l_-$ has a unique fixed point on $\mc T $.
\end{corollary}
\noindent This corollary is proven in Section \ref{sc:mess}, following the strategy of \cite{BS3}.

Finally, in Section~\ref{sec:landslides} we apply our results to solve in the affirmative a conjecture of Bonsante--Mondello--Schlenker concerning \emph{landslides} (see \cite{cyclic}, \cite{cyclic2}). This is the analogue in the landslide setting of Thurston's Earthquake theorem (see \cite{ThuE}).

\begin{corollary}\label{cr:landslides}
For all $\theta_-,\theta_+\in(0,\pi)$ and all $h_-,h_+\in\mc T $, the composition
\begin{equation*}
\mathcal S_{e^{i\theta_+},h_+}\circ\mathcal S_{e^{i\theta_-},h_-}:\mc T \longrightarrow\mc T\ .
\end{equation*}
has a unique fixed point.
\end{corollary}

\subsection{Context and related work}
Theorems \ref{rigiditytheorem} and \ref{tm:mess} are motivated by the analogies between, on the one hand, quasi-Fuchsian hyperbolic $3$-manifolds and hyperbolic ends, studied extensively since the work of Thurston, and, on the other, $(2,1)$-dimensional GHMC spacetimes, whose significance for Teichm\"uller theory has been understood since the work \cite{Mes} of Mess. There are $3$ types of GHMC spacetimes of interest to us, namely \emph{de Sitter spacetimes}, which are equivalent to hyperbolic ends via the duality described above (see \cite{FS4}), \emph{Minkowski spacetimes}, and \emph{anti-de Sitter spacetimes}.

The quasi-Fuchsian/de Sitter analogue of Theorem \ref{tm:mess} was first conjectured by Thur\-ston, with existence proven by Bonahon--Otal in \cite{BO}. Uniqueness was proven  by Bonahon in \cite{bonahon} for sufficiently small measured geodesic laminations, by Series in \cite{series:thurston} for the product of an interval and a punctured torus, and in the general case by Dular--Schlenker in \cite{DS}.

Thurston also stated a dual conjecture, namely that a quasi-Fuchsian manifold is uni\-que\-ly determined by the marked hyperbolic metrics of the two boundary components of its convex core. This led Labourie to study the quasi-Fuchsian/de Sitter analogue of Theorem \ref{rigiditytheorem} in \cite{Lab}, where he proved existence, with uniqueness being shown later by Schlenker in \cite{Sch}. Existence in Thurston's conjecture then follows by a limiting argument, whilst uniqueness remains open in general (although the first author proved uniqueness for an open, dense subset of the respective moduli space in \cite{Pro2}).

In the Minkowski case, the analogue of Theorem \ref{rigiditytheorem} was addressed by the third author in \cite{Smith2021Weyl}, building on ideas developed in collaboration with Fillastre in \cite{FillastreSmith2020}. A similar result for polyhedral surfaces was addressed by Fillastre and the first author in~\cite{FP}.

In all the above cases, the proof of existence has generally followed a similar strategy, whilst that of uniqueness has consistently posed a greater challenge, generally requiring distinct arguments in each setting.

Existence in Theorem \ref{rigiditytheorem} was shown by Tamburelli in \cite{Tam}. However, to date, only partial uniqueness results had been shown, with the general case remaining open. In \cite{cyclic2}, Bonsante--Mondello--Schlenker proved uniqueness in the case where the metrics $h_1$ and $h_2$ have constant curvatures $k_1$ and $k_2$ satisfying
\begin{equation*}
k_2 = \frac{k_1}{\left|k_1\right|-1}\ .
\end{equation*}
Uniqueness for Fuchsian spacetimes was shown by Labourie--Schlenker in \cite{LS}, and for spacetimes close to the Fuchsian locus by Prosanov--Schlenker in \cite{PS}. Similarly, uniqueness in Theorem \ref{tm:mess} was proven for sufficiently small measured geodesic laminations by Bonsante--Schlenker in \cite{BS3}. Theorems \ref{rigiditytheorem} and \ref{tm:mess} completely solve these problems.

Our results can also be viewed as a continuation of an ongoing tradition of extending the classical Weyl problem of differential geometry to locally homogeneous $3$-manifolds with non-trivial topology (see, for example, \cite{Sch5} for a recent survey). Among various questions in this vein, the closest related to our work is that of existence of a convex isometric embedding of a Riemannian surface with curvature $<-1$ into $(2,1)$-dimensional GHMC anti-de Sitter spacetimes with prescribed left or right part of the holonomy. Here, existence was addressed by Tamburelli in \cite{Tam} and uniqueness by Chen--Schlenker in \cite{CS}.

Finally, an analogue of Theorem~\ref{rigiditytheorem} for polyhedral surfaces was considered by the first author in~\cite{Pro}, where existence was established, while uniqueness was shown only under an additional ``smallness'' assumption on the boundary metrics. Additionally, in \cite{Lab2}, Labeni proved an existence result for metrics of curvature $\leq -1$ in the sense of Alexandrov (known as CAT($-1$)-metrics) in the case where $h_1=h_2$ and the spacetime has Fuchsian holonomy.

\subsection{Techniques}

The main novelty of this paper is the infinitesimal rigidity result of Theorem \ref{infinitesimaltheorem1}. Following the overall strategy of \cite{Sch}, we use the Pogorelov transform to reduce the problem to one of infinitesimal rigidity for convex subsets in Euclidean space. However, the semi-riemannian nature of the present setting presents us with two significant challenges.

The first challenge lies in proving that any anti-de Sitter first order variation field $\xi$ extends continuously up to the boundary in the Kleinian model. The argument of \cite{Sch}, which implicitly relies on the strict convexity of the ideal boundary of hyperbolic space, does not apply in the anti-de Sitter case. Instead, continuity must be proven via a careful study of the properties of timelike geodesics in anti-de Sitter space. Indeed, in this manner, the asymptotic behaviour of $\xi$ is described in terms of the action of the holonomy on the space of timelike geodesics, which is itself a riemannian manifold isometric to $\hyp^2\times\hyp^2$. This is carried out in Sections \ref{s:TimelikeGeodesics}, \ref{s:RepresentationsAndInvariantSets}, and \ref{s:AutomorphicVectorFields}.

The second challenge is more conceptual, and lies in establishing sufficient regularity of the first order variation field $\xi$ in order for the infinitesimal rigidity result proven by Vekua in \cite{Vek} to be applied. Continuity of $\xi$ is not, on the face of it, sufficient: Vekua's argument requires $W^{1,\infty}$-regularity, whilst it seems unlikely that we can hope for better than $C^{0,1-k}$-regularity, for $k>\kappa_\rho$, where $\kappa_\rho$ is the holonomy invariant defined in \eqref{eqn:DefinitionOfKappa}.

We address this in Section \ref{sec:rigidity} by showing how Vekua's argument may in fact be expressed in purely topological terms. Indeed, first, we show (c.f. Lemma \ref{lemma:LocalGuageEquivalence}) that this argument reduces to an application of the Poincar\'e--Hopf theorem. Next, we show, using Lemma \ref{lemma:NowhereVanishingProperty}, that the restriction of $\xi$ to the boundary curve $\Lambda$ is generically non-vanishing, and, in particular, has zero contribution to the Poincar\'e--Hopf degree. Since it is only along $\Lambda$ that regularity is a concern, infinitesimal rigidity is then proven upon applying these properties in the correct order.

\subsection{Acknowledgements}
A significant portion of this work was carried out whilst the first and third author were together visiting the Department of Mathematics ``Giuseppe Peano'' of the University of Turin, hosted by Andrea Seppi, and we are very grateful for the warm welcome and excellent working conditions provided during that stay. Parts of this work were also carried out whilst the first author was visiting the University of Luxembourg and the Mathematics Department of the Pontif\'\i cia Universidade Cat\'olica do Rio de Janeiro (PUC-Rio), and he is very grateful for the excellent working conditions provided by these institutions.

R. Prosanov's work was funded in whole by the Austrian Science Fund (FWF) \url{https://doi.org/10.55776/J4955}.

The authors acknowledge funding by the ERC-consol\-idator grant 101124349 GENERATE (GeomEtry and aNalysis for $(G,X)-$struc\-turEs and their defoRmATion spacEs). Views and opinions expressed are however those of the authors only and do not necessarily reflect those of the European Union or the European Research Council Executive Agency. Neither the European Union nor the granting authority can be held responsible for them.

For open access purposes, the authors have applied a CC BY public copyright license to any author-accepted manuscript version arising from this submission.

\subsection{GenAI statement}

GenAI was used to assist with creating figures and with proofreading earlier drafts of this paper. All mathematical arguments were developed \emph{without} the assistance of GenAI. 
\section{Preliminaries}
The main spaces used in this paper will be $(2,1)$-dimensional anti-de Sitter space, $1$-dimensional real projective space, and $(1,1)$-dimensional Einstein space. We will describe these spaces using the terminology of \cite[Section 2]{diaf2026weilpetersson}, and we refer the reader to that paper for a deeper discussion of the different models used.

\subsection{The Kleinian model}
\label{ss:PreliminariesKleinianModel}

\subsubsection{Anti-de Sitter and Einstein spaces}
\label{sss:AntideSitterandEinstein}
For all $(p,q)$, let $\mathbb{R}^{p,q}$ denote the vector space $\mathbb{R}^{p+q}$ endowed with the quadratic form
\begin{equation}
\label{eqn:MinkowskiForm}
\langle\mathbf{x},\mathbf{y}\rangle_{p,q} := -x_1y_1-\cdots-x_qy_q+x_{q+1}y_{q+1}+\cdots+x_{q+p}y_{q+p}\ .
\end{equation}
We define $(2,1)$-\emph{dimensional anti-de Sitter space} to be the projectivized quadric of vectors in $\mathbb{R}^{2,2}$ of norm-squared equal to $-1$, that is
\begin{equation}
\label{eqn:DefnAdSSpace}
\AdS := \big\{\mathbf{x}\in\mathbb{R}^{2,2}\ \big|\ \langle\mathbf{x},\mathbf{x}\rangle=-1\big\}/\{\pm 1\}\ .
\end{equation}
This is a Lorentzian manifold of constant sectional curvature equal to $-1$. Its geodesics are given by intersections with planes passing through the origin.

We define $(1,1)$-\emph{dimensional Einstein space} to be the projectivized light cone in $\mathbb{R}^{2,2}$, that is
\begin{equation}
\label{eqn:DefnEinsteinSpace}
\Ein := \big\{\mathbf{x}\in\mathbb{R}^{2,2}\ \big|\ \langle\mathbf{x},\mathbf{x}\rangle=0\big\}/\mathbb{R}\ .
\end{equation}
This is a conformal Lorentzian manifold. It is the topological boundary of $\AdS$ in projective space, and thus serves as a model for the ideal boundary $\ideal\AdS$ of $\AdS$.

\subsubsection{The Kleinian model}
\label{sss:KleinianModelDefinitions}
We will mainly work in Kleinian charts. The \emph{Kleinian projection} $\mathcal{K}:\mathbb{R}^{2,2}\rightarrow\mathbb{R}^{1,2}$ is defined by
\begin{equation}
\label{eqn:DefnKleinianProjection}
\mathcal{K}\big(x_0,\cdots,x_3\big) := \frac{1}{x_0}\big(x_1,x_2,x_3\big)\ .
\end{equation}
Its domain is
\begin{equation}
\mathrm{Dom}(\mathcal{K}) = \mathbb{R}^{2,2}\setminus\{x_0=0\}\ .
\end{equation}

The Kleinian projection yields maps from $\AdS$ and $\Ein$ into $\mathbb{R}^{2,1}$, which we also denote by $\mathcal{K}$. The domain of $\mathcal{K}$ in $\AdS$ is the complement of a totally geodesic copy of $\hyp^2$, and its image is
\begin{equation*}
\Omega := \big\{ \mathbf{y}\in\mathbb{R}^{2,1}\ \big|\ y_1^2+y_2^2-y_0^2 < 1\big\}\ .
\end{equation*}
Likewise, the domain of $\mathcal{K}$ in $\Ein$ is the complement of a spacelike copy of $\mathbb{S}^1$, and its image is
\begin{equation*}
\partial\Omega := \big\{ \mathbf{y}\in\mathbb{R}^{2,1}\ \big|\ y_1^2+y_2^2-y_0^2 = 1\big\}\ .
\end{equation*}
Note that $\partial\Omega$ identifies with $(1,1)$-dimensional de Sitter space.

We denote respectively by $\langle\cdot,\cdot\rangle_a$, $\langle\cdot,\cdot\rangle_m$, and $\langle\cdot,\cdot\rangle_e$ the anti-de Sitter, Min\-kowski, and Euclidian metrics over $\Omega$, we denote by $\|\cdot\|^2_a$, $\|\cdot\|^2_m$, and $\|\cdot\|^2_e$ their respective norms-squared, and we denote by $\times^a$, $\times^m$, and $\times^e$ their respective vector products.

We denote by $\Omega^t\subseteq\mathbb{R}^{2,1}$ the \emph{open light cone} in $\mathbb{R}^{2,1}$, that is, the set of all points which are in \emph{timelike relation} to the origin. We define \emph{hyperbolic polar coordinates} of this domain by
\begin{equation}
\label{eqn:defn:HyperbolicPolarCoordinates}
\Phi^t(r,\theta,t) := \big(r\cosh(t),r\cos(\theta)\sinh(t),r\sin(\theta)\sinh(t)\big)\ .
\end{equation}
The Minkowski and anti-de Sitter metrics are respectively given with respect to these coordinates over $\Omega^t$ by
\begin{equation}
\label{eqn:MetricsInHyperbolicCoordinates}
\begin{aligned}
\langle\cdot,\cdot\rangle_m &= -dr^2 + r^2d\Theta^2\ ,\ \text{and}\vphantom{\frac{1}{2}}\\
\langle\cdot,\cdot\rangle_a &= \frac{-1}{(1+r^2)^2}dr^2 + \frac{r^2}{(1+r^2)}d\Theta^2\ ,
\end{aligned}
\end{equation}
where
\begin{equation*}
d\Theta^2 = \sinh^2(t)d\theta^2 + dt^2
\end{equation*}
denotes the metric of $\hyp^2$ in polar coordinates about a point.

We denote by $\Omega^s\subseteq\mathbb{R}^{2,1}$ the complement of the closed light cone in $\mathbb{R}^{2,1}$, that is, the set of all points in $\mathbb{R}^{2,1}$ which are in \emph{spacelike relation} to the origin. We define \emph{pseudospherical polar coordinates} of this domain by
\begin{equation}
\label{eqn:defn:PseudoSphericalCoordinates}
\Phi^s(r,\theta,t):=\big(r\sinh(t),r\cos(\theta)\cosh(t),r\sin(\theta)\cosh(t)\big)\ .
\end{equation}
In these coordinates, the Minkowski and anti-de Sitter metrics are given respectively by
\begin{equation}
\label{eqn:MetricsInPseudosphericalCoordinates}
\begin{aligned}
\langle\cdot,\cdot\rangle_m &= dr^2 + r^2d\Theta^2\ ,\ \text{and}\vphantom{\frac{1}{2}}\\
\langle\cdot,\cdot\rangle_a &= \frac{1}{(1-r^2)^2}dr^2 + \frac{r^2}{(1-r^2)}d\Theta^2\ ,
\end{aligned}
\end{equation}
where here
\begin{equation*}
d\Theta^2 = \cosh^2(t)d\theta^2 - dt^2\ ,
\end{equation*}
denotes the metric of $\dS$ in Fermi coordinates about a totally geodesic spacelike circle.

We denote by $r:=r(\mathbf{y})$ the \emph{Minkowski distance} to the origin in $\mathbb{R}^{2,1}$, that is
\begin{equation}
\label{eqn:MinkowskiDistance}
r(\mathbf{y}) := \sqrt{\left|y_2^2+y_1^2-y_0^2\right|}\ .
\end{equation}
We likewise denote by $r_a$ the \emph{anti-de Sitter distance} to the origin in $\Omega$. Over $\Omega^s$,
\begin{equation}
\label{eqn:AntiDeSitterDistance}
r_a(\mathbf{y})=\mathrm{atanh}\big(r(\mathbf{y})\big)\ .
\end{equation}
Finally, in what follows, we will make use of the fundamental quantity
\begin{equation}
\label{eqn:DefinitioOfDelta}
\delta := \delta(\mathbf{y}) := (1-r^2) = (1-r)(1+r)\ .
\end{equation}
Note that, over $\Omega^s$, $\delta$ is a measure of the distance to $\partial\Omega$ in $\mathbb{R}^{2,1}$.

\subsection{The matrix model}
\label{ss:TheMatrixModel}
\subsubsection{The matrix model}
We will also make considerable use of the matrix models of anti-de Sitter space and Einstein space which we now describe. Let $\Mat$ denote the space of real $2\times 2$ matrices. This space carries a natural bilinear form given by the polarization of $-1$ times the determinant and which is given explicitly by
\begin{equation}
\label{eqn:defn:MatrixInnerProduct}
\langle\mathbf{M},\mathbf{N}\rangle_\opmat := \frac{1}{2}\opTr\big(\mathbf{M}\cdot\mathbf{J}_0\cdot\mathbf{N}^t\cdot\mathbf{J}_0\big)\ ,
\end{equation}
where
\begin{equation}
\label{eqn:defn:StandardComplexStructure}
\mathbf{J}_0 := \begin{pmatrix} 0 & -1\\ 1 & 0\end{pmatrix}
\end{equation}
denotes the standard complex structure over $\R^2$. Indeed, this is a consequence of the useful and readily verified formula
\begin{equation}
\label{eqn:DeterminantFormula}
\mathbf{J}_0\cdot\mathbf{M}\cdot\mathbf{J}_0\cdot\mathbf{M}^t = -\opDet(\mathbf{M})\cdot\opId\ ,
\end{equation}
valid for all $2\times 2$ matrices $\mathbf{M}$.

The bilinear form \eqref{eqn:defn:MatrixInnerProduct} is non-degenerate and of signature $(2,2)$, with a natural orthonormal basis given by the \emph{skew Pauli matrices}
\begin{equation}
\label{eqn:defn:TwistedPauliMatrices}
\sigma_0 := \begin{pmatrix} 1 & 0\\ 0 & 1\end{pmatrix}\ ,\qquad
\sigma_1 := \begin{pmatrix} 0 & -1\\ 1 & 0\end{pmatrix}\ ,\qquad
\sigma_2 := \begin{pmatrix} 1 & 0\\ 0 & -1\end{pmatrix}\ ,\qquad\text{and}\qquad
\sigma_3 := \begin{pmatrix} 0 & 1\\ 1 & 0\end{pmatrix}\ .
\end{equation}
The matrices $\sigma_0$ and $\sigma_1$ are timelike, and $\sigma_2$ and $\sigma_3$ are spacelike. The skew Pauli matrices are particularly useful owing to their similarity to the standard basis of the quaternion algebra. Indeed,
\begin{equation}
\label{eqn:BasicPauliRelationsI}
\sigma_1^2 = -1\qquad\text{and}\qquad\sigma_2^2 = \sigma_3^2 = 1\ .
\end{equation}
Furthermore, $\sigma_1$, $\sigma_2$, and $\sigma_3$ anticommute with one another, and
\begin{equation}
\label{eqn:BasicPauliRelationsII}
\sigma_1\cdot\sigma_2 = \sigma_3\ ,\qquad \sigma_1\cdot\sigma_3 = -\sigma_2\ ,\qquad\text{and}\qquad \sigma_2\cdot\sigma_3 = -\sigma_1\ .
\end{equation}

The Lie subalgebra $\slalg$ is generated as a vector space by $\sigma_1$, $\sigma_2$, and $\sigma_3$. We orient $\slalg$ such that the basis $(\sigma_1,\sigma_2,\sigma_3)$ is positive. This yields in particular a simple geometric visualization of the vector product. Recall first that, given a $3$-dimensional inner-product space with volume form, its \emph{vector product} is defined such that, for all $x$, $y$, and $z$,
\begin{equation}
\langle x\times y, z\rangle = \opdVol(x,y,z)\ .
\end{equation}
We verify by inspection that, for all $x,y\in\slalg$,
\begin{equation*}
[x,y] = 2\cdot x\times y\ .
\end{equation*}
The geometry of the vector product in mixed signature differs from that of the standard case. First, since $\sigma_1$ is timelike, $\sigma_1\times$ rotates the plane $\langle\sigma_2,\sigma_3\rangle$ by $\pi/2$, and we adopt the convention that this rotation acts in the positive (i.e. anticlockwise) direction. Since $\sigma_2$ and $\sigma_3$ are spacelike, $\sigma_2\times$ and $\sigma_3\times$ act respectively on the planes $\langle\sigma_1,\sigma_3\rangle$ and $\langle\sigma_1,\sigma_2\rangle$ by reflections which preserve the null spaces and which exchange spacelike and timelike vectors.

\subsubsection{Anti-de Sitter space}
\label{sss:AntideSitterSpaceInMatrixModel}
We denote by $\SL$ the \emph{special linear group}
\begin{equation*}
\SL := \big\{\mathbf{M}\ \big|\ \opDet(\mathbf{M})=1\big\}\ ,
\end{equation*}
and we denote by $\PSL$ its projectivization,
\begin{equation*}
\PSL := \SL/\{\pm 1\}\ .
\end{equation*}
The canonical projection
\begin{equation*}
\pi:\SL\rightarrow\PSL
\end{equation*}
is trivially a double cover.

We identify $(2,1)$-dimensional anti-de Sitter space $\AdS$ with $\PSL$, that is
\begin{equation}
\label{eqn:defn:AntideSitterSpace}
\AdS = \big\{ \mathbf{M}\ \big|\ \|\mathbf{M}\|^2_\opmat = -1\big\}/\{\pm 1\} = \PSL\ .
\end{equation}
We note however that it will be useful for many calculations to work with the double cover $\SL$.

We identify $2$-\emph{dimensional hyperbolic space} with the submanifold
\begin{equation}
\label{eqn:defn:TwoDimensionalHyperbolicSpace}
\hyp^2 := \big\{ \mathbf{M}\in\slalg\ \big|\ \|\mathbf{M}\|^2_\opmat<0\big\}/\R^*\ .
\end{equation}
It will likewise be useful for many calculations to work with the quadric
\begin{equation}
\tilde{\hyp}^2 := \big\{ \mathbf{M}\in\slalg\ \big|\ \|\mathbf{M}\|^2_\opmat = -1\ ,\ \langle\mathbf{M},\mathbf{J}_0\rangle_\opmat < 0 \big\}\ .
\end{equation}
Finally, note that, for all $\mathbf{M}\in\tilde{\hyp}^2$, and for any $\mathbf{N}$ tangent to $\tilde{\hyp}^2$ at this point
\begin{equation}
\label{eqn:CommutatorIsCrossProduct}
\mathbf{M}\cdot\mathbf{N} = \frac{1}{2}[\mathbf{M},\mathbf{N}] = \mathbf{M}\times\mathbf{N}\ ,
\end{equation}
which, by convention, coincides with rotation by $\pi/2$ in the positive direction.

We view the \emph{tangent bundle} of $\SL$ as the submanifold of $\Mat\times\Mat$ given by
\begin{equation*}
\TSL = \big\{ (\mathbf{M},\mathbf{N})\ \big|\ \|\mathbf{M}\|^2_\opmat = -1\ ,\ \langle\mathbf{M},\mathbf{N}\rangle_\opmat = 0\big\}\ ,
\end{equation*}
and we view the future oriented \emph{unit timelike tangent bundle} as the submanifold
\begin{equation*}
\USL := \big\{ (\mathbf{M},\mathbf{N})\ \big|\ \|\mathbf{M}\|^2_\opmat = \|\mathbf{N}\|^2_\opmat = -1\ ,\ \langle\mathbf{M},\mathbf{N}\rangle_\opmat = 0\ ,\ \text{and}\ \opTr(\mathbf{M}^{-1}\cdot\mathbf{N}\cdot\mathbf{J}_0) < 0 \big\}\ .
\end{equation*}
The \emph{geodesic flow} over $\USL$ is
\begin{equation*}
\Phi_\theta(\mathbf{M},\mathbf{N}) := \big(\cos(\theta)\mathbf{M} + \sin(\theta)\mathbf{N},-\sin(\theta)\mathbf{M} + \cos(\theta)\mathbf{N}\big)\ .
\end{equation*}

We denote by $\geodesics$ the space of future oriented timelike geodesics in $\AdS$. Since every element $\zeta$ of $\geodesics$ lifts to a unique timelike geodesic $\tilde{\zeta}$ in $\SL$, which is furthermore a double cover of $\zeta$, we may equally view $\geodesics$ as the space of future oriented timelike geodesics in $\SL$.

We define the \emph{right} and \emph{left translations} $\tau^r,\tau^l:\TSL\rightarrow\slalg$ by
\begin{equation}
\label{eqn:defn:LeftAndRightTranslationsI}
\tau^r(\mathbf{M},\mathbf{N}) := \mathbf{N}\cdot\mathbf{M}^{-1}\qquad\text{and}\qquad
\tau^l(\mathbf{M},\mathbf{N}) := \mathbf{M}^{-1}\cdot\mathbf{N}\ ,
\end{equation}
and we define the \emph{full translation} $\tau:\TSL\rightarrow\slalg\times\slalg$ by
\begin{equation}
\label{eqn:defn:LeftAndRightTranslationsII}
\tau := \big(\tau^r,\tau^l\big)\ .
\end{equation}
We verify by inspection that $\tau$ is invariant under the geodesic flow, that is, for all $\theta$,
\begin{equation*}
\tau\circ\Phi_\theta = \tau\ ,
\end{equation*}
so that $\tau$ descends to a map from $\geodesics$ to $\hyp^2\times\hyp^2$ which we also denote by $\tau$. This map is a smooth diffeomorphism (see, for example, \cite{barbot2018lorentzian}), and thus defines a smooth parametrization of $\mathcal{G}^+$ by $\hyp^2\times\hyp^2$.

The natural action of $\SL\times\SL$ on $\Mat$ is
\begin{equation}
\label{eqn:defn:ActionOfSLTimesSL}
\hat{\alpha}(\mathbf{X},\mathbf{Y})\cdot\mathbf{M} := \mathbf{X}\cdot\mathbf{M}\cdot\mathbf{Y}^{-1}\ ,
\end{equation}
and its action on $\USL$ is
\begin{equation}
\label{eqn:defn:ActionOfSLTimesSL}
\hat{\alpha}(\mathbf{X},\mathbf{Y})\cdot(\mathbf{M},\mathbf{N}) := \big(\mathbf{X}\cdot\mathbf{M}\cdot\mathbf{Y}^{-1},\mathbf{X}\cdot\mathbf{N}\cdot\mathbf{Y}^{-1}\big)\ .
\end{equation}
The \emph{adjoint action} of $\SL$ on $\slalg$ is
\begin{equation}
\label{eqn:defn:AdjointAction}
\alpha(\mathbf{X})\cdot\mathbf{M} = \mathbf{X}\cdot\mathbf{M}\cdot\mathbf{X}^{-1}\ .
\end{equation}

\begin{lemma}
The full translation $\tau$ intertwines $\hat{\alpha}$ and $\alpha\times\alpha$, that is, for all $\mathbf{X},\mathbf{Y}\in\SL$, and for all $(\mathbf{M},\mathbf{N})\in\USL$,
\begin{equation}
\tau\big(\hat{\alpha}(\mathbf{X},\mathbf{Y})\cdot(\mathbf{M},\mathbf{N})\big)
=\big(\alpha(\mathbf{X})\cdot\tau^r(\mathbf{M},\mathbf{N}),\alpha(\mathbf{Y})\cdot\tau^l(\mathbf{M},\mathbf{N})\big)\ .
\end{equation}
\end{lemma}
\begin{proof}
Indeed,
\begin{align*}
\tau^r\big(\hat{\alpha}(\mathbf{X},\mathbf{Y})\cdot(\mathbf{M},\mathbf{N})\big)
&=\tau^r\big(\mathbf{X}\cdot\mathbf{M}\cdot\mathbf{Y}^{-1},\mathbf{X}\cdot\mathbf{N}\cdot\mathbf{Y}^{-1}\big)\\
&=\mathbf{X}\cdot\mathbf{N}\cdot\mathbf{Y}^{-1}\cdot\mathbf{Y}\cdot\mathbf{M}^{-1}\cdot\mathbf{X}^{-1}\\
&=\mathbf{X}\cdot\mathbf{N}\cdot\mathbf{M}^{-1}\cdot\mathbf{X}^{-1}\\
&=\alpha(\mathbf{X})\cdot\tau^r\big(\mathbf{M},\mathbf{N}\big)\ .
\end{align*}
The left translation is addressed similarly, and this completes the proof.
\end{proof}

\subsubsection{Einstein space}
\label{sss:EinsteinSpaceI}
We identify $(1,1)$-dimensional Einstein space with the projectivized null quadric
\begin{equation}
\label{eqn:defn:IdealBoundaryOfAdS}
\Ein = \big\{ \mathbf{M}\ \big|\ \opDet(\mathbf{M}) = 0\ ,\ \mathbf{M}\neq 0 \big\}/\R^* = \big\{ \mathbf{M}\ \big|\ \rank(\mathbf{M}) = 1 \big\}/\mathbf{R}^*\ .
\end{equation}
Observe that the projective class of any $2\times 2$ rank $1$ matrix is determined entirely by its kernel and cokernel. We thus define $\hat{\mathcal{N}}:\R^2\times\R^2\rightarrow\Mat$ by
\begin{equation*}
\hat{\mathcal{N}}(x,y) := x\cdot(J_0\cdot y)^t\ ,
\end{equation*}
where here we use the formalism that the matrix product of a vertical vector followed by a horizontal vector yields a matrix. This map quotients to a diffeomorphism
\begin{equation*}
\mathcal{N}:\RP\times\RP\rightarrow\ideal\AdS\ .
\end{equation*}

We define the action $\alpha$ of $\SL$ on $\RP$ by
\begin{equation*}
\label{eqn:defn:ActionOnRP}
\alpha(\mathbf{X})\cdot[x] = [\mathbf{X}\cdot x]\ .
\end{equation*}

\begin{lemma}
The function $\mathcal{N}$ intertwines $\alpha\times\alpha$ and $\hat{\alpha}$, that is, for all $\mathbf{X},\mathbf{Y}\in\SL$, and for all $([x],[y])\in\RP\times\RP$,
\begin{equation}
\mathcal{N}\big(\alpha(\mathbf{X})\cdot[x],\alpha(\mathbf{Y})\cdot[y]\big) = \hat{\alpha}(\mathbf{X},\mathbf{Y})\cdot\mathcal{N}\big([x],[y]\big)\ .
\end{equation}
\end{lemma}
\begin{proof} Indeed,
\begin{equation*}
\mathcal{N}\big(\alpha(\mathbf{X})\cdot[x],\alpha(\mathbf{Y})\cdot[y]\big)
=\mathcal{N}\big([\mathbf{X}\cdot x],[\mathbf{Y}\cdot y]\big) = [(\mathbf{X}\cdot x)\cdot(\mathbf{J}_0\cdot\mathbf{Y}\cdot y)^t]\ .
\end{equation*}
However, since $\opDet(\mathbf{Y})=1$, by \eqref{eqn:DeterminantFormula},
\begin{equation*}
\mathbf{J}_0\cdot\mathbf{Y} = (\mathbf{Y}^{-1})^t\cdot\mathbf{J}_0\ .
\end{equation*}
Hence
\begin{align*}
\mathcal{N}\big(\alpha(\mathbf{X})\cdot[x],\alpha(\mathbf{Y})\cdot[y]\big)
&= [(\mathbf{X}\cdot x)\cdot((\mathbf{Y}^{-1})^t\cdot\mathbf{J}_0\cdot y)^t]\\
&= [(\mathbf{X}\cdot x)\cdot(y^t\cdot\mathbf{J}_0^t\cdot\mathbf{Y}^{-1})]\\
&= [\mathbf{X}\cdot(x\cdot(\mathbf{J}_0\cdot y)^t)\cdot\mathbf{Y}^{-1}]\\
&= [\mathbf{X}\cdot\hat{\mathcal{N}}(x,y)\cdot\mathbf{Y}^{-1}]\\
&= \hat{\alpha}(\mathbf{X},\mathbf{Y})\cdot\hat{\mathcal{N}}(x,y)\ ,
\end{align*}
as desired.
\end{proof}

We identify the ideal boundary of $\hyp^2$ with the projective submanifold
\begin{equation}
\label{eqn:defn:IdealBoundaryOfHyperbolicSpace}
\ideal\hyp^2 := \big\{ \mathbf{M}\in\slalg\ \big|\ \opDet(\mathbf{M})=0\ ,\ \mathbf{M}\neq 0\big\}/\R^*\ .
\end{equation}
There is a unique map $\tau_\infty:\Ein\rightarrow\ideal\hyp^2\times\ideal\hyp^2$ which intertwines $\hat{\alpha}$ and $\alpha\times\alpha$. This map extends $\tau$ in the sense that, for any hyperbolic element $\gamma:=\big(\gamma^r,\gamma^l\big)\in\SL\times\SL$,
\begin{equation}
\label{eqn:defn:EquivarianceOfLimitI}
\tau_\infty(\gamma^+) = \big(\gamma^{r,+},\gamma^{l,+}\big)\ ,
\end{equation}
where $\gamma^+$, $\gamma^{l,+}$, and $\gamma^{r,+}$ here denote the respective attractive fixed points of $\gamma$, $\gamma^l$, and $\gamma^r$ in $\Ein$, $\ideal\mathbf{H}^2$ and $\ideal\mathbf{H}^2$. Likewise, the composition $(\tau_\infty\circ\mathcal{N})$ is $(\alpha\times\alpha)$-equivariant so that, for all such $\gamma$,
\begin{equation}
\label{eqn:defn:EquivarianceOfLimitII}
(\tau_\infty\circ\mathcal{N})\big(\gamma_{\mathrm{rp}}^{r,+},\gamma_{\mathrm{rp}}^{l,+}\big) = \big(\gamma^{r,+},\gamma^{l,+}\big)\ ,
\end{equation}
where $\gamma_{\mathrm{rp}}^{r,+}$ and $\gamma_{\mathrm{rp}}^{r,+}$ here denote respectively the attractive fixed points of $\gamma^r$ and $\gamma^l$ in $\RP$.

\subsection{Homogeneous spaces}
We will make considerable use of Killing vector fields and automorphic vector fields in what follows. We thus now recall the general framework of homogeneous spaces.

\subsubsection{Killing fields and automorphic vector fields}
Let $G$ be a Lie group with Lie algebra $\mathfrak{g}$, let $H$ be a Lie subgroup with Lie algebra $\mathfrak{h}$, let $M:=G/H$ denote the quotient space, and let $\pi:G\rightarrow M$ denote the canonical projection. We will call $G$ and $H$ respectively the \emph{structure group} and \emph{factor group} of the homogeneous space $M$.

It is helpful to view $G$ as a subgroup of $\opGL(N)$ and to use the formalism of matrix multiplication. Thus, for $\mathbf{M}\in G$ and $\mathbf{A}\in\mathfrak{g}$, we denote the \emph{left} and \emph{right translations} of $\mathbf{A}$ at $\mathbf{M}$ respectively by
\begin{equation*}
\mathbf{A}^l(\mathbf{M}):=\mathbf{M}\cdot\mathbf{A}\qquad\text{and}\qquad\mathbf{A}^r(\mathbf{M}):=\mathbf{A}\cdot\mathbf{M}\ .
\end{equation*}

Killing vector fields over $M$ are given by the infinitesimal action of $\mathfrak{g}$ on $M$. For every element $\mathbf{A}$ of $\mathfrak{g}$, its \emph{Killing vector field} $\xi_{\mathbf{A}}$ is defined by
\begin{equation}
\label{eqn:DefnKillingField}
\xi_{\mathbf{A}}\circ\pi(\mathbf{M}) := \frac{d}{dt}\pi\big(e^{t\mathbf{A}}\cdot\mathbf{M}\big)\bigg|_{t=0} = D\pi(x)\cdot\big(\mathbf{A}\cdot\mathbf{M}\big)\ .
\end{equation}

Let $\Pi$ be a compact surface group, and let $\rho:\Pi\rightarrow G$ be a representation. A $\rho$-\emph{cocycle} is a function $\alpha:\Pi\rightarrow\mathfrak{g}$ which satisfies, for all $\gamma,\eta\in\Pi$,
\begin{equation}
\label{eqn:DefnCocycle}
\alpha(\gamma\cdot\eta) = \alpha(\gamma) + \rho(\gamma)\cdot\alpha(\eta)\cdot\rho(\gamma)^{-1}\ .
\end{equation}
Cocycles are infinitesimal variations of representations. In particular, we associate a $\rho$-cocycle to every element $\mathbf{A}$ of $\mathfrak{g}$ by
\begin{equation}
\label{eqn:DefnCoboundary}
\alpha_{\mathbf{A}}(\gamma) := \frac{d}{dt}e^{t\mathbf{A}}\cdot\rho(\gamma)\cdot e^{-t\mathbf{A}}\bigg|_{t=0}\cdot\rho(\gamma)^{-1} = \mathbf{A} - \rho(\gamma)\cdot\mathbf{A}\cdot\rho(\gamma)^{-1}\ .
\end{equation}
We call $\alpha_{\mathbf{A}}$ the $\rho$-\emph{coboundary} of $\mathbf{A}$.

Given a representation $\rho:\Pi\rightarrow G$ and a $\rho$-cocycle $\alpha$, we define a $(\rho,\alpha)$-\emph{automorphic vector field} over $M$ to be a vector field $\xi$ satisfying, for all $\gamma\in\Pi$ and for all $p\in M$,
\begin{equation}
\label{eqn:DefnAutomorphic}
\xi\big(\rho(\gamma)\cdot p\big) = D\rho(\gamma)(p)\cdot\xi(p) + \xi_{\alpha(\gamma)}\big(\rho(\gamma)\cdot p\big)\ ,
\end{equation}
where $\xi_{\alpha(\gamma)}$ is as in \eqref{eqn:DefnKillingField}.

\begin{lemma}
For all $\mathbf{A}\in\mathfrak{g}$, the Killing vector field $\xi_{\mathbf{A}}$ is $(\rho,\alpha_{\mathbf{A}})$-automorphic.
\end{lemma}
\begin{proof} Choose $\mathbf{M}\in G$ and denote $p:=\pi(\mathbf{M})$. For all $\gamma\in\Pi$,
\begin{equation*}
\begin{aligned}
\xi_{\mathbf{A}}\big(\rho(\gamma)\cdot p\big)
&=\frac{d}{dt}\pi\big(e^{t\mathbf{A}}\cdot\rho(\gamma)\cdot\mathbf{M}\big)\bigg|_{t=0}\\
&=\frac{d}{dt}\big(\big(e^{t\mathbf{A}}\cdot\rho(\gamma)\cdot e^{-t\mathbf{A}}\big)\cdot\big(e^{t\mathbf{A}}\cdot\mathbf{M}\big)\big)\bigg|_{t=0}\\
&=D\pi\big(\rho(\gamma)\cdot\mathbf{M}\big)\cdot\big([\mathbf{A}-\rho(\gamma)\cdot\mathbf{A}\cdot\rho(\gamma)^{-1}]\cdot\big(\rho(\gamma)\cdot\mathbf{M}\big)\big)\\
&\qquad\qquad +D\pi\big(\rho(\gamma)\cdot\mathbf{M}\big)\cdot\big(\rho(\gamma)\cdot\mathbf{A}\cdot\mathbf{M}\big)\vphantom{\frac{1}{2}}\ .
\end{aligned}
\end{equation*}
This first term is equal to $\xi_{\alpha_{\mathbf{A}}(\gamma)}(\rho(\gamma)\cdot p)$. For the second term, differentiating the identity $\pi\circ\rho(\gamma)=\rho(\gamma)\circ\pi$ yields
\begin{equation*}
D\pi\big(\rho(\gamma)\cdot\mathbf{M}\big)\circ D\rho(\gamma)(\mathbf{M}) = D\rho(\gamma)\big(\pi(\mathbf{M})\big)\circ D\pi(\mathbf{M})\ .
\end{equation*}
Hence
\begin{equation*}
\begin{aligned}
D\pi\big(\rho(\gamma)\cdot\mathbf{M}\big)\cdot\big(\rho(\gamma)\cdot\mathbf{A}\cdot\mathbf{M}\big)
&= D\pi\big(\rho(\gamma)\cdot\mathbf{M}\big)\cdot D\rho(\gamma)(\mathbf{M})\cdot(\mathbf{A}\cdot\mathbf{M})\\
&= D\rho(\gamma)\big(\pi(\mathbf{M})\big)\cdot D\pi(\mathbf{M})\cdot(\mathbf{A}\cdot\mathbf{M})\\
&= D\rho(\gamma)(p)\cdot\xi_{\mathbf{A}}(p)\ ,
\end{aligned}
\end{equation*}
and the result follows.
\end{proof}

\subsubsection{Real projective space} \label{sec:homspacerps} We now apply the preceding formalism to the case where $M$ is $1$-dimensional real projective space. Identifying $\RP$ with the extended real line $\mathbb{R}\cup\{\infty\}$, the structure group is $G=\PSL$ acting on this space by fractional linear transformations. We may take the factor group to be the stabilizer of $\infty$, which identifies with the space of affine transformations of the form
\begin{equation*}
\phi(x) = ax + b\ ,
\end{equation*}
for $a,b\in\mathbb{R}$ and $a>0$.

We will require the following elementary vanishing result for Killing vector fields.
\begin{lemma}
\label{lemma:VanishingConditionOfKillingFieldsI}
Let $p_1,p_2,p_3\in\RP$ be distinct points. Let $\xi=\xi_{\mathbf{A}}$ be a Killing vector field. If $\xi(p_i)=0$ for all $i$, then $\xi=0$.
\end{lemma}
\begin{proof} Upon applying an element of $\PSL$ if necessary, we may suppose that $p_3=\infty$. Since $\xi(\infty)=0$,
\begin{equation*}
\xi\in\mathfrak{h}=\stab(\infty)\ ,
\end{equation*}
so that
\begin{equation*}
\xi(x) = ax + b\ ,
\end{equation*}
for some $a,b\in\mathbb{R}$. Since $\xi(p_1)=\xi(p_2)=0$, it follows that $a=b=0$, and hence $\xi=0$, as desired.
\end{proof}

Let $\Pi$ be the fundamental group of a closed surface $S$. Recall that a representation $\rho:\Pi\rightarrow\PSL$ is said to be \emph{Fuchsian} whenever it is discrete and faithful. Recalling also that $\PSL$ naturally acts on $\hyp^2$, for any Fuchsian representation $\rho:\Pi\rightarrow\PSL$, we define the \emph{translation length} $l_\rho:\Pi\rightarrow[0,\infty[$ by
\begin{equation}
\label{eqn:defn:LengthFunction}
l_\rho(\gamma) := \inf_{p\in\hyp^2}d\big(\rho(\gamma)\cdot p,p\big)\ .
\end{equation}
Recall that, for all $\gamma$,
\begin{equation*}
\cosh\big(l_\rho(\gamma)/2\big) = \frac{1}{2}\left|\opTr(\rho(\gamma))\right|\ .
\end{equation*}

\begin{lemma}
Let $\rho:\Pi\rightarrow\PSL$ be Fuchsian, let $\alpha$ be a $\rho$-cocycle, and let $\xi$ be a $(\rho,\alpha)$-automorphic vector field. For all $\gamma\in\Pi\setminus\{\opId\}$, denoting by $p:=\rho(\gamma)^+$ the attractive fixed point of $\rho(\gamma)$ in $\RP$,
\begin{equation}
\label{eqn:FormulaForAutomorphicVectorField}
\xi(p) = \frac{1}{\big(1-e^{-l_\rho(\gamma)}\big)}\xi_{\alpha(\gamma)}(p)\ .
\end{equation}
\end{lemma}
\begin{proof}Since $\rho$ is Fuchsian, $l_\rho(\gamma)>0$ for all non-trivial $\gamma$, and hence \eqref{eqn:FormulaForAutomorphicVectorField} is always finite. Upon conjugating with a suitable element of $\PSL$, we may suppose that the repulsive and attractive fixed points of $\rho(\gamma)$ are $\infty$ and $0$ respectively, and so
\begin{equation*}
\rho(\gamma)\cdot z = e^{-l_\rho(\gamma)}z\ .
\end{equation*}
Hence, for any $\xi$ tangent to $\RP$ at $p=0$,
\begin{equation*}
D\rho(\gamma)(p)\cdot\xi = e^{-l_\rho(\gamma)}\xi\ .
\end{equation*}
Applying the cocycle condition \eqref{eqn:DefnCocycle} now yields
\begin{equation*}
\begin{aligned}
\xi(p) &= \xi\big(\rho(\gamma)\cdot p\big)\\
&= D\rho(\gamma)(p)\cdot\xi(p) + \xi_{\alpha(\gamma)}\big(\rho(\gamma)\cdot p)\big)\\
&= e^{-l_\rho(\gamma)}\cdot\xi(p) + \xi_{\alpha(\gamma)}(p)\ ,
\end{aligned}
\end{equation*}
and \eqref{eqn:FormulaForAutomorphicVectorField} follows.
\end{proof}

\begin{lemmadefn}
\label{lemma:ExistenceAndUniquenessOfAutomorphicFieldsI}
Let $\rho:\Pi\rightarrow\PSL$ be Fuchsian. For every cocycle $\alpha$, there exists a unique continuous $(\rho,\alpha)$-automorphic vector field $\xi_\alpha$, which we call the \emph{automorphic vector field} of $\alpha$. This vector field is of Zygmund type. In particular, it is of H\"older type $C^{0,\eta}$ for all $\eta<1$.
\end{lemmadefn}
\begin{proof} Since the set of points of $\RP$ which are attractive fixed points of $\rho(\gamma)$ for some non-trivial $\gamma\in\Pi$ is dense in $\RP$, uniqueness follows from \eqref{eqn:FormulaForAutomorphicVectorField}. Existence is proven, for example, in \cite[Section 16.6]{GL}. Likewise, by \cite[Section 16.6]{GL}, $\xi_\alpha$ is of Zygmund type, and this completes the proof.
\end{proof}

\subsubsection{Einstein space}
\label{ss:EinsteinSpace}
We now consider the case where $M$ is $(1,1)$-dimensional Einstein space. Recall from Section \ref{sss:EinsteinSpaceI} that $\Ein$ may be viewed as the Cartesian product $\RP\times\RP$, with structure group $G=\PSL\times\PSL$ and factor group $H=\Stab(\infty)\times\Stab(\infty)$.

We recall that a representation $\rho=(\rho^r,\rho^l):\Pi\rightarrow\PSL\times\PSL$ is \emph{quasi-Fuchsian} whenever each of its two components is Fuchsian. In particular, every $\rho$-cocycle $\alpha$ takes the form $\alpha=(\alpha_r,\alpha_l)$, where $\alpha_r$ (resp. $\alpha_l$) is a $\rho_r$ (resp. $\rho_l$) cocycle.
 
\begin{lemmadefn}
\label{def:limset}
Let $\rho=(\rho^r,\rho^l):\Pi\rightarrow\PSL\times\PSL$ be quasi-Fuchsian. There exists a unique closed, minimal, $\rho$-invariant subset $\Lambda_\rho$ of $\Ein$. We call $\Lambda_\rho$ the \emph{limit set} of $\rho$.
\end{lemmadefn}
\noindent Lemma \ref{def:limset} follows readily from the topological conjugacy of the action of any two Fuchsian representations over $\RP$ (see \cite[Chapter 16]{GL}).
\begin{theorem}
\label{thm:homeoconj}
For all Fuchsian $\rho_1,\rho_2:\Pi\rightarrow\PSL$, there exists a unique orientation-preserving homeomorphism $\phi:\RP\rightarrow\RP$ such that, for all $\gamma\in\Pi$,
\begin{equation*}
\rho_1(\gamma)=\phi\circ\rho_2(\gamma)\circ\phi^{-1}\ .
\end{equation*}
\end{theorem}
\begin{proof}[Sketch of proof of Lemma~\ref{def:limset}.]
Indeed, with $\phi$ is an Theorem \ref{thm:homeoconj}, $\Lambda_\rho$ is the graph of $\phi$. We refer the reader to \cite[Theorem 10.13]{Bar3} for details.
\end{proof}

Every Killing vector field $\xi$ over $\Ein$ takes the form $\xi:=(\xi_r,\xi_l)$ where $\xi_r$ and $\xi_l$ are Killing vector fields over $\RP$. Lemma \ref{lemma:VanishingConditionOfKillingFieldsI} yields the following elementary vanishing result for Killing vector fields over Einstein space.
\begin{lemma}
\label{lemma:VanishingConditionOfKillingFieldsII}
Let $\rho:\Pi\rightarrow\PSL\times\PSL$ be quasi-Fuchsian. Let $p_1,p_2,p_3\in\Lambda_\rho$ be distinct points, and let $\xi$ be a Killing vector field. If $\xi(p_i)=0$ for all $i$, then $\xi=0$.
\end{lemma}
\begin{proof}Denote $\xi:=(\xi_r,\xi_l)$ and, for each $i$, denote $p_i:=(p_{i,r},p_{i,l})$. The set $\Lambda_\rho$ is the graph of a homeomorphism $\phi_\rho:\RP\rightarrow\RP$. In particular, it projects bijectively onto each copy of $\RP$. It follows that $(p_{1,r},p_{2,r},p_{3,r})$ and $(p_{1,l},p_{2,l},p_{3,l})$ are triples of distinct points in $\RP$. For each $i$,
\begin{equation*}
\xi_r(p_{i,r}),\xi_l(p_{i,l})=0\ ,
\end{equation*}
so that, by Lemma \ref{lemma:VanishingConditionOfKillingFieldsI},
\begin{equation*}
\xi_r,\xi_l = 0\ ,
\end{equation*}
as desired.
\end{proof}

\begin{lemmadefn}
\label{lemma:ExistenceAndUniquenessOfAutomorphicFieldsII}
Let $\rho:\Pi\rightarrow\PSL\times\PSL$ be quasi-Fuchsian. For every $\rho$-cocycle $\alpha$, there exists a unique continuous $(\rho,\alpha)$-automorphic vector field $\xi_\alpha$, which we call the \emph{automorphic vector field} of $\alpha$. The restriction of this vector field to $\Lambda_\rho$ is of Zygmund type. In particular, it is of H\"older type $C^{0,\eta}$ for all $\eta<1$.
\end{lemmadefn}
\begin{proof}Since every $(\rho,\alpha)$-automorphic vector field takes the form $\xi:=(\xi_r,\xi_l)$, existence and uniqueness follow by Lemma \ref{lemma:ExistenceAndUniquenessOfAutomorphicFieldsI}. To prove the Zygmund property note that, since $\Lambda_\rho$ is the graph of a homeomorphism $\phi_\rho:\RP\rightarrow\RP$, in particular, it is rectifiable, and each of the two projections $\pi_r,\pi_l:\Lambda_\rho\rightarrow\RP$ restricts to a Lipschitz function over $\Lambda_\rho$. It follows by Lemma \ref{lemma:ExistenceAndUniquenessOfAutomorphicFieldsI} that $\xi=(\xi_r\circ\pi_r,\xi_l\circ\pi_l)$ is of Zygmund type, and this completes the proof.
\end{proof}

A key property of automorphic vector fields over $\Ein$ is that they are generically nowhere vanishing over the limit set $\Lambda_\rho$ of the representation. This is formalized as follows.
\begin{lemma}
\label{lemma:NowhereVanishingProperty}
Let $\rho:\Pi\rightarrow\PSL\times\PSL$ be quasi-Fuchsian, let $\alpha$ be a $\rho$-cocycle, and let $\xi_\alpha$ denote the automorphic vector field of $\alpha$. For every neighbourhood $\Omega$ of $(0,0)$ in $\slalg\times\slalg$, there exists $\mathbf{A}\in\Omega$ such that $\xi_\alpha-\xi_{\mathbf{A}}$ is nowhere vanishing over $\Lambda_\rho$.
\end{lemma}
\begin{remark}
The proof is essentially a standard transversality argument, adapted to the low regularity of the present case.
\end{remark}
\begin{proof}
We view $\Ein$ as the boundary of $\AdS$. Let $\zeta$ be a timelike geodesic in $\AdS$, let $\Stab(\zeta)$ denote its stabilizer in $\PSL\times\PSL$, and let $\stab(\zeta)$ denote the Lie algebra of this subgroup. Since $\Stab(\zeta)$ is isomorphic to $\mathbb{S}^1\times\mathbb{S}^1$, and acts transitively and without fixed points on $\Ein$, the function
\begin{equation*}
\phi:\Ein\times\stab(\zeta)\rightarrow\opT\Ein;(p,\mathbf{A})\mapsto\xi_{\mathbf{A}}(p)
\end{equation*}
is a bundle isomorphism. Let $\pi:\Ein\times\stab(\zeta)\rightarrow\stab(\zeta)$ denote the projection onto the fibre.

By Lemma \ref{lemma:ExistenceAndUniquenessOfAutomorphicFieldsII}, the composition $(\pi\circ\phi^{-1}\circ\xi)$ restricts to a function of Zygmund type over $\Lambda_\rho$. Since $\Lambda_\rho$ is rectifiable, the image $(\pi\circ\phi^{-1}\circ\xi)(\Lambda_\rho)$ has Hausdorff dimension $1$, and thus Lebesgue measure $0$. The set $(\Omega\cap\stab(\zeta))\setminus(\pi\circ\phi^{-1}\circ\xi)(\Lambda_\rho)$ is thus non-empty, and for any element $\mathbf{A}$ of this set, the vector field $\xi-\xi_{\mathbf{A}}$ is nowhere vanishing over $\Lambda_\rho$. The result follows.
\end{proof}

\subsubsection{Anti-de Sitter space}
Finally, we consider the case where $M=\AdS$. Here the structure group is $G=\PSL\times\PSL$ and the factor group is $H=\PSL$. It will actually be convenient in what follows to work instead over the double cover $\SL$, which we view as a homogeneous space with structure group $G=\SL\times\SL$ and factor group $H=\SL$.

Since the action of $G$ on $\SL$ is given by \eqref{eqn:defn:ActionOfSLTimesSL}, Killing vector fields over $\SL$ are precisely those vector fields $\xi$ of the form,
\begin{equation*}
\xi = \mathbf{X}^r - \mathbf{Y}^l\ ,
\end{equation*}
for some $\mathbf{X},\mathbf{Y}\in\slalg$. We thus introduce the trivial bundle
\begin{equation}
\label{eqn:defn:TrivialBundleE}
E:=\SL\times\slalg\times\slalg\ ,
\end{equation}
and we define the bundle morphism $\sigma:E\rightarrow\TSL\oplus\TSL$ by
\begin{equation}
\label{eqn:defn:BundleMorphism}
\sigma\big((\mathbf{X},\mathbf{Y})_{\mathbf{M}}\big) := \big(\mathbf{X}^r(\mathbf{M})-\mathbf{Y}^l(\mathbf{M}),\mathbf{X}^r(\mathbf{M})+\mathbf{Y}^l(\mathbf{M})\big)\ .
\end{equation}
\begin{remark}
In the terminology of \cite{Sch}, the first component of $\sigma$ is the pure translation component of the Killing vector field, and its second component is the pure rotation component. The latter is extracted from the Killing vector field using the curl operator, as we will see presently.
\end{remark}
\begin{lemma}
\label{lemma:SigmaIsABundleIsomorphism}
$\sigma$ is a bundle isomorphism.
\end{lemma}
\begin{proof}By the rank-nullity theorem, it suffices to show that $\sigma$ is everywhere injective. However, if $\sigma\big((\mathbf{X},\mathbf{Y})_{\mathbf{M}}\big)=0$ for some $\mathbf{M}\in\SL$ and $\mathbf{X},\mathbf{Y}\in\slalg$, then
\begin{equation*}
\mathbf{X}^r(\mathbf{M}) = \mathbf{Y}^l(\mathbf{M}) = 0\ ,
\end{equation*}
so that, by invertibility of $\mathbf{M}$,
\begin{equation*}
\mathbf{X} = \mathbf{Y} = 0\ ,
\end{equation*}
as desired.
\end{proof}

Let $D$ denote the standard flat derivative of $E$. We compare this to the Levi--Civita covariant derivative of $TSL\oplus TSL$.
\begin{lemma}
\label{lemma:LCCovariantDerivativeOfAdS}
For any vector field $\xi$ and for all $\mathbf{X}\in\slalg$,
\begin{equation}
\label{eqn:LCCovariantDerivativeOfAdS}
\nabla_\xi\mathbf{X}^l = \xi\times\mathbf{X}^l\qquad\text{and}\qquad
\nabla_\xi\mathbf{X}^r = -\xi\times\mathbf{X}^r\ .
\end{equation}
\end{lemma}
\begin{proof}Indeed (see, for example, \cite{milnor1976curvatures}), the Levi--Civita covariant derivatives of left and right invariant vector fields satisfy
\begin{equation*}
\nabla_{\mathbf{X}^l}\mathbf{Y}^l = \frac{1}{2}[\mathbf{X},\mathbf{Y}]^l\qquad\text{and}\qquad
\nabla_{\mathbf{X}^r}\mathbf{Y}^r = -\frac{1}{2}[\mathbf{X},\mathbf{Y}]^r\ ,
\end{equation*}
so that, by \eqref{eqn:CommutatorIsCrossProduct},
\begin{equation*}
\nabla_{\mathbf{X}^l}\mathbf{Y}^l = \mathbf{X}^l\times\mathbf{Y}^l\qquad\text{and}\qquad
\nabla_{\mathbf{X}^r}\mathbf{Y}^r = -\mathbf{X}^r\times\mathbf{Y}^r\ .
\end{equation*}
The result follows by tensoriality of the covariant derivative with respect to the first component.
\end{proof}

Define the section
\begin{equation*}
\Omega\in\Gamma\big(\opT^*\SL\otimes\mathrm{End}(\opT\SL\oplus\opT\SL)\big)
\end{equation*}
by
\begin{equation}
\label{eqn:defn:ConnectionFormOfE}
\Omega_\xi(\mathbf{X},\mathbf{Y}) := \nabla_\xi(\mathbf{X},\mathbf{Y}) - (\sigma_*D)_\xi(\mathbf{X},\mathbf{Y})\ .
\end{equation}
\begin{lemma}
\label{lemma:FormulaForConnectionForm}
For every vector field $\xi$, and for every section $(\mathbf{X},\mathbf{Y})$ of $E$,
\begin{equation}
\label{eqn:FormulaForConnectionFormI}
(\sigma^*\Omega)_\xi(\mathbf{X},\mathbf{Y}) = \big(-\tau^r(\xi)\times\mathbf{X},\tau^l(\xi)\times\mathbf{Y}\big)\ .
\end{equation}
Equivalently, for all vector fields $\xi$, $\eta_1$, and $\eta_2$,
\begin{equation}
\label{eqn:FormulaForConnectionFormII}
\Omega_\xi(\eta_1,\eta_2) = (-\xi\times\eta_2,-\xi\times\eta_1)\ .
\end{equation}
\end{lemma}
\begin{proof}Indeed, choose $\mathbf{X},\mathbf{Y}\in\slalg$, and identify the pair $(\mathbf{X},\mathbf{Y})$ with a constant section of $E$. In particular, for any $\xi$,
\begin{equation*}
D_\xi(\mathbf{X},\mathbf{Y}) = 0\ ,
\end{equation*}
so that
\begin{equation*}
(\sigma^*\Omega)_\xi(\mathbf{X},\mathbf{Y}) = (\sigma^*\nabla)_\xi(\mathbf{X},\mathbf{Y})\ .
\end{equation*}
By \eqref{eqn:LCCovariantDerivativeOfAdS}, for all $\xi$,
\begin{equation*}
\begin{aligned}
\nabla_\xi\sigma_*(\mathbf{X},\mathbf{Y})
&=\nabla_\xi\big(\mathbf{X}^r-\mathbf{Y}^l,\mathbf{X}^r+\mathbf{Y}^l\big)\\
&=\big(-\xi\times\mathbf{X}^r-\xi\times\mathbf{Y}^l,-\xi\times\mathbf{X}^r+\xi\times\mathbf{Y}^r\big)
\end{aligned}
\end{equation*}
Note now that, since the vector product is invariant under isometries of $\SL$, for all $\mathbf{U},\mathbf{V}\in\slalg$,
\begin{equation*}
(\mathbf{U}\times\mathbf{V})^r = \mathbf{U}^r\times\mathbf{V}^r\qquad\text{and}\qquad
(\mathbf{U}\times\mathbf{V})^l = \mathbf{U}^l\times\mathbf{V}^l\ .
\end{equation*}
Hence,
\begin{equation*}
\begin{aligned}
\nabla_\xi\sigma_*(\mathbf{X},\mathbf{Y})
&=\big(-(\tau^r(\xi)\times\mathbf{X})^r-(\tau^l(\xi)\times\mathbf{Y})^l,-(\tau^r(\xi)\times\mathbf{X})^r+(\tau^l(\xi)\times\mathbf{Y})^l\big)\\
&=\sigma_*\big(-\tau^r(\xi)\times\mathbf{X},\tau^l(\xi)\times\mathbf{Y}\big)\ ,
\end{aligned}
\end{equation*}
so that
\begin{equation*}
(\sigma^*\Omega)_\xi(\mathbf{X},\mathbf{Y}) = \big(-\tau^r(\xi)\times\mathbf{X},\tau^l(\xi)\times\mathbf{Y}\big)\ .
\end{equation*}
The first assertion follows, and the second assertion follows upon setting $\eta_1:=\mathbf{X}^r-\mathbf{Y}^l$ and $\eta_2:=\mathbf{X}^r+\mathbf{Y}^l$.
\end{proof}

We now introduce the curl operator which may be used to characterize Killing vector fields over $\AdS$. We first recall that, for any semi-riemannian $3$-manifold $M$ with metric $g$ and Levi--Civita covariant derivative $\nabla$, its \emph{curl operator} is defined by
\begin{equation}
\label{eqn:defn:CurlOperatorI}
\nabla\times\xi := g^{ij}\partial_i\times\nabla_{\partial_j}\xi\ ,
\end{equation}
where $(\partial_i)_{1\leq i\leq 3}$ is any local frame, and Einstein's summation convention is assumed. In particular, for any orthonormal frame $(e_i)_{1\leq i\leq 3}$,
\begin{equation}
\label{eqn:defn:CurlOperatorII}
\nabla\times\xi := \sum_i g(e_i,e_i)e_i\times\nabla_{e_i}\xi\ ,
\end{equation}
so that, in contrast to the Riemannian case, the sign of $e_i$ contributes non-trivially to the curl operator.
\begin{lemma}
\label{lemma:CharacterizationOfKillingFields}
A vector field $\xi$ is a Killing vector field if and only if $\sigma^*(\xi,-(1/2)\nabla\times\xi)$ is a flat section of $E$.
\end{lemma}
\begin{proof} Since the spaces of Killing vector fields and of flat sections of $E$ are both $6$-dimen\-sional, it suffices to show that, for all $\mathbf{X},\mathbf{Y}\in\slalg$,
\begin{equation*}
-\frac{1}{2}\nabla\times(\mathbf{X}^r-\mathbf{Y}^l) = \mathbf{X}^r + \mathbf{Y}^l\ .
\end{equation*}
To this end, let $(e_i)_{1\leq i\leq 3}$ be an orthonormal frame of $\SL$. Note that, for each $i$, and for every tangent vector $\xi$,
\begin{equation*}
e_i\times(e_i\times\xi) = \langle e_i,e_i\rangle\pi_i(\xi)\ ,
\end{equation*}
where $\pi_i$ denotes the orthogonal projection onto $\langle e_i\rangle^\perp$. Thus, by \eqref{eqn:LCCovariantDerivativeOfAdS}, for all $\mathbf{X}\in\slalg$,
\begin{equation*}
\begin{aligned}
\nabla\times\mathbf{X}^r
&=\sum_i\langle e_i,e_i\rangle e_i\times\big(\nabla_{e_i}\mathbf{X}^r\big)\\
&=\sum_i-\langle e_i,e_i\rangle e_i\times\big(e_i\times\mathbf{X}^r\big)\\
&=\sum_i-\langle e_i,e_i\rangle^2\pi_i\big(\mathbf{X}^r\big)\\
&=-2\mathbf{X}^r\ .
\end{aligned}
\end{equation*}
Likewise, for all $\mathbf{Y}\in\slalg$,
\begin{equation*}
\nabla\times\mathbf{Y}^l = 2\mathbf{Y}^l\ .
\end{equation*}
Thus,
\begin{equation*}
\nabla\times\big(\mathbf{X}^r-\mathbf{Y}^l) = -2\big(\mathbf{X}^r + \mathbf{Y}^l\big)\ ,
\end{equation*}
and the result follows.
\end{proof}

As in Section \ref{ss:EinsteinSpace}, a representation $\rho:=(\rho^r,\rho^l):\Pi\rightarrow\SL\times\SL$ is quasi-Fuchsian whenever each of its two components is Fuchsian. Likewise, every $\rho$-cocycle $\alpha$ has the form $\alpha=(\alpha^r,\alpha^l)$, where $\alpha^r$ (resp. $\alpha^l$) is a $\rho^r$ (resp. $\rho^l$) cocycle. Given such a representation and cocycle, we see that a $\rho$-\emph{automorphic vector field} over a $\rho$-invariant subset $M\subset \SL$ is a vector field $\xi$ such that, for all $\mathbf{M}\in M$, and for all $\gamma\in\Pi$,
\begin{equation}
\label{eqn:defn:AutomorphicVectorField}
\xi\big(\rho(\gamma)\cdot\mathbf{M}\big) = D\rho(\gamma)(\mathbf{M})\cdot\xi(\mathbf{M}) + \alpha^r(\gamma)^r\big(\rho(\gamma)\cdot\mathbf{M}\big) - \alpha^l(\gamma)^l\big(\rho(\gamma)\cdot\mathbf{M}\big)\ .
\end{equation}

Given an automorphic vector field $\xi$, we denote
\begin{equation*}
\eta := -\frac{1}{2}\nabla\times\xi\ ,
\end{equation*}
and we observe that $\sigma^*(\xi,\eta)$ may be viewed as an automorphic $\slalg\times\slalg$-valued function. We define the $\slalg\times\slalg$-valued $1$-form $\omega^\xi$ by
\begin{equation}
\label{eqn:defn:LieAlgebraValuedOneForm}
\omega^\xi := d\big(\sigma^*(\xi,\eta)\big)\ .
\end{equation}
\begin{lemma}
For any automorphic vector field $\xi$, the $1$-form $\omega^\xi$ is $\rho$-equivariant.
\end{lemma}
\begin{proof} Indeed, by Lemma \ref{lemma:CharacterizationOfKillingFields}, for any Killing field $\tilde{\xi}$, the $\slalg\times\slalg$-valued function $\sigma^*(\tilde{\xi},-(1/2)\nabla\times\tilde{\xi})$ is constant, and the result now follows by \eqref{eqn:defn:AutomorphicVectorField}.
\end{proof} 
\section{Timelike geodesics in anti-de Sitter space}
\label{s:TimelikeGeodesics}
A number of estimates of interest to us will be obtained using the geometry of timelike geodesics in anti-de Sitter space, which we now study. Throughout this section, we will work with the matrix model described in Section \ref{ss:TheMatrixModel}.

\subsection{Elementary properties}
Let $\zeta\in\geodesics$ be a future oriented timelike geodesic in $\SL$. Recall that $\zeta$ is the intersection of $\SL$ with a plane $P:=P_\zeta$ of signature $(0,2)$, and that $\zeta$ identifies with the unit circle in this plane. Let $\pi_\zeta:\Mat\rightarrow P$ and $\pi^\perp_\zeta:\Mat\rightarrow P^\perp$ denote the orthogonal projections, and note that $\pi_\zeta$ \emph{increases} the length of any timelike vector.

Let $\tilde{\zeta}\in\geodesics$ be another future oriented timelike geodesic with plane $\tilde{P}:=P_{\tilde{\zeta}}$. Since $\pi_\zeta$ increases length, $\pi_\zeta(\tilde{\zeta})$ is an ellipse in $P$ with minor and major axes both of radius at least $1$. The radius of the minor axis is equal to $1$ if and only if $\tilde{\zeta}$ meets $\zeta$ at some point. When this holds, and when $\zeta$ and $\tilde{\zeta}$ do not coincide, the intersection consists of exactly two points in $\SL$, corresponding to exactly one point in the projectivization. This is illustrated in Figure \ref{fig:TwoLightlikeGeodesics}.

\begin{figure}[h!]
\includegraphics[width=8cm]{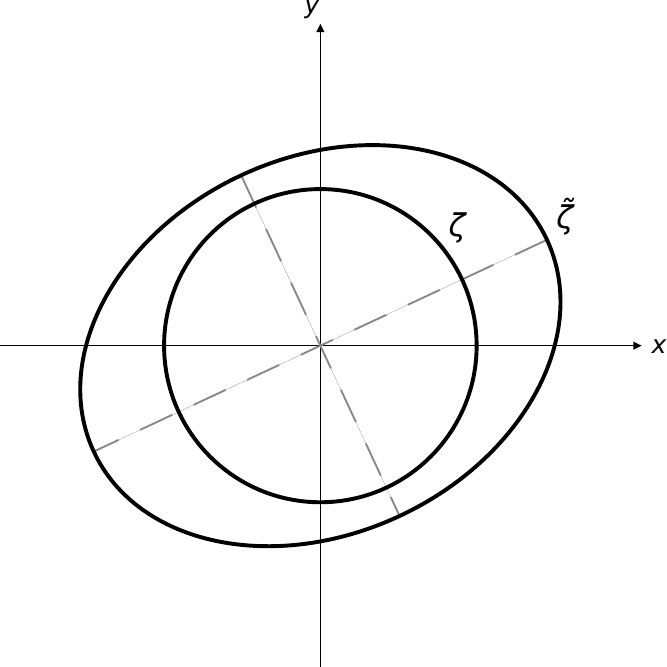}
\caption{The timelike geodesic $\zeta$ is the unit circle in the plane $P$. The orthogonal projection onto this plane maps the second timelike geodesic $\tilde{\zeta}$ to an ellipse, also centred on the origin, with minor and major axes of radius at least $1$.}
\label{fig:TwoLightlikeGeodesics}
\end{figure}

\begin{lemma}
\label{lemma:GeometryOfTimelikeGeodesics}
Let $\zeta$ and $\tilde{\zeta}$ be distinct timelike geodesics in $\SL$.
\begin{itemize}
\item[(1)] Two points $\mathbf{M}\in\zeta$ and $\tilde{\mathbf{M}}\in\tilde{\zeta}$ are in spacelike relation if and only if
\begin{equation}
\left|\langle\mathbf{M},\pi_\zeta(\tilde{\mathbf{M}})\rangle_\opmat\right| > 1\ .
\end{equation}
Furthermore, when this holds, the spacelike distance between these points satisfies
\begin{equation}
\label{eqn:SpacelikeDistanceI}
\cosh\big(d(\mathbf{M},\tilde{\mathbf{M}})\big) = \left|\langle\mathbf{M},\pi_\zeta(\tilde{\mathbf{M}})\rangle_\opmat\right|\ .
\end{equation}
\item[(2)] For all $\tilde{\mathbf{M}}\in\tilde{\zeta}$, the radial projection of $\pi_\zeta(\tilde{\mathbf{M}})$ is, up to sign, the unique point of $\zeta$ maximizing spacelike distance to $\tilde{\mathbf{M}}$. Furthermore, this distance satisfies
\begin{equation}
\label{eqn:SpacelikeDistanceII}
\cosh\big(d(\zeta,\tilde{\mathbf{M}})\big) = \|\pi_\zeta(\tilde{\mathbf{M}})\|_\opmat\ .
\end{equation}
\item[(3)] If the spacelike distance $d_\zeta$ to $\zeta$ is non-constant over $\tilde{\zeta}$, then there exists, up to sign, a unique point $\tilde{\mathbf{M}}$ of $\tilde{\zeta}$ maximizing spacelike distance to $\zeta$. Furthermore, $\pi_\zeta(\tilde{\mathbf{M}})$ is an end-point of the major axis of $\pi_\zeta(\tilde{\zeta})$.
\end{itemize}
\end{lemma}
\begin{proof}To prove $(1)$, note that $\mathbf{M}$ and $\tilde{\mathbf{M}}$ are in spacelike relation if and only if the plane they span has signature $(1,1)$. This in turn holds if and only if
\begin{equation*}
\left|\begin{matrix}\langle\mathbf{M},\mathbf{M}\rangle_\opmat&\langle\mathbf{M},\tilde{\mathbf{M}}\rangle_\opmat\\
\langle\mathbf{M},\tilde{\mathbf{M}}\rangle_\opmat&\langle\tilde{\mathbf{M}},\tilde{\mathbf{M}}\rangle_\opmat\end{matrix}\right| < 0\ .
\end{equation*}
Calculating this determinant, and bearing in mind that
\begin{equation*}
\|\mathbf{M}\|_\opmat^2 = \|\tilde{\mathbf{M}}\|_\opmat^2 = -1\ ,
\end{equation*}
we find that this holds if and only if
\begin{equation*}
\left|\langle\mathbf{M},\pi_\zeta(\tilde{\mathbf{M}})\rangle_\opmat\right| = \left|\langle\mathbf{M},\tilde{\mathbf{M}}\rangle_\opmat\right| > 1\ ,
\end{equation*}
as asserted. In addition, when this holds, the spacelike geodesic in $\AdS$ from $\mathbf{M}$ to $\tilde{\mathbf{M}}$ lies in the span of these two vectors, and \eqref{eqn:SpacelikeDistanceI} follows by classical hyperbolic geometry. This proves $(1)$.

By \eqref{eqn:SpacelikeDistanceI}, the spacelike distance to any fixed $\tilde{\mathbf{M}}$ is maximized whenever the absolute value of the inner product
\begin{equation*}
\langle\mathbf{M},\tilde{\mathbf{M}}\rangle_\opmat = \langle\mathbf{M},\pi_\zeta(\tilde{\mathbf{M}})\rangle_\opmat
\end{equation*}
is maximized. This proves $(2)$. The final assertion follows upon maximizing \eqref{eqn:SpacelikeDistanceII}, and this completes the proof.
\end{proof}
Projectivization this lemma yields the following result.
\begin{corollary}
Let $\zeta,\tilde{\zeta}\in\geodesics$ be distinct oriented timelike geodesics in $\AdS$ and let $d_\zeta$ denote the spacelike distance to $\zeta$. If $d_\zeta$ is non-constant over $\tilde{\zeta}$, then there exists a unique spacelike geodesic $\eta$ joining their furthest points. Furthermore, $\eta$ meets both $\zeta$ and $\tilde{\zeta}$ orthogonally at these points.
\end{corollary}

\subsection{The ``H'' Lemma}
\label{ss:TimelikeGeodesicsInTheMatrixModel}
A more quantitative picture is obtained using the full translation map $\tau$, introduced in Section \ref{sss:AntideSitterSpaceInMatrixModel}. This will yield two technical results concerning the geometry of timelike geodesics in $\AdS$ which we believe to be of independent interest.

Upon applying a suitable isometry, we may suppose that
\begin{equation*}
\tau^r(\zeta) = \tau^l(\zeta) = \sigma_1\ .
\end{equation*}
We verify by inspection that, up to reparametrization
\begin{equation}
\label{eqn:ParametrizationOfFirstGeodesic}
\zeta(s) = \cos(s)\sigma_0 + \sin(s)\sigma_1\ .
\end{equation}
Upon applying suitable elliptic isometries, we may also suppose that
\begin{equation}
\label{eqn:NormalizedRightAndLeftProjectionsOfZeta}
\begin{aligned}
\tau^r(\tilde{\zeta}) &= \cosh(2\lambda)\sigma_1 + \sinh(2\lambda)\sigma_2\ \text{and}\\
\tau^l(\tilde{\zeta}) &= \cosh(2\mu)\sigma_1 + \sinh(2\mu)\sigma_2\ ,
\end{aligned}
\end{equation}
where
\begin{equation}
\label{eqn:DefnsOfLambdaAndMu}
\lambda := \frac{1}{2}d\big(\tau^r(\tilde{\zeta}),\tau^r(\zeta)\big)\qquad\text{and}\qquad
\mu := \frac{1}{2}d\big(\tau^l(\tilde{\zeta}),\tau^l(\zeta)\big)\ .
\end{equation}

\begin{lemma}
\label{lemma:ParametrizationOfSecondGeodesic}
With the above normalizations, up to reparametrization,
\begin{equation}
\label{eqn:ParametrizationOfSecondGeodesic}
\begin{aligned}
\tilde{\zeta}(s) &= \cosh(\lambda-\mu)\cos(s)\sigma_0 + \cosh(\lambda+\mu)\sin(s)\sigma_1\\
&\qquad + \sinh(\lambda+\mu)\sin(s)\sigma_2 + \sinh(\lambda-\mu)\cos(s)\sigma_3\ .
\end{aligned}
\end{equation}
In particular
\begin{equation}
\label{eqn:ParametrizationOfSecondGeodesicII}
\pi_\zeta\big(\tilde{\zeta}(s)\big) = \cosh(\lambda-\mu)\cos(s)\sigma_0 + \cosh(\lambda+\mu)\sin(s)\sigma_1\ .
\end{equation}
Furthermore, spacelike distance from $\zeta$ is maximized when $s=\pi/2$, where it is equal to $(\lambda+\mu)$.
\end{lemma}
\begin{remark} In particular, the spacelike distance between $\zeta$ and $\tilde{\zeta}$ is constant whenever one of $\lambda$ or $\mu$ is equal to $0$, that is, whenever either $\tau^r(\tilde{\zeta})=\tau^r(\zeta)$ or $\tau^l(\tilde{\zeta})=\tau^l(\zeta)$.
\end{remark}
\begin{proof} Indeed, viewed as a set, $\tilde{\zeta}$ consists of all $\mathbf{M}$ such that
\begin{equation*}
\mathbf{M}\cdot\tau^l(\tilde{\zeta}) = \tau^r(\tilde{\zeta})\cdot\mathbf{M}\ .
\end{equation*}
Define $\mathbf{M}_0$ and $\mathbf{N}_0$ by
\begin{equation*}
\begin{aligned}
\mathbf{M}_0 &:= \cosh(\lambda)\sigma_0 + \sinh(\lambda)\sigma_3\ \text{and}\\
\mathbf{N}_0 &:= \cosh(\mu)\sigma_0 + \sinh(\mu)\sigma_3\ .
\end{aligned}
\end{equation*}
By \eqref{eqn:BasicPauliRelationsI} and \eqref{eqn:BasicPauliRelationsII},
\begin{equation*}
\mathbf{M}_0\cdot\sigma_1\cdot\mathbf{M}_0^{-1} = \tau^r(\tilde{\zeta})\qquad\text{and}\qquad\mathbf{N}_0\cdot\sigma_1\cdot\mathbf{N}_0^{-1} = \tau^l(\tilde{\zeta})\ .
\end{equation*}
Hence, up to reparametrization
\begin{equation*}
\tilde{\zeta}(s) = \mathbf{M}_0\cdot\zeta(s)\cdot\mathbf{N}_0^{-1}\ ,
\end{equation*}
where $\zeta$ is as in \eqref{eqn:ParametrizationOfFirstGeodesic}, and \eqref{eqn:ParametrizationOfSecondGeodesic} now follows by explicit calculation. Since the plane of $\zeta$ is $P:=\langle\sigma_0,\sigma_1\rangle$, \eqref{eqn:ParametrizationOfSecondGeodesicII} follows immediately. Finally, the minor and major axes of $\pi_\zeta(\tilde{\zeta})$ are respectively $\sigma_0$ and $\sigma_1$, and it follows by Lemma \ref{lemma:GeometryOfTimelikeGeodesics}, Item $(3)$ that spacelike distance to $\zeta$ is maximized over $\tilde{\zeta}$ at $s=\pi/2$, where it is equal to $(\lambda+\mu)$. This completes the proof.
\end{proof}

Our first technical result concerns the spacelike distance betwen points of distinct timelike geodesics. A characteristic of negative curvature in mixed signature is that timelike geodesics may traverse arbitrarily large expanses of space in uniformly bounded time. Intuitively, a particle moving sufficiently close to the speed of light will reach us within a certain fixed time, no matter how distant it initially was. Nevertheless, negative curvature also ensures that timelike geodesics which start distant from one another remain distant from one another \emph{for most of their existence}. This key property, illustrated in Figure \ref{fig:DistanceBetweenTimelikeGeodesics}, is quantified as follows.

\begin{figure}[h!]
\includegraphics[width=8cm]{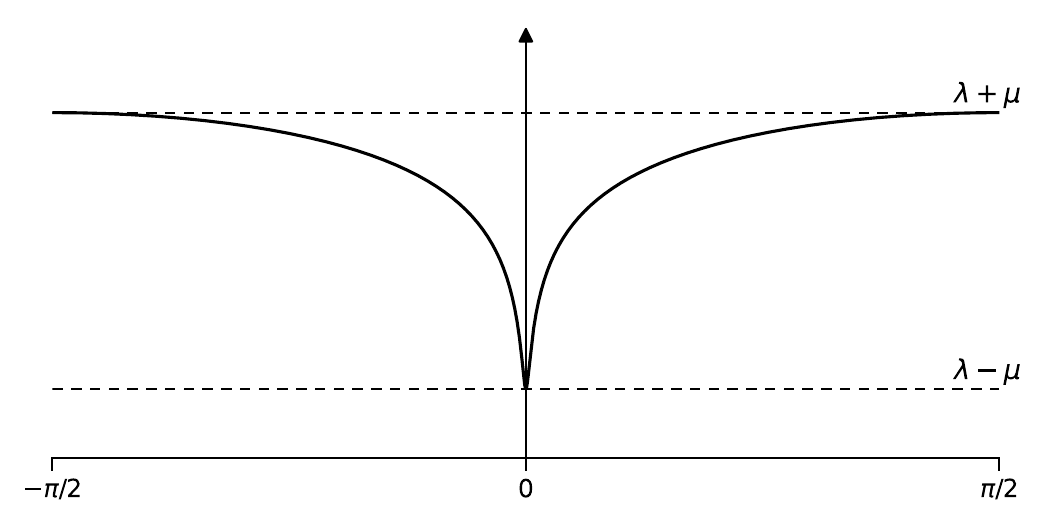}
\caption{Let $\zeta$ and $\tilde{\zeta}$ be timelike geodesics, and consider the restriction to $\tilde{\zeta}$ of the spacelike distance $d_\zeta$ to $\zeta$ in $\AdS$. Owing to the properties of the logarithm, whose geometric origin here is negativity of curvature, the function $d_\zeta$ remains close to its maximal value over most of its domain.}
\label{fig:DistanceBetweenTimelikeGeodesics}
\end{figure}

\begin{lemma}
\label{lemma:DistanceBetweenPointsOfTimelikeGeodesics}
For all $\epsilon>0$, there exists $C_1,C_2>0$ such that, if $\lambda,\mu>C_1$, and if $s,\tilde{s}\in[\epsilon,\pi-\epsilon]$, then $\zeta_0(s)$ and $\tilde{\zeta}(\tilde{s})$ are in spacelike relation, and
\begin{equation}
\label{eqn:DistanceBetweenPointsOfTimelikeGeodesics}
d\big(\zeta(s),\tilde{\zeta}(\tilde{s})\big) \geq (\lambda+\mu) - C_2\ .
\end{equation}
\end{lemma}
\begin{proof}Indeed, by \eqref{eqn:ParametrizationOfSecondGeodesicII}, for all $s,\tilde{s}\in[0,\pi]$,
\begin{equation*}
\left|\langle\zeta(s),\tilde{\zeta}(\tilde{s})\rangle_\opmat\right| = \cosh(\lambda-\mu)\cos(s)\cos(\tilde{s}) + \cosh(\lambda+\mu)\sin(s)\sin(\tilde{s})\ .
\end{equation*}
We minimize this quantity over $[\epsilon,\pi-\epsilon]^2$ by choosing $s$ and $\tilde{s}$ such that the coefficient of $\cosh(\lambda+\mu)$ is as small as possible and the coefficient of $\cosh(\lambda-\mu)$ is negative and as large as possible. This happens when $s=\epsilon$ and $\tilde{s}=\pi-\epsilon$. Hence, for all $s,s'\in[\epsilon,\pi-\epsilon]$,
\begin{equation*}
\left|\langle\zeta(s),\tilde{\zeta}(\tilde{s})\rangle_\opmat\right| \geq \cosh(\lambda+\mu)\sin^2(\epsilon) - \cosh(\lambda-\mu)\cos^2(\epsilon)\ .
\end{equation*}
For $C_1$ sufficiently large, this quantity is greater than $1$ for all such $s$ and $\tilde{s}$. Hence, by Lemma \ref{lemma:GeometryOfTimelikeGeodesics}, Item $(2)$, for $C_1$ sufficiently large, and for all such $s$ and $\tilde{s}$, $\zeta(s)$ and $\tilde{\zeta}(\tilde{s})$ are in spacelike relation, and
\begin{equation*}
e^{d(\zeta(s),\tilde{\zeta}(\tilde{s}))} \geq\cosh\big(d(\zeta(s),\tilde{\zeta}(\tilde{s}))\big) \geq \frac{1}{2}\cosh(\lambda+\mu)\sin^2(\epsilon) \geq \frac{1}{4}e^{\lambda+\mu}\sin^2(\epsilon)\ .
\end{equation*}
The result follows upon taking logarithms.
\end{proof}

Our second technical result is an explicit formula for the geodesic joining the furthest points of $\zeta$ and $\tilde{\zeta}$ in terms of $\tau(\zeta)$ and $\tau(\tilde{\zeta})$. This is provided by the following ``H'' Lemma, illustrated by Figure \ref{fig:HLemma}. We recall that the ideal boundary of $\AdS$ identifies with $\Ein$, which then identifies with $\RP\times\RP$. Furthermore, each of these copies of $\RP$ may be viewed as the boundary of $\hyp^2$.

\begin{figure}[h!]
\includegraphics[width=7cm]{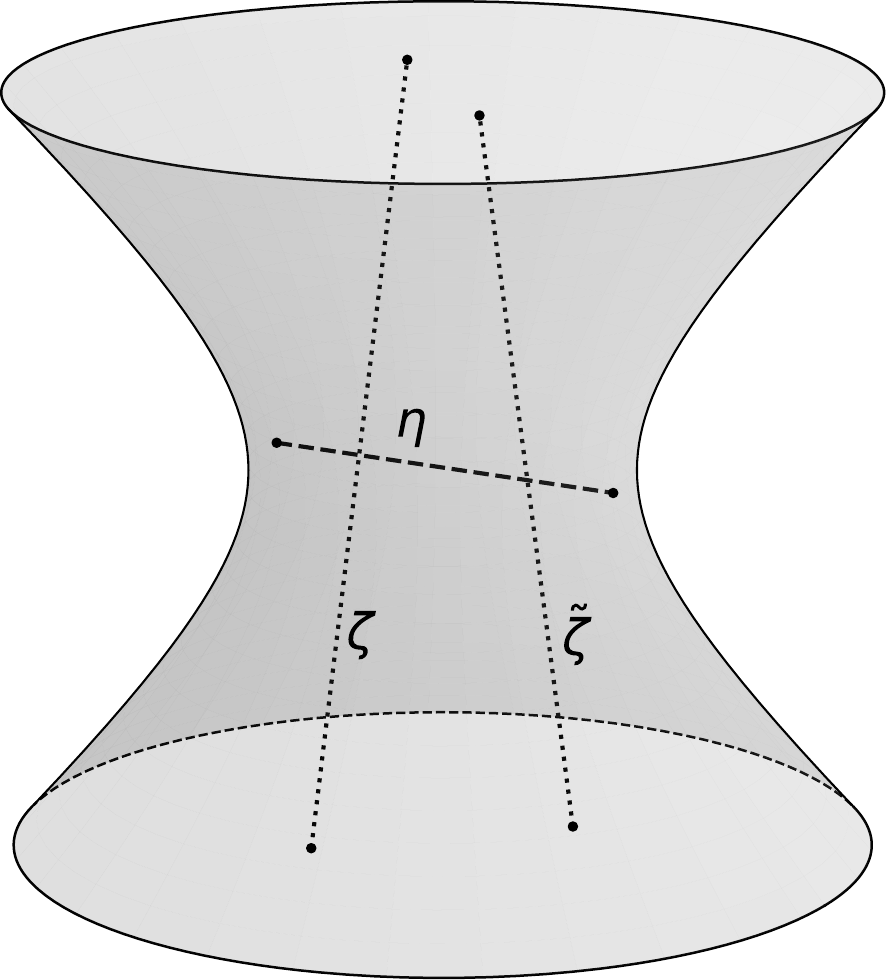}
\caption{The spacelike geodesic $\eta$ joins the furthest points of the timelike geodesics $\zeta$ and $\tilde{\zeta}$, thus forming a figure ``H''. The ``H'' Lemma provides a formula for $\eta$ in terms of the images under the full translation map of $\zeta$ and $\tilde{\zeta}$.}
\label{fig:HLemma}
\end{figure}

\begin{lemma}[``H'' Lemma]
\label{lemma:HLemma}
Let $\zeta$ and $\tilde{\zeta}$ be future oriented timelike geodesics in $\AdS$ such that the spacelike distance between them is non-constant. Let $\eta$ denote the unique oriented spacelike geodesic in $\AdS$ joining their furthest points, and let $\eta^r$ (resp. $\eta^l$) denote the oriented geodesic in $\hyp^2$ from $\tau^r(\zeta)$ to $\tau^r(\tilde{\zeta})$ (resp. from $\tau^l(\zeta)$ to $\tau^l(\tilde{\zeta})$). The end-points of $\eta$ are related to those of $\eta^r$ and $\eta^l$ by
\begin{equation}
\label{eqn:HLemma}
\tau_\infty(\eta^\pm) = \big(\eta^{r,\pm},\eta^{l,\pm}\big)\ .
\end{equation}
\end{lemma}
\begin{proof} Indeed, by Lemma \ref{lemma:ParametrizationOfSecondGeodesic}, with the above normalizations,
\begin{equation*}
\eta(t) = \eta^r(t) = \eta^l(t) = \cosh(t)\sigma_1 + \sinh(t)\sigma_2 = \mathbf{M}(t)\cdot\sigma_1\cdot\mathbf{M}(t)^{-1}\ ,
\end{equation*}
where, for all $t$,
\begin{equation*}
\mathbf{M}(t) := \cosh(t/2)\sigma_0 + \sinh(t/2)\sigma_3\ .
\end{equation*}
Denote $\mathbf{M}_1:=\mathbf{M}(1)$. Let $\hat{\alpha}\big(\mathbf{M}_1,\mathbf{M}_1\big)^\pm$ denote the attractive and repulsive fixed points of $\hat{\alpha}\big(\mathbf{M}_1,\mathbf{M}_1\big)$ in $\Ein$, and let $\alpha\big(\mathbf{M}_1\big)^\pm$ denote the attractive and repulsive fixed points of $\alpha\big(\mathbf{M}_1\big)$ in $\ideal\mathbf{H}^2$. Bearing in mind \eqref{eqn:defn:EquivarianceOfLimitI},
\begin{equation*}
\tau_\infty(\eta^{\pm}) = \tau_\infty\big(\hat{\alpha}\big(\mathbf{M}_1,\mathbf{M}_1\big)^\pm\big) = \big(\alpha\big(\mathbf{M}_1\big)^\pm,\alpha\big(\mathbf{M}_1\big)^\pm\big) = \big(\eta^{r,\pm},\eta^{l,\pm}\big)\ ,
\end{equation*}
as desired.
\end{proof}

\subsection{Cylindrical polar vector fields}
\label{ss:VectorFieldsAroundTimelikeGeodesics}
We conclude this section by studying vector fields over $\SL$ naturally defined by the timelike geodesic $\zeta$ and considering their restrictions to the second timelike geodesic $\tilde{\zeta}$. The estimates obtained here will allow us to later describe the tangent vector field of $\tilde{\zeta}$ in terms of the basis vector fields of cylindrical polar coordinates about $\zeta$. We continue to use the notation of the preceding section.

Consider first the flow $\Phi_\zeta$ defined over $\SL$ by
\begin{equation*}
\label{eqn:defn:TranslationFlow}
\Phi_{\zeta,s}(\mathbf{M}) := \zeta(s/2)\cdot\mathbf{M}\cdot\zeta(s/2)\ .
\end{equation*}
This is the flow of translations of $\SL$ along the axis $\zeta$, and its vector field is
\begin{equation}
\label{eqn:TranslationVectorField}
X_\zeta(\mathbf{M}) := \partial_s\Phi_{\zeta,s}(\mathbf{M})|_{s=0} = \frac{1}{2}\big(\sigma_1\cdot\mathbf{M}+\mathbf{M}\cdot\sigma_1\big)\ .
\end{equation}
This vector field is timelike, well-defined, and non-zero over the whole of $\SL$. Since $\sigma_1$, $\sigma_2$, and $\sigma_3$ anticommute, \eqref{eqn:ParametrizationOfSecondGeodesic} and \eqref{eqn:TranslationVectorField} yield, for all $s\in[0,\pi]$,
\begin{equation}
\label{eqn:RestrictionOfTranslationToGeodesic}
X_\zeta\big(\tilde{\zeta}(s)\big) = \cosh(\lambda-\mu)\cos(s)\sigma_1 - \cosh(\lambda+\mu)\sin(s)\sigma_0\ .
\end{equation}
\begin{lemma}
\label{lemma:EstimateOfTranslationField}
For all $\epsilon>0$, there exists $C>0$ such that, for all $s\in[\epsilon,\pi-\epsilon]$,
\begin{equation}
\label{eqn:EstimateOfTranslationField}
-\|X_\zeta(\tilde{\zeta}(s))\|^2_\opmat \geq C_2\cosh^2(\lambda+\mu)\ .
\end{equation}
\end{lemma}
\begin{proof}Indeed, since $\sigma_0$ and $\sigma_1$ are orthonormal and timelike, by \eqref{eqn:RestrictionOfTranslationToGeodesic},
\begin{equation*}
-\|X_\zeta(\tilde{\zeta}(s))\|_\opmat^2 = \cosh^2(\lambda-\mu)\cos^2(s) + \cosh^2(\lambda+\mu)\sin^2(s)\geq \cosh^2(\lambda+\mu)\sin^2(\epsilon)\ .
\end{equation*}
The result follows upon setting $C:=\sin^2(\epsilon)$.
\end{proof}

\begin{lemma}
\label{lemma:TangentComponentOfTranslationField}
For all $s\in[0,\pi]$,
\begin{equation}
\label{eqn:TangentComponentOfTranslationField}
\langle X_\zeta(\tilde{\zeta}(s)),\partial_s\tilde{\zeta}(s)\rangle_\opmat = -\cosh(\lambda+\mu)\cosh(\lambda-\mu)\ .
\end{equation}
\end{lemma}
\begin{proof}Indeed, differentiating \eqref{eqn:ParametrizationOfSecondGeodesicII} yields, for all $s$,
\begin{equation*}
\partial_s\pi_\zeta\big(\tilde\zeta(s)\big) = -\cosh(\lambda-\mu)\sin(s)\sigma_0 + \cosh(\lambda+\mu)\cos(s)\sigma_1\ .
\end{equation*}
The result follows upon taking the inner product with \eqref{eqn:RestrictionOfTranslationToGeodesic}.
\end{proof}

For completeness, we consider also the flow $\Psi_\zeta$ defined over $\SL$ by
\begin{equation*}
\label{eqn:defn:RotationFlow}
\Psi_\zeta(\mathbf{M}) := \zeta(s/2)\cdot\mathbf{M}\cdot\zeta(-s/2)\ .
\end{equation*}
This is the flow of rotations of $\SL$ about the axis $\zeta$, and its vector field is
\begin{equation}
\label{eqn:RotationVectorField}
Y_\zeta(\mathbf{M}) := \partial_s\Psi_{\zeta,s}(\mathbf{M})|_{s=0} = \frac{1}{2}\big(\sigma_1\cdot\mathbf{M} - \mathbf{M}\cdot\sigma_1\big)\ .
\end{equation}
This vector field is well-defined and spacelike over the whole of $\SL$. Since $\sigma_1$, $\sigma_2$, and $\sigma_3$ anticommute, \eqref{eqn:ParametrizationOfSecondGeodesic} yields
\begin{equation}
\label{eqn:RotationVectorFieldOverGeodesic}
Y_\zeta\big(\tilde{\zeta}(s)\big) = -\sinh(\lambda-\mu)\cos(s)\sigma_2 + \sinh(\lambda+\mu)\sin(s)\sigma_3\ .
\end{equation}
\begin{lemma}
For all $\epsilon>0$, there exists $C>0$ such that, for all $s\in[\epsilon,\pi-\epsilon]$,
\begin{equation}
\label{eqn:NormOfRestrictionOfRotationVectorField}
\|Y_\zeta(\tilde{\zeta}(s))\|^2_\opmat \geq C_2\sinh^2(\lambda+\mu)\ .
\end{equation}
\end{lemma}
\begin{proof} Indeed, since $\sigma_2$ and $\sigma_3$ are orthonormal and spacelike, by \eqref{eqn:RotationVectorFieldOverGeodesic},
\begin{equation*}
\|Y_\zeta(\tilde{\zeta}(s))\|^2_\opmat = \sinh^2(\lambda-\mu)\cos^2(s) + \sinh^2(\lambda+\mu)\sin^2(s) \geq \frac{1}{2}\sin^2(\epsilon)\sinh^2(\lambda+\mu)\ .
\end{equation*}
The result follows upon setting $C_2:=\sin^2(\epsilon)$.
\end{proof}
\begin{lemma}
For all $s\in[0,\pi]$,
\begin{equation*}
\langle Y_\zeta(\tilde{\zeta}(s)),\partial_s\tilde{\zeta}(s)\rangle = -\sinh(\lambda+\mu)\sinh(\lambda-\mu)\ .
\end{equation*}
\end{lemma}
\begin{proof} Indeed, differentiating \eqref{eqn:ParametrizationOfSecondGeodesic} yields, for all $s$,
\begin{equation*}
\begin{aligned}
\partial_s\tilde{\zeta}(s) &= -\cosh(\lambda-\mu)\sin(s)\sigma_0 + \cosh(\lambda+\mu)\cos(s)\sigma_1\\
&\qquad +\sinh(\lambda+\mu)\cos(s)\sigma_2 - \sinh(\lambda-\mu)\sin(s)\sigma_3\ .
\end{aligned}
\end{equation*}
The result follows upon taking the inner product with \eqref{eqn:RotationVectorFieldOverGeodesic}.
\end{proof} 
\section{Representations and invariant sets}
\label{s:RepresentationsAndInvariantSets}
We now study globally hyperbolic, convex, Cauchy compact (GHCC) anti-de Sitter $(2,1)$-spacetimes from the point of view of the work \cite{Mes} of Mess, which was inspired by works of Thurston on hyperbolic 3-manifolds. From this perspective, whenever any such spacetime has convex, spacelike boundary, its universal cover embeds as a convex subset $M$ of $\AdS$ which is equivariant with respect to some quasi-Fuchsian representation $\rho$, referred to as its \emph{holonomy}. In this section, we consider the case where the boundary $\partial M$ is smooth and strongly convex and we study the elementary geometry of this set. Of particular interest to us will be the asymptotic behaviour of the unit normal vector field $N$. Its asymptotics in the radial direction are relatively straightforward to determine and are addressed in Lemma \ref{lemma:BoundsOfRadialComponentI}. Its asymptotics in the pseudoangular direction are, however, more challenging. It is here that our study diverges most significantly from the Riemannian case addressed in \cite{Sch}, as it is now necessary to control the rate at which $N$ becomes lightlike. This is precisely the content of Lemma \ref{lemma:GrowthRateOfTilt}, which will be proven using the ``H'' Lemma (Lemma \ref{lemma:HLemma}) of Section \ref{ss:TimelikeGeodesicsInTheMatrixModel}.

\subsection{Fuchsian representations}
\label{ss:FuchsianRepresentations}
Let $\Pi$ be the fundamental group of a compact surface $S$. We recall that a representation $\rho:\Pi\rightarrow\PSL$ is said to be \emph{Fuchsian} whenever it is discrete, faithful and torsion free. The set of conjugacy classes of Fuchsian representations identifies with the \emph{Teichm\"uller space} $\teich_\oprep(S)$ of $S$. Given two Fuchsian representations $\rho,\tilde{\rho}\in\teich_\oprep(S)$, we define their \emph{Lipschitz distance} $d_\Lip(\rho,\tilde{\rho})$ to be the infimal value of $\ln(D)$ for which there exists a $D$-bilipschitz $(\rho,\tilde{\rho})$-equivariant map $\phi:\hyp^2\rightarrow\hyp^2$. By the Arzela--Ascoli theorem and lower semi-continuity of the Lipschitz seminorm, this infimum is always attained, and it follows that $d_\Lip$ defines a metric over $\teich_\oprep(S)$.

The Lipschitz distance is related to the Thurston's asymmetric distance (see \cite{Thu2}), the latter being defined similarly, with the infimum being taken instead over $D$-Lipschitz homeomorphisms which are not necessarily bilipschitz. The above defined Lipschitz distance thus bounds Thurston's distance from above. Thurston also characterized his distance in terms of translation lengths, which yields in the present case the following useful estimate. Recall from Section~\ref{sec:homspacerps} that, for any Fuchsian representation $\rho:\Pi\rightarrow\PSL$, its \emph{translation length} $l_\rho:\Pi\rightarrow[0,\infty[$ is defined by
\begin{equation*}
l_\rho(\gamma) := \inf_{p\in\hyp^2}d\big(\rho(\gamma)\cdot p,p\big)\ .
\end{equation*}

\begin{lemma}
Let $\rho,\tilde{\rho}:\Pi\rightarrow\PSL$ be Fuchsian representations. Then
\begin{equation}
\label{eqn:thurstonlipschitz}
\sup_{\gamma\in\Pi\setminus\{e\}}\left|\ln\big(l_\rho(\gamma)\big) - \ln\big(l_{\tilde{\rho}}(\gamma)\big)\right| \leq d_\Lip(\rho,\tilde{\rho})\ .
\end{equation}
\end{lemma}

\begin{remark}
By \cite{Thu2}, the left-hand side of \eqref{eqn:thurstonlipschitz} is equal to the Thurston distance between $\rho$ and $\tilde\rho$.
\end{remark}

\begin{proof} Indeed, denote $d:=d_\Lip(\rho,\tilde{\rho})$ and $D:=e^d$. Denote $S:=\hyp^2/\rho(\Pi)$ and $\tilde{S}:=\hyp^2/\tilde{\rho}(\Pi)$. For all non-trivial $\gamma\in\Pi$, let $c(\gamma)$ (resp. $\tilde{c}(\gamma)$) denote the unique closed geodesic freely homotopic to $\rho(\gamma)$ (resp. $\tilde{\rho}(\gamma)$) in $S$ (resp. $\tilde{S}$). Let $\hat{\phi}:\hyp^2\rightarrow\hyp^2$ be a $(\rho,\tilde{\rho})$-equivariant $D$-bilipschitz map, and let $\phi$ denote its projection to $S$. For all non-trivial $\gamma\in\Pi$,
\begin{equation*}
l_\rho(\gamma) = \Length\big(c(\gamma)\big) \geq \frac{1}{D}\Length\big((\phi\circ c)(\gamma)\big)\ .
\end{equation*}
Since $\tilde{c}(\gamma)$ minimizes length in the free homotopy class of $\tilde{\rho}(\gamma)$, it follows that
\begin{equation*}
l_\rho(\gamma) \geq \frac{1}{D}\Length\big(\tilde c(\gamma)\big) = \frac{1}{D} l_{\tilde \rho}(\gamma)\ .
\end{equation*}
Taking logarithms yields
\begin{equation*}
\ln\big(l_{\tilde{\rho}}(\gamma)\big) - \ln\big(l_{\rho}(\gamma)\big) \leq \ln(D) = d\ .
\end{equation*}
Likewise, by symmetry,
\begin{equation*}
\ln\big(l_\rho(\gamma)\big) - \ln\big(l_{\tilde{\rho}}(\gamma)\big) \leq \ln(D) = d\ .
\end{equation*}
The result follows upon combining these estimates.
\end{proof}

For any Fuchsian representation $\rho:\Pi\rightarrow\PSL$, and for any $p\in\hyp^2$, we define the function $l_{\rho,p}:\Pi\rightarrow[0,\infty[$ by
\begin{equation}
\label{eqn:defn:PointedLengthFunction}
l_{\rho,p}(\gamma) := d\big(\rho(\gamma)\cdot p,p\big)\ .
\end{equation}
Trivially, for all $\gamma\in\Pi$,
\begin{equation*}
l_\rho(\gamma) = \inf_{p\in\hyp^2}l_{\rho,p}(\gamma)\ .
\end{equation*}
\begin{lemma}
\label{lemma:HowKappaIsReallyUsed}
Let $\rho,\tilde{\rho}:\Pi\rightarrow\PSL$ be Fuchsian representations, and let $d:=d_\Lip(\rho,\rho')$ denote the Lipschitz distance between the two. For every point $p\in\hyp^2$,
\begin{equation}
\label{eqn:HowKappaIsReallyUsed}
\limsup_{\left|\gamma\right|\rightarrow\infty}\frac{\left|l_{\rho,p}(\gamma)-l_{\tilde{\rho},p}(\gamma)\right|}{l_{\rho,p}(\gamma)+l_{\tilde{\rho},p}(\gamma)} \leq \tanh(d/2)\ ,
\end{equation}
where $|\gamma|$ here denotes the word length of $\gamma$ for an arbitrary finite set of generators of $\Pi$.
\end{lemma}
\begin{proof}Denote $d:=d_\Lip(\rho,\tilde{\rho})$ and denote $D:=e^d$. Let $\phi:\hyp^2\rightarrow\hyp^2$ be a $(\rho,\tilde{\rho})$-equivariant $D$-bilipschitz map. For all non-trivial $\gamma\in\Pi$,
\begin{equation*}
\begin{aligned}
l_{\tilde{\rho},p}(\gamma) &= d\big(\tilde{\rho}(\gamma)\cdot p,p\big)\\
&\leq d\big(\tilde{\rho}(\gamma)\cdot p,\tilde{\rho}(\gamma)\cdot\phi(p)\big) + d\big(\tilde{\rho}(\gamma)\cdot\phi(p),\phi(p)\big) + d\big(\phi(p),p\big)\\
&\leq 2d\big(\phi(p),p\big) + d\big(\phi(\rho(\gamma)\cdot p),\phi(p)\big)\\
&\leq 2d\big(\phi(p),p\big) + D\cdot d\big(\rho(\gamma)\cdot p,p\big)\\
&\leq C + D\cdot l_{\rho,p}(\gamma)\ ,
\end{aligned}
\end{equation*}
for some $C$ independent of $\gamma$. By symmetry,
\begin{equation*}
l_{\rho,p}(\gamma) \leq C + D\cdot l_{\tilde{\rho},p}(\gamma)\ .
\end{equation*}
Hence,
\begin{equation*}
\frac{1}{D} \leq \liminf_{\left|\gamma\right|\rightarrow\infty}\frac{l_{\tilde{\rho},p}(\gamma)}{l_{\rho,p}(\gamma)}
\leq \limsup_{\left|\gamma\right|\rightarrow\infty}\frac{l_{\tilde{\rho},p}(\gamma)}{l_{\rho,p}(\gamma)}
\leq D\ .
\end{equation*}

Consider now the M\"obius map
\begin{equation*}
M(x) := \frac{1-x}{1+x}\ .
\end{equation*}
This map sends the interval $[1/D,D]$ to the interval $[-\tanh(d/2),\tanh(d/2)]$. Thus, since
\begin{equation*}
\frac{\left|l_{\rho,p}(\gamma)-l_{\tilde{\rho},p}(\gamma)\right|}{l_{\rho,p}(\gamma)+l_{\tilde{\rho},p}(\gamma)}
=\left|M\bigg(\frac{l_{\tilde{\rho},p}(\gamma)}{l_{\rho,p}(\gamma)}\bigg)\right| ,
\end{equation*}
the result follows.
\end{proof}

\subsection{Quasi-Fuchsian representations}
\label{ss:QuasiFuchsianRepresentations}
Recall that a representation $\rho:=(\rho_r,\rho_l):\Pi\rightarrow\PSL\times\PSL$ is said to be \emph{quasi-Fuchsian} whenever each of its two components $\rho_r$ and $\rho_l$ is Fuchsian. The fundamental invariant
\begin{equation}
\label{eqn:DefinitionOfKappa}
\kappa_\rho := \tanh\big(d_\Lip(\rho_r,\rho_l)/2\big) < 1
\end{equation}
will play a key role throughout what follows.

Let $\Lambda:=\Lambda_\rho$ denote the limit set of $\rho$, as in Section \ref{ss:EinsteinSpace}. Let $M\subseteq\AdS$ be a strongly convex, $\rho$-invariant subset with non-trivial interior and smooth, spacelike boundary such that $\ideal M=\Lambda$. The representation $\rho$ acts properly discontinuously on $M$, the quotient $M/\rho(\Pi)$ is a GHCC anti-de Sitter spacetime, and in \cite{Mes} (see also \cite{Bar3}), Mess showed that every GHCC anti-de Sitter spacetime arises in this manner.

Recall that the \emph{invisible domain} $\Inv(\Lambda)$ of $\Lambda$ is the set of all points of $\AdS$ that lie in spacelike relation to every point of $\Lambda$. Equivalently, it is the set of all points that cannot be joined to $\Lambda$ by a lightlike segment. Its closure is the dual convex subset of the convex hull $\Conv(\Lambda)$ of $\Lambda$ (see, for example, \cite{FS4,BS2}). The representation $\rho$ acts properly discontinuously on $\Inv(\Lambda)$, the quotient $M/\rho(\Pi)$ is a GHMC anti-de Sitter spacetime, and every GHMC anti-de Sitter spacetime arises in this manner.

The invisible domain can also be characterized as the envelope of planes tangent to $\Ein$ along $\Lambda$. Since every tangent plane of $\Ein$ along $\Lambda$ is also a support tangent of $M$, it follows that $M$ is contained in $\Inv(\Lambda)$. Taking quotients, we see that the GHCC spacetime $M/\rho(\Pi)$ embeds into the GHMC spacetime $\Inv(\Lambda)/\rho(\Pi)$.

We denote by $\partial M$ the boundary of $M$, viewed as a submanifold with boundary of $\AdS$. The topological closure of $M$ in $\AdS\cup\Ein$ is the set $\overline{M}:=M\cup\Lambda$, and the topological boundary of this closure is $\partial M\cup\Lambda$. Note that, since $\overline{M}$ is a convex set with non-trivial interior, its boundary is homeomorphic to the $2$-sphere. Since $\Lambda$ is a Jordan curve, it follows by the Jordan curve theorem that $\partial M$ consists of two connected components, each homeomorphic to the $2$-disk. Let $\Sigma$ be a component of $\partial M$. $\Sigma$ is a $\rho$-invariant, strongly convex, smooth spacelike surface in $\AdS$ and $\partial_\infty\Sigma\subseteq\Lambda$.

We denote by $N$ the outward-pointing unit normal vector field over $\partial M$. We will study the asymptotic behaviour of $N$ by comparing it to certain well-chosen vector fields. We first pick a base point $p_0\in\AdS$ lying in the invisible domain of $\Lambda$, and we consider the \emph{unit radial vector field} $e_r$ pointing away from this point.

Let $\Omega^s:=\Omega^s_{p_0}\subseteq\AdS$ denote the set of all points in spacelike relation to $p_0$, let $\ideal\Omega$ denote its ideal boundary in $\ideal\AdS$ and note that
\begin{equation*}
\ideal M=\ideal\Sigma=\Lambda\subseteq\partial_\infty\Omega\ .
\end{equation*}
In particular, the complements of $\Omega$ in $\Sigma$ and $M$ are compact. Note that, using pseudospherical polar coordinates of $\Omega$, we may view $e_r$ as the first component of an orthonormal frame $(e_r,e_\theta,e_t)$ of $\Omega$. We will make more use of this frame in Section \ref{ss:AutormorphicVectorFields}, below.

The radial component of the normal vector is controlled as follows.
\begin{lemma}
\label{lemma:BoundsOfRadialComponentI}
There exists a compact subset $K\subseteq\Sigma$, and $C>0$ such that, for all $p\in\Sigma\setminus K$,
\begin{equation}
\label{eqn:BoundsOfRadialComponentI}
\frac{1}{C} \leq -\langle N(p),e_r(p)\rangle \leq C\ .
\end{equation}
\end{lemma}
\begin{proof} For all $p\in\Sigma$, let $\pi_p^\pa(\Lambda)\subseteq\opU_p\AdS$ denote the radial projection of $\Lambda$ into the space of unit, spacelike tangent vectors at $p$. Since every such point lies in the invisible domain of $\Lambda$, this is always a compact set of unit, spacelike vectors over $p$.

Let $\Sigma_0$ be a fundamental domain of $\Sigma$. For $p\in\Sigma_0$ and $q\in\Lambda$, let $\eta$ denote the unit speed spacelike geodesic from $p$ to $q$ such that $\eta(0)=p$. Its derivative $\xi:=\partial\eta(0)$ is an element of $\pi_p^\pa(\Lambda)$, and every element of $\pi_p^\pa(\Lambda)$ arises in this manner. By strict convexity, $\xi$ is strictly inward-pointing, so that
\begin{equation*}
\langle N(p),\xi\rangle > 0\ .
\end{equation*}
On the other hand, trivially,
\begin{equation*}
\langle N(p),\xi\rangle < \infty\ .
\end{equation*}
It follows by compactness of $\Lambda$ and $\Sigma_0$ that there exists $C>0$ such that, for all $p\in\Sigma_0$, and for all $\xi\in\pi_p^\pa(\Lambda)$,
\begin{equation}
\label{eqn:proof:BoundsOfRadialComponentI}
\frac{2}{C} < \langle N(p),\xi\rangle < \frac{C}{2}\ .
\end{equation}

We now conclude using $\rho$-invariance. Let $(p_m)_\minn$ be a diverging sequence of points of $\Sigma$. For all $m$, let $\gamma_m\in\Pi$ be such that $\tilde{p}_m:=\rho(\gamma_m)\cdot p_m$ is an element of $\Sigma_0$. Suppose furthermore that $(\tilde{p}_m)_\minn$ converges to some limit point $\tilde{p}_\infty$, say. For all $m$, denote $q_m:=\rho(\gamma_m)\cdot p_0$ and $\xi_m:=-D\rho(\gamma_m)(p_m)\cdot e_r(p_m)$. Since $(q_m)_\minn$ diverges, we may suppose that it converges to some point $q_\infty$, say, of $\Lambda$. It then follows that $(\xi_m)_\minn$ converges to some point $\xi_\infty$, say, of $\pi_{\tilde{p}_\infty}^\pa(\Lambda)$. Consequently,
\begin{equation*}
\begin{aligned}
\lim_{m\rightarrow\infty}\langle N(p_m),e_r(p_m)\rangle
&=\lim_{m\rightarrow\infty}\langle D\rho(\gamma_m)(p_m)\cdot N(p_m),D\rho(\gamma_m)(p_m)\cdot e_r(p_m)\rangle\\
&=\lim_{m\rightarrow\infty}-\langle N(\tilde{p}_m),\xi_m\rangle\\
&=-\langle N(\tilde{p}_\infty),\xi_\infty\rangle\ ,
\end{aligned}
\end{equation*}
and the result now follows by \eqref{eqn:proof:BoundsOfRadialComponentI}.
\end{proof}
A similar argument yields the following more general result.
\begin{lemma}
For any $\rho$-invariant vector field $X$ over $M$, there exists a compact subset $K\subseteq M$ and $C>0$ such that, for all $p\in M\setminus K$,
\begin{equation*}
\left|\langle X(p),e_r(p)\rangle\right| < C\ .
\end{equation*}
\end{lemma}

\subsection{The cylindrical coordinate frame}
\label{ss:TheCylindricalCoordinateFrame}
We now control the rate at which the normal vector field becomes lightlike. To this end, let $\zeta_0\in\geodesics$ be a future oriented, timelike geodesic in $\AdS$ such that $p_0\in\zeta_0$. Cylindrical polar coordinates around this geodesic yield a natural orthonormal frame $(f_r,f_\theta,f_t)$ of $\AdS$ as follows. We denote first by $f_r$ the unit, radial vector field pointing away from $\zeta_0$. We construct the two remaining components of this frame using the stabilizer $\Stab(\zeta_0)$ of $\zeta_0$ in $\PSL\times\PSL$. Indeed, note first that this group is isomorphic to $\mathbb{S}^1\times\mathbb{S}^1$. The subgroup
\begin{equation*}
\Stab(p_0)\subseteq\Stab(\zeta_0)
\end{equation*}
is a copy of $\mathbb{S}^1$ corresponding to \emph{rotations} of $\AdS$ \emph{around} the axis $\zeta_0$. We denote by $f_\theta$ the unit vector field over $\AdS\setminus\zeta_0$ in the direction of these rotations. Note that this is none other than the unit vector field parallel to the field $Y_{\zeta}$ studied in Section \ref{ss:VectorFieldsAroundTimelikeGeodesics}.

Finally, we define
\begin{equation*}
f_t := -f_r\times f_\theta\ .
\end{equation*}
This is the unit vector field in the direction of \emph{translations} of $\AdS$ \emph{along} the axis $\zeta_0$, and is none other than the unit vector field parallel to the field $X_\zeta$, also studied in Section \ref{ss:VectorFieldsAroundTimelikeGeodesics}.

The main estimate of this section will be best expressed in terms of the \emph{pseudoangle} $\angle_a(U_p,V_p)$ between two future oriented, unit timelike vectors in $\opU^+_p\AdS$ which we recall is defined to be the hyperbolic distance between $U_p$ and $V_p$ in $\hyp^2\cong \opU^+_p\AdS$. The rest of this section is devoted to proving the following estimate.
\begin{lemma}
\label{lemma:GrowthRateOfTilt}
For every $\rho$-invariant, unit timelike vector field $\xi$ over $M$ and for all $k>\kappa_\rho$, there exists a compact subset $K$ of $M$ and $C\geq 0$ such that, for all $p\in M\setminus K$,
\begin{equation}
\label{eqn:GrowthRateOfTilt}
\angle_a\big(\xi(p),f_t(p)\big) \leq k\cdot d(p_0,p) + C\ .
\end{equation}
\end{lemma}
\begin{remark} In particular, this applies to the case where $\xi=N$ is the unit normal vector field.
\end{remark}

For all $p\in\Sigma$, we denote by $\zeta_p:\R\rightarrow\AdS$ the unit speed parametrized timelike geodesic such that
\begin{equation}
\label{eqn:defn:ZetaP}
\zeta_p(0) = p\qquad\text{and}\qquad \partial_s\zeta_p(0) = N(p)\ ,
\end{equation}
and we denote also by $\zeta_p$ its image in $\AdS$.

For all $p$, we denote by $\pi^\cyl_p(\Lambda)$ the radial projection of $\Lambda$ onto $\zeta_p$ and, by abuse of terminology, we denote also by $\pi^\cyl_p(\Lambda)$ its preimage under the parametrization $\zeta_p$.
\begin{lemma}
\label{lemma:ProjectionOntoGeodesicOfLambda}
There exists $\epsilon>0$ such that, for all $p\in\Sigma$,
\begin{equation*}
\pi_p^\cyl(\Lambda)\subseteq[-\pi/2+\epsilon,-\epsilon]\ .
\end{equation*}
\end{lemma}
\begin{proof} Note that the problem is $\rho$-invariant. Let $\Sigma_0$ be a fundamental domain of $\Sigma$ and choose $p\in\Sigma_0$. Choose a Kleinian chart such that $p=(1,0,0)^t$ and $\zeta_p$ is the upward-pointing $x_0$-axis. In this chart, the outward-pointing Minkowski unit normal to $\Sigma$ at $p$ is $N:=(1,0,0)^t$. Let $\Sigma^*$ denote the dual convex surface of $\Sigma$, and observe that $\Sigma^*$ contains the point $(-1,0,0)^t$ with outward-pointing Minkowski normal $(-1,0,0)^t$ at this point. Thus, if $\tilde{M}$ denotes the convex set bounded by $\Sigma$ and $\Sigma^*$, then $\tilde{M}$ is contained in the slab
\begin{equation*}
S:=\{(x_0,x_1,x_2)\ |\ -1\leq x_0\leq 1\}\ .
\end{equation*}
Since $\Sigma$ and $\Sigma^*$ are strongly convex, and since $\Lambda\subseteq\ideal\tilde{M}$, it follows that $\Lambda$ is contained in the slightly narrower slab
\begin{equation*}
S':=\{(x_0,x_1,x_2)\ |\ -1+\epsilon_1\leq x_0\leq 1-\epsilon_1\}\ ,
\end{equation*}
for some $\epsilon_1>0$, which can be chosen uniformly for all $p \in \Sigma_0$. The orthogonal projection of $\Lambda$ onto the $x_0$-axis is thus contained in the interval $[-1+\epsilon_1,1-\epsilon_1]$, and the preimage under $\zeta_p$ of this interval is contained in $[-\pi/2+\epsilon,-\epsilon]$, for some $\epsilon>0$, which may likewise be chosen uniformly for all $p \in \Sigma_0$. The result now follows by $\rho$-invariance.
\end{proof}

\begin{lemma}
\label{lemma:LocationOfFurthestPoint}
There exists a compact subset $K$ of $\Sigma$ and $\epsilon>0$ such that, for all $p\in\Sigma\setminus K$, the furthest point of $\zeta_p$ from $\zeta_0$ is an element of $\zeta_p([-\pi/2+\epsilon,-\epsilon])$.
\end{lemma}
\begin{proof}
Let $(p_m)_\minn$ be a divergent sequence in $\Sigma$. For all sufficiently large $m$, let $t_m\in[-\pi/2,\pi/2]$ be such that $\zeta_{p_m}(t_m)$ is the furthest point in $\zeta_{p_m}$ from $\zeta_0$, and suppose that $(t_m)_\minn$ converges to some limit $t_\infty$, say. It will suffice to show that $t_\infty$ is an element of the open subinterval $]-\pi/2,0[$.

Let $\Sigma_0$ be a fundamental domain of $\Sigma$ and, for all $m$, let $\gamma_m\in\Pi$ be such that $\tilde{p}_m:=\rho(\gamma_m)\cdot p_m$ is an element of $\Sigma_0$. We may suppose that $(\tilde{p}_m)_\minn$ converges to some limit $\tilde{p}_\infty$, say. For all $m$, denote $\zeta_m:=\rho(\gamma_m)\cdot\zeta_0$ and $\tilde{\zeta}_m:=\rho(\gamma_m)\circ\zeta_{p_m}$, and note that $(\tilde{\zeta}_m)_\minn$ converges to $\tilde{\zeta}_\infty:=\zeta_{\tilde{p}_\infty}$. For all $m$, let $\eta_m$ denote the unit speed parametrized spacelike geodesic in $\AdS$ joining the furthest points of $\tilde{\zeta}_m$ and $\zeta_m$ such that $\eta_m$ points towards $\zeta_m$, and $\eta_m(0)\in\tilde{\zeta}_m$. In particular, for all $m$, $\eta_m(0)=\tilde{\zeta}_m(t_m)$ and $\partial_s\eta_m(0)$ is orthogonal to $\tilde{\zeta}_m$ at this point. For all $m$, let $\eta_m(\infty)$ denote the end-point of $\eta_m$ in $\ideal\AdS$.

Recall now the full translation map $\tau$ defined in Section \ref{sss:AntideSitterSpaceInMatrixModel}. Since $(\tau(\zeta_m))_\minn$ diverges in $\hyp^2\times\hyp^2$, we may suppose that it converges to some point $q$, say, of the limit set $\Lambda$. On the other hand $(\tau(\tilde{\zeta}_m))_\minn$ converges to $\tau(\tilde\zeta_\infty)$, which is a finite point of $\hyp^2\times\hyp^2$. It follows by the ``H''-Lemma (Lemma \ref{lemma:HLemma}) that $(\eta_m(\infty))_\minn$ also converges to $q$. The sequence $(t_m)_\minn$ therefore converges to a point of $\pi^\cyl_{\tilde{p}_\infty}(\Lambda)$, and the result now follows by Lemma \ref{lemma:ProjectionOntoGeodesicOfLambda}.
\end{proof}

\begin{lemma}
\label{lemma:GrowthRateOfTiltStepI}
For all $k>\kappa_\rho$, there exists a compact subset $K$ of $\Sigma$ and $C>0$ such that, for all $p\in\Sigma\setminus K$, and for all $s,\tilde{s}\in[-\pi/2,0]$,
\begin{equation}
\label{eqn:GrowthRateOfTiltStepI}
\angle_a\big(\partial_s\zeta_p(s),f_t(\zeta_p(s))\big) \leq k\cdot d\big(\zeta_0(\tilde{s}),\zeta_p(s)\big) + C\ .
\end{equation}
\end{lemma}
\begin{proof} Let $\Sigma_0$ be a fundamental domain of $\Sigma$ and suppose that $\zeta_0$ passes through $\Sigma_0$. Let $\tau$ denote the full translation map defined in Section \ref{sss:AntideSitterSpaceInMatrixModel} and denote $(q_r,q_l):=\tau(\zeta_0)$. Without loss of generality, we may suppose that there exists $\gamma\in\Pi$ such that $\zeta_p = \rho(\gamma)\circ\zeta_0$. As in Section \ref{ss:TimelikeGeodesicsInTheMatrixModel}, denote
\begin{equation*}
\lambda:=\frac{1}{2}d\big(q_r,\tau_r(\zeta_p)\big)\qquad\text{and}\qquad\mu:=\frac{1}{2}d\big(q_l,\tau_l(\zeta_p)\big)\ ,
\end{equation*}
and note that both $\lambda$ and $\mu$ are non-zero.

By Lemma \ref{lemma:LocationOfFurthestPoint}, the furthest point $s_0$ of $\zeta_p$ from $\zeta_0$ is an element of $[-\pi/2+\epsilon,-\epsilon]$, so that
\begin{equation}
\label{lemma:GrowthRateOfTiltStepI:I}
\left|s-s_0\right|\leq\frac{\pi}{2}-\epsilon\ .
\end{equation}
Similarly, if $\tilde{s}_0$ denotes the furthest point of $\zeta_0$ from $\zeta_p$, then
\begin{equation}
\label{lemma:GrowthRateOfTiltStepI:II}
\left|\tilde{s}-\tilde{s}_0\right|\leq\frac{\pi}{2}-\epsilon\ .
\end{equation}

By \eqref{lemma:GrowthRateOfTiltStepI:I}, since $f_t$ is the unit vector field in the direction of $X_{\zeta_0}$, \eqref{eqn:EstimateOfTranslationField} and \eqref{eqn:TangentComponentOfTranslationField} yield
\begin{equation*}
\left|\langle \partial_s\zeta_p(s),f_t(\zeta_p(s))\rangle\right| \leq \frac{1}{C_1}\cosh(\lambda-\mu)\ .
\end{equation*}
However, by \eqref{eqn:HowKappaIsReallyUsed}, for $q$ outside a sufficiently large compact set,
\begin{equation*}
\cosh(\lambda-\mu) \leq \cosh\big(k\cdot(\lambda+\mu)\big)\ ,
\end{equation*}
so that
\begin{equation*}
\left|\langle \partial_s\zeta_p(s),f_t(\zeta_p(s))\rangle\right| \leq \frac{1}{C_2}\cosh\big(k\cdot(\lambda+\mu)\big)\ .
\end{equation*}
Taking logarithms yields
\begin{equation*}
\angle_a\big(\partial_s\zeta_p(s),f_t(\zeta_p(s))\big) \leq k\cdot(\lambda+\mu) + C_3\ .
\end{equation*}

By \eqref{lemma:GrowthRateOfTiltStepI:II}, \eqref{eqn:DistanceBetweenPointsOfTimelikeGeodesics} yields
\begin{equation*}
\angle_a\big(\partial_s\zeta_p(s),f_t(\zeta_p(s))\big) \leq k\cdot d\big(\zeta_0(\tilde{s}),\zeta_p(s)\big) + C_3\ ,
\end{equation*}
and the result follows.
\end{proof}

To conclude, we will require the following technical result. Let $\Sigma_1$ and $\Sigma_2$ denote the two boundary components of $M$. Let $\tilde{M}_1$ (resp. $\tilde{M}_2$) denote the convex subset of $\AdS$ bounded by $\Sigma_1$ (resp. $\Sigma_2$) and its dual $\Sigma_1^*$ (resp. $\Sigma_2^*$).
\begin{lemma}
$M$ is contained in the union of $\tilde{M}_1$ and $\tilde{M}_2$.
\end{lemma}
\begin{proof} Let $\Conv(\Lambda)$ denote the convex hull of $\Lambda$ in $\AdS$. Since $\Conv(\Lambda)$ is the intersection of all convex sets containing $\Lambda$, in particular,
\begin{equation*}
\Conv(\Lambda) \subseteq \tilde{M}_1,\tilde{M}_2,M\ .
\end{equation*}
Note that $M\setminus\Conv(\Lambda)$ consists of two connected components, one corresponding to each boundary component of $M$. Since $\tilde{M}_1$ contains $\Sigma_1$ and $\Conv(\Lambda)$, by convexity, it contains the entire connected component of $M\setminus\Conv(\Lambda)$ containing $\Sigma_1$. Similarly, $\tilde{M}_2$ contains the entire connected component of $M\setminus\Conv(\Lambda)$ containing $\Sigma_2$. It follows that $M\subseteq\tilde{M}_1\cup\tilde{M}_2$, as desired.
\end{proof}

We are now ready to prove Lemma \ref{lemma:GrowthRateOfTilt}.
\begin{proof}[Proof of Lemma \ref{lemma:GrowthRateOfTilt}.]
Choose $q\in M$. Without loss of generality, we may assume that $q\in\tilde{M}_1$. Let $p\in\Sigma$ denote the point in $\Sigma$ maximizing timelike distance to $q$ and note that $q=\zeta_p(s)$ for some $s\in [-\pi/2, 0]$. By $\rho$-invariance and cocompactness,
\begin{equation*}
\angle_a\big(\xi(p),\partial_s\zeta_p(s)\big) \leq C\ ,
\end{equation*}
for some $C$ that is independent of $p$. The result follows by \eqref{eqn:GrowthRateOfTiltStepI} and the triangle inequality.
\end{proof} 
\section{Automorphic vector fields}
\label{s:AutomorphicVectorFields}
Let $M\subseteq\AdS$ be a $\rho$-invariant subset with smooth, strongly convex, spacelike boundary, as studied in Section \ref{ss:QuasiFuchsianRepresentations}. Let $\alpha$ be a $\rho$-cocycle, and consider a $(\rho,\alpha)$-automorphic vector field $\xi$ over $M$. Working in the Kleinian model of $\AdS$ as in Section \ref{ss:PreliminariesKleinianModel}, we view $\xi$ as a vector field over a subset of $\R^{2,1}$. The main result of this section is Lemma \ref{lemma:ContinuousExtension}, which states that $\xi$ extends continuously to $\Lambda$.

\label{ss:AutormorphicVectorFields}
\subsection{The Kleinian chart}
Let $p_0$ be a point of $\Conv(\Lambda)$, and let $\zeta_0$ be an oriented timelike geodesic through $p_0$. We will henceforth suppose that the Kleinian chart has been chosen such that $p_0$ identifies with the origin, and $\zeta_0$ identifies with the upward-pointing $x_0$-axis.

As in Section \ref{sss:AntideSitterandEinstein}, we denote by $\Omega$ the image of the Kleinian chart in $\mathbb{R}^{2,1}$ and by $\Omega^s$ the set of all points of $\R^{2,1}$ which are in spacelike relation to the origin. We recall the pseudospherical polar coordinates defined over this domain by \eqref{eqn:defn:PseudoSphericalCoordinates}, and we denote the corresponding coordinate frame by $(\partial_r,\partial_\theta,\partial_t)$. That is, $\partial_r$ is the radial vector field, $\partial_\theta$ is the vector field of rotations about the $x_0$-axis, and, for any fixed $\theta$, $\partial_t$ is the vector field of boosts around $(0,\cos(\theta),\sin(\theta))^t$.

The above coordinate frame is orthogonal, but not orthonormal, with respect to both the Minkowski and anti-de Sitter metrics. We denote by $(e_r^m,e_\theta^m,e_t^m)$ the \emph{Minkowski orthonormal frame} parallel $(\partial_r,\partial_\theta,\partial_t)$, and we readily verify that
\begin{equation}
\label{eqn:defn:MinkowskiOrthonormalFrame}
e_r^m = \partial_r\ ,\qquad e_\theta^m = \frac{1}{r\cosh(t)}\partial_\theta\qquad\text{and}\qquad e_t^m = \frac{1}{r}\partial_t\ .
\end{equation}
We likewise denote by $(e_r^a,e_\theta^a,e_t^a)$ the \emph{anti de Sitter orthonormal frame} parallel to $(\partial_r,\partial_\theta,\partial_t)$, and we verify that
\begin{equation}
\label{eqn:defn:AntiDeSitterOrthonormalFrame}
e_r^a = \delta(\mathbf{y})\partial_r\ ,\qquad e_\theta^a = \frac{\sqrt{\delta(\mathbf{y})}}{r\cosh(t)}\partial_\theta\qquad\text{and}\qquad e_t^a = \frac{\sqrt{\delta(\mathbf{y})}}{r}\partial_t\ ,
\end{equation}
where $\delta$ is as in \eqref{eqn:DefinitioOfDelta}.

We recall also the frame $(f_r^a,f_\theta^a,f_t^a)$ defined in Section \ref{ss:VectorFieldsAroundTimelikeGeodesics}. That is, $f^a_r$ is the anti-de Sitter unit radial vector field pointing outwards from $\zeta_0$, $f_\theta^a$ is the anti-de Sitter unit vector field of rotations around $\zeta_0$, and $f_t^a$ is the anti-de Sitter unit vector field of translations along $\zeta_0$. In particular, note that
\begin{equation*}
f_\theta^a = e_\theta^a\ .
\end{equation*}
We recall the pseudoangle defined in Section \ref{ss:TheCylindricalCoordinateFrame}.
\begin{lemma}
\label{lemma:FramesApproachExponentially}
There exists a compact subset $K$ of $M$ and $C>0$ such that, for all $\mathbf{y}\in M\setminus K$,
\begin{equation}
\label{eqn:FramesApproachExponentiallyI}
\angle_a\big(f_t^a(\mathbf{y}),e_t^a(\mathbf{y})\big)\leq Ce^{-r_a(\mathbf{y})}\ ,
\end{equation}
where $r_a(\mathbf{y})$ denotes the anti-de Sitter distance to $p_0$, as defined in \eqref{eqn:AntiDeSitterDistance}. In particular,
\begin{equation}
\label{eqn:FramesApproachExponentiallyII}
\lim_{r_a\rightarrow+\infty}\angle_a(f_t^a,e_t^a) = 0\ .
\end{equation}
\end{lemma}
\begin{proof} Indeed, let $C>0$ be such that $\Lambda$ is contained in the slab
\begin{equation*}
S := \big\{ (y_0,y_1,y_2)^t\ |\ \left|y_0\right| \leq C \big\}\ .
\end{equation*}
Consider now the point $\mathbf{y}:=(y_0,y_1,y_2)=\Phi(r,\theta,t)$. Without loss of generality, we may suppose that $\theta=0$, so that $y_2=0$. Differentiating $\Phi$ in the $r$-direction yields
\begin{equation*}
\partial_r = \frac{1}{r}(y_0,y_1,0)^t\ .
\end{equation*}
Hence, by \eqref{eqn:defn:AntiDeSitterOrthonormalFrame},
\begin{equation*}
e_r^a = \frac{(1-r^2)}{r}(y_0,y_1,0)^t\ .
\end{equation*}

Consider now the curve $\gamma(s) := (y_0,s,0)^t$. Trivially,
\begin{equation*}
\partial_s\gamma(s) = (0,1,0)^t\ .
\end{equation*}
In order to determine $f^a_r$, we determine the anti de Sitter length of this vector as follows. First, the preimage of $\gamma$ in $\mathbb{R}^{2,2}$ is
\begin{equation*}
\hat{\gamma}(s) = \frac{1}{\sqrt{1+y_0^2-s^2}}(1,y_0,s,0)^t\ .
\end{equation*}
Differentiating yields
\begin{equation*}
\partial_s\hat{\gamma}(s) = \frac{s}{\sqrt{1+y_0^2-s^2}^3}(1,y_0,0,0)^t + \frac{(1+y_0^2)}{\sqrt{1+y_0^2-s^2}^3}(0,0,1,0)^t\ .
\end{equation*}
Hence
\begin{equation*}
\|\partial_s\gamma\|_a^2 = \|\partial_s\hat{\gamma}\|^2 = \frac{(1+y_0^2)}{(1+y_0^2-s^2)^2}\ .
\end{equation*}
Thus, noting that $r^2=s^2-y_0^2$,
\begin{equation*}
f^a_r = \frac{(1-r^2)}{\sqrt{1+y_0^2}}(0,1,0)^t\ .
\end{equation*}

We now have
\begin{equation*}
\cosh\big(\angle_a(e_t^a,f_t^a)\big) = \cosh\big(\angle_a(e_r^a,f_r^a)\big) = \left|\langle e_r^a,f_r^a\rangle\right| = \frac{y_1}{r\sqrt{1+y_0^2}}\ .
\end{equation*}
Hence
\begin{equation*}
\sinh\big(\angle_a(e_t^a,f_t^a)\big) = \frac{y_0}{\sqrt{1-r^2}}{r(1+y_0^2)} = \frac{y_0}{r(1+y_0^2)\cosh\big(r_a(\mathbf{y})\big)}\ ,
\end{equation*}
and the result follows.
\end{proof}

\begin{lemma}
\label{lemma:GrowthRateOfTAngle}
For every $\rho$-invariant field $X$ of anti-de Sitter unit, future oriented timelike vectors over $M$, and for all $k>\kappa_\rho$, there exists a compact subset $K$ of $M$, and $C>0$ such that, for all $\mathbf{y}\in M\setminus K$,
\begin{equation}
\label{eqn:GrowthRateOfTAngle}
\angle_a\big(e_t^a(\mathbf{y}),X(\mathbf{y})\big) \leq k\cdot r_a(\mathbf{y}) + C\ .
\end{equation}
\end{lemma}
\begin{proof} Indeed, by Lemma \ref{lemma:GrowthRateOfTilt}, there exists $C>0$ such that, for all $\mathbf{y}$,
\begin{equation*}
\angle_a\big(f_t^a(\mathbf{y}),X(\mathbf{y})\big) \leq k\cdot r_a(\mathbf{y}) + C\ .
\end{equation*}
The result now follows by \eqref{eqn:FramesApproachExponentiallyI} and the triangle inequality.
\end{proof}

Let $\xi$ be an automorphic vector field over $M$, denote $\eta:=-(1/2)\nabla\times^a\xi$, and let $\omega^\xi:=(\omega_1^\xi,\omega_2^\xi)$ be as \eqref{eqn:defn:LieAlgebraValuedOneForm}.
\begin{lemma}
\label{lemma:EstimatesOfOmega}
For all $k>\kappa_\rho$, there exists a compact subset $K$ of $M$, and $C>0$ such that, for all $\mathbf{y}\in M\setminus K$, for each $i$, and for $\eta_1,\eta_2\in\{e_t^m,e_\theta^m\}$,
\begin{equation}
\label{eqn:defn:EstimatesOfOmega}
\begin{aligned}
\left|\langle\omega_i^\xi(e_r^m),e_r^m\rangle_e\right| &\leq C\ ,\\
\left|\langle\omega_i^\xi(\eta_1),e_r^m\rangle_e\right| &\leq C\delta(\mathbf{y})^{(1-k)/2}\ ,\\
\left|\langle\omega_i^\xi(e_r^m),\eta_1\rangle_e\right| &\leq C\delta(\mathbf{y})^{-(1+k)/2}\ ,\ \text{and}\\
\left|\langle\omega_i^\xi(\eta_1),\eta_2\rangle_e\right| &\leq C\delta(\mathbf{y})^{-k}\ .
\end{aligned}
\end{equation}
\end{lemma}
\begin{proof}Let $X$ be a $\rho$-invariant field of timelike directions over $M$. Let $M_0$ be a fundamental domain of $M$. Let $\opU^{\mathrm{sp}}\AdS$ denote the bundle of unit spacelike tangent vectors over $\AdS$. For $\mathbf{y}\in M$, let $\pi_{\mathbf{y}}^{\mathrm{sp}}(\Lambda)$ denote the radial projection of $\Lambda$ onto the fibre $\opU^{\mathrm{sp}}_{\mathbf{y}}\AdS$, and note that
\begin{equation*}
\Omega_0:=\bigcup_{\mathbf{y}\in M_0}\pi_{\mathbf{y}}^{\mathrm{sp}}(\Lambda)
\end{equation*}
is a compact subset of $\opU^{\mathrm{sp}}\AdS$. Let $\Omega$ be a relatively compact neighbourhood of $\Omega_0$ in $\opU^{\mathrm{sp}}\AdS$. For all $\mathbf{y}\in M_0$, let $\Omega_{\mathbf{y}}$ denote the fibre of $\Omega$ over $\mathbf{y}$, and extend this by $\rho$-invariance to a family $(\Omega_{\mathbf{y}})_{\mathbf{y}\in M}$ over the whole of $M$. For $\mathbf{y}\in M$, and for $\nu\in\Omega_{\mathbf{y}}$, we define the orthonormal frame $F^{\mathbf{y},\nu}:=\big(f_1^{\mathbf{y},\nu},f_2^{\mathbf{y},\nu},f_3^{\mathbf{y},\nu}\big)$ such that $f_1^{\mathbf{y},\nu}=\nu$, $f_2^{\mathbf{y},\nu}$ lies in the orthogonal projection of ${X}$ onto $\langle\nu\rangle^\perp$, and $f_3^{\mathbf{y},\nu}=f_1^{\mathbf{y},\nu}\times^a f_2^{\mathbf{y},\nu}$. Note that
\begin{equation*}
(F^{\mathbf{y},\nu})_{\mathbf{y}\in M_0,\nu\in\overline{\Omega}_\mathbf{y}}
\end{equation*}
is a compact subset of the orthonormal frame bundle of $\AdS$.

By compactness, there exists $C>0$ such that, for all $\mathbf{y}\in M_0$, $\nu\in\Omega_{\mathbf{y}}$, and for all $i$, $j$, $k$,
\begin{equation}
\label{eqn:proof:EstimatesOfOmegaI}
\left|\langle\omega_i^\xi(f_i^{\mathbf{y},\nu}),f_j^{\mathbf{y},\nu}\rangle_a\right| \leq C\ .
\end{equation}
By $\rho$-invariance, \eqref{eqn:proof:EstimatesOfOmegaI} holds for all $\mathbf{y}\in M$.

Choose $\mathbf{y}\in M$ and set $\nu:=e_r^a$. We estimate the distance between the frames $F^{\mathbf{y},\nu}$ and $(e_r^a,e_t^a,e_\theta^a)$. Indeed, by definition,
\begin{equation*}
f_1^{\mathbf{y},\nu} = e_r^a\ .
\end{equation*}
Now, since $e_t^a$ is orthogonal to $e_r^a$,
\begin{equation*}
\angle_a(f_2^{\mathbf{y},\nu},e_t^a) \leq \angle_a(X(\mathbf{y}),e_t^a)\ .
\end{equation*}
Hence, by \eqref{eqn:GrowthRateOfTAngle}, there exists a compact set $K$ and $C_1>0$ such that, over $M\setminus K$,
\begin{equation*}
\angle_a\big(f_2^{\mathbf{y},\nu},e_t^a\big) \leq k\cdot r_a(\mathbf{y}) + C_1\ .
\end{equation*}
Finally,
\begin{equation*}
\angle_a\big(f_3^{\mathbf{y},\nu},e_\theta^a\big)=\angle_a\big(f_2^{\mathbf{y},\nu},e_t^a\big) \leq k\cdot r_a(\mathbf{y}) + C_1\ .
\end{equation*}
This suffices for our purposes to describe the relationship between these two frames.

By the above estimates, for each $i$, and for all $\eta_1^a,\eta_2^a\in\{e_t^a,e_\theta^a\}$,
\begin{equation*}
\begin{aligned}
\left|\langle\omega_i^\xi(e_r^a),e_r^a\rangle_a\right| &\leq& &C_2\ , & &\\
\left|\langle\omega_i^\xi(\eta_1),e_r^a\rangle_a\right| &\leq& &C_2 e^{k\cdot r_a(\mathbf{y})} &\leq &\ C_3 \delta(\mathbf{y})^{-k/2}\ ,\\
\left|\langle\omega_i^\xi(e_r^a),\eta_1\rangle_a\right| &\leq& &C_2 e^{k\cdot r_a(\mathbf{y})} &\leq &\ C_3 \delta(\mathbf{y})^{-k/2}\ ,\ \text{and}\\
\left|\langle\omega_i^\xi(\eta_1),\eta_2\rangle_a\right| &\leq& &C_2 e^{2k\cdot r_a(\mathbf{y})} &\leq &\ C_3 \delta(\mathbf{y})^{-k}\ .
\end{aligned}
\end{equation*}Hence, for each $i$, and for all $\eta_1^m,\eta_2^m\in\{e_t^m,e_\theta^m\}$,
\begin{equation*}
\begin{aligned}
\left|\langle\omega_i^\xi(e_r^m),e_r^m\rangle_m\right| &\leq C_4\ ,\\
\left|\langle\omega_i^\xi(\eta_1),e_r^m\rangle_m\right| &\leq C_4 \delta(\mathbf{y})^{(1-k)/2}\ ,\\
\left|\langle\omega_i^\xi(e_r^m),\eta_1\rangle_m\right| &\leq C_4 \delta(\mathbf{y})^{-(1+k)/2}\ ,\ \text{and}\\
\left|\langle\omega_i^\xi(\eta_1),\eta_2\rangle_m\right| &\leq C_4 \delta(\mathbf{y})^{-k}\ .
\end{aligned}
\end{equation*}
The result now follows upon observing that the Wick rotation of the Minkowski metric about $e_t^m$ is uniformly equivalent to the Euclidian metric.
\end{proof}

\subsection{Continuity of automorphic vector fields}
We now work in pseudospherical polar coordinates of $\AdS$ as in Section \ref{sss:KleinianModelDefinitions}. For any vector field $\xi$, we denote respectively by $\xi^r$ and $\xi^\pa$ its \emph{radial} and \emph{pseudoangular} components, that is, its component in the radial direction and its component orthogonal to the radial direction. The remainder of this section is devoted to proving continuous extendibility of automorphic vector fields.
\begin{lemma}
\label{lemma:ContinuousExtension}
Every automorphic vector field $\xi$ over $M$ extends continuously to a vector field over $M\cup\Lambda$. Furthermore, its restriction to $\Lambda$ is an automorphic vector field tangent to $\partial\Omega=\Ein$.
\end{lemma}

Let $\nabla^a$ denote the $\AdS$ covariant derivative and $\nabla$ denote the Euclidean covariant derivative, which we note coincides with the Minkowski covariant derivative. Since $\nabla^a$ and $\nabla$ are affine equivalent,
\begin{equation}
\label{eqn:AffineConnectionForm}
\nabla^a_X Y = \nabla_X Y + \alpha(X)Y + \alpha(Y)X\ ,
\end{equation}
for some $1$-form $\alpha$.

\begin{lemma}
\label{lemma:ExplicitAffineConnectionForm}
\begin{equation*}
\label{eqn:ExplicitAffineConnectionForm}
\alpha = \frac{r}{(1-r^2)}dr\ .
\end{equation*}
\end{lemma}
\begin{proof}By symmetry, $\alpha=\phi dr$ for some $\phi$. Next observe that, for any unit, spacelike vector $\mathbf{y}$,
\begin{equation*}
\gamma := \tanh(t)\mathbf{y}
\end{equation*}
is a unit speed $\AdS$ geodesic. Denote $r:=\tanh(t)$ and differentiating yields
\begin{equation*}
\partial_t\gamma = \frac{1}{\cosh^2(t)}\partial_r = (1-r^2)\partial_r\ .
\end{equation*}
Hence
\begin{equation*}
\nabla^a_{\partial_r}\partial_r = \nabla^a_{\partial_r}(1-r^2)^{-1}\partial_t\gamma = \frac{2r}{(1-r^2)^2}\partial_t\gamma = \frac{2r}{(1-r^2)}\partial_r\ ,
\end{equation*}
so that
\begin{equation*}
2\alpha(\partial_r)\partial_r = \nabla^a_{\partial_r}\partial_r - \nabla_{\partial_r}\partial_r = \frac{2r}{(1-r^2)}\partial_r\ ,
\end{equation*}
and the result follows.
\end{proof}

Let $\xi$ be an automorphic vector field, denote $\eta :=-(1/2)\nabla^a\times^a\xi$, and let $\omega^\xi:=(\omega_1^\xi,\omega_2^\xi)$ be as in \eqref{eqn:defn:LieAlgebraValuedOneForm}.

\begin{lemma}
\label{lemma:CovariantDerivativeOfXi}
For all $X$,
\begin{equation}
\label{eqn:CovariantDerivativeOfXi}
\begin{aligned}
\nabla_X\xi &= \omega_1^\xi(X) + X\times^a\eta - \alpha(X)\xi - \alpha(\xi)X\ ,\ \text{and}\\
\nabla_X\eta &= \omega_2^\xi(X) + X\times^a\xi - \alpha(X)\eta - \alpha(\eta)X\ .
\end{aligned}
\end{equation}
\end{lemma}
\begin{proof} Indeed, by \eqref{eqn:FormulaForConnectionFormII},
\begin{equation*}
\nabla_X^a(\xi,\eta) = D_x(\xi,\eta) + (X\times^a\eta,X\times^a\xi)\ .
\end{equation*}
Thus, bearing in mind \eqref{eqn:AffineConnectionForm},
\begin{equation*}
\begin{aligned}
\nabla_X\xi
&= \nabla_X^a\xi - \alpha(X)\xi - \alpha(\xi)X\\
&= D_X\xi + X\times^a\eta - \alpha(X)\xi - \alpha(\xi)X\\
&= \omega_1^\xi(X) + X\times^a \eta - \alpha(X)\xi - \alpha(\xi)X\ ,
\end{aligned}
\end{equation*}
and the first identity follows. The second identity is proven in a similar manner, and this completes the proof.
\end{proof}

We now address independently the radial and pseudoangular components $\xi^r$ and $\xi^\pa$ of $\xi$ defined above.
\begin{lemma}
\label{lemma:EstimateOfRadialComponent}
There exists $C>0$ such that
\begin{equation}
\label{eqn:EstimateOfRadialComponent}
\|\xi^r\|_e,\|\eta^r\|_e \leq C\delta(\mathbf{y})\big(1+\left|\ln\big(\delta(\mathbf{y})\big)\right|\big)\ .
\end{equation}
\end{lemma}
\begin{remark}In particular, this implies that the radial component of $\xi(\mathbf{y})$ converges to zero as $\mathbf{y}$ approaches $\Lambda$. We will see presently that, together with our asymptotic study of the behaviour of the normal vector field, this also implies that the Euclidian orthogonal projection of its Pogorelov transform also converges to zero.
\end{remark}
\begin{proof} Indeed, bearing in mind \eqref{eqn:CovariantDerivativeOfXi},
\begin{equation*}
\begin{aligned}
\partial_r\xi^r &= \langle\nabla_{\partial_r}\xi,\partial_r\rangle_m\\
&=\langle\omega_1^\xi(\partial_r),\partial_r\rangle_m - 2\alpha(\partial_r)\langle\xi,\partial_r\rangle_m\\
&=\langle\omega_1^\xi(\partial_r),\partial_r\rangle_m - \frac{2r}{(1-r^2)}\xi^r\ .
\end{aligned}
\end{equation*}
Thus, bearing in mind \eqref{eqn:defn:EstimatesOfOmega},
\begin{equation*}
\left|\partial_r\bigg(\frac{1}{(1-r^2)}\xi^r\bigg)\right| = \frac{1}{(1-r^2)}\left|\langle\omega_1^\xi(\partial_r),\partial_r\rangle_m\right| \leq \frac{C}{(1-r^2)}\ .
\end{equation*}
Integrating over radial lines yields
\begin{equation*}
\frac{1}{(1-r^2)}\left|\xi^r\right| \leq C\big(1+\left|\ln(1-r)\right|\big)\ ,
\end{equation*}
so that
\begin{equation*}
\left|\xi^r\right| \leq C\delta\big(1+\left|\ln(\delta)\right|\big)\ ,
\end{equation*}
as desired. The estimate for $\eta^r$ is addressed in a similar manner, and this completes the proof.
\end{proof}

\begin{lemma}
\label{lemma:UniformConvergence}
For all $k>\kappa_\rho$, there exists $C>0$ such that, if $\mathbf{y}$ and $\mathbf{y}'$ are colinear radial vectors, then
\begin{equation*}
\|\xi^\pa(\mathbf{y}) - \xi^\pa(\mathbf{y}')\|_e \leq C\max\big(\delta(\mathbf{y})^{\frac{(1-k)}{2}},\delta(\mathbf{y}')^{\frac{(1-k)}{2}}\big)\ .
\end{equation*}
\end{lemma}
\begin{proof} Let $\pi$ denote the orthogonal projection along $\partial_r$, and observe that
\begin{equation*}
\nabla_{\partial_r}\pi = 0\ .
\end{equation*}
In particular,
\begin{equation*}
\nabla_{\partial_r}\xi^\pa = \nabla_{\partial_r}\pi(\xi) = \pi\big(\nabla_{\partial_r}\xi\big)\ ,
\end{equation*}
and
\begin{equation*}
\nabla_{\partial_r}\eta^\pa = \nabla_{\partial_r}\pi(\eta) = \pi\big(\nabla_{\partial_r}\eta\big)\ .
\end{equation*}
Thus, by \eqref{eqn:CovariantDerivativeOfXi},
\begin{equation*}
\nabla_{\partial_r}\xi^\pa = \omega_1^\xi(\partial_r)^\pa + \partial_r\times^a\eta - \alpha(\partial_r)\xi\ ,
\end{equation*}
and
\begin{equation*}
\nabla_{\partial_r}\eta^\pa = \omega_2^\xi(\partial_r)^\pa + \partial_r\times^a\xi - \alpha(\partial_r)\eta\ .
\end{equation*}
Note now that, by \eqref{eqn:defn:MinkowskiOrthonormalFrame} and \eqref{eqn:defn:AntiDeSitterOrthonormalFrame},
\begin{equation*}
\partial_r\times^a = e_r^m\times^a = \frac{1}{(1-r^2)}e_r^a\times^a = \frac{1}{(1-r^2)}e_r^m\times^m\ ,
\end{equation*}
so that
\begin{equation}
\label{eqn:DerivativeOfXiEta}
\begin{aligned}
\nabla_{\partial_r}\big(\xi^\pa,\eta^\pa)
&=\big(\omega_1^\xi(\partial_r)^\pa,\omega_2^\xi(\partial_r)^\pa\big)\\
&\qquad\qquad + \frac{1}{(1-r^2)}\big(e_r^m\times^m\eta^\pa,e_r^m\times^m\xi^\pa\big)
 - \frac{r}{(1-r^2)}\big(\xi^\pa,\eta^\pa\big)\ .
\end{aligned}
\end{equation}

Consider now the vector field $\nu := \xi^\pa - e_r^m\times^m\eta^\pa$. The preceding identity yields
\begin{equation*}
\nabla_{\partial_r}\nu = \big(\omega_1^\xi(\partial_r)^\pa - e_r^m\times^m\omega_2^\xi(\partial_r)^\pa\big) - \frac{1}{(1-r)}\nu\ .
\end{equation*}
Hence
\begin{equation*}
\nabla_{\partial_r}\frac{1}{(1-r)}\nu = \frac{1}{(1-r)}\big(\omega_1^\xi(\partial_r)^\pa - e_r^m\times^m\omega_2^\xi(\partial_r)^\pa\big)\ ,
\end{equation*}
so that, by \eqref{eqn:defn:EstimatesOfOmega},
\begin{equation*}
\big\|\nabla_{\partial_r}\frac{1}{(1-r)}\nu\big\|_e \leq C_1(1-r)^{-\frac{(3+k)}{2}}\ .
\end{equation*}
Integrating over the radial line passing through $\mathbf{y}$ and $\mathbf{y}'$ yields
\begin{equation*}
\|\frac{1}{(1-r)}\nu\| \leq C_1\big(1 + (1-r)^{-\frac{(1+k)}{2}}\big)\ ,
\end{equation*}
so that
\begin{equation*}
\label{eqn:EstimateForNu}
\|\nu\| \leq C_1(1-r)^{\frac{(1-k)}{2}}\ .
\end{equation*}

Substiting $\nu$ into the first component of \eqref{eqn:DerivativeOfXiEta} now yields
\begin{equation*}
\nabla_{\partial_r}\xi^\pa = \omega_1^\xi(\partial_r)^\pa - \frac{1}{(1-r^2)}\nu + \frac{1}{(1+r)}\xi^\pa\ ,
\end{equation*}
so that, by \eqref{eqn:defn:EstimatesOfOmega} again
\begin{equation*}
\big\|\nabla_{\partial_r}\frac{1}{(1+r)}\xi^\pa\big\|_e \leq C_2 (1-r)^{-\frac{(1-k)}{2}}\ ,
\end{equation*}
and the result follows upon integrating this identity along the radial segment joining $\mathbf{y}$ and $\mathbf{y}'$.
\end{proof}

We are now ready to prove Lemma \ref{lemma:ContinuousExtension}.
\begin{proof}[Proof of Lemma \ref{lemma:ContinuousExtension}.] Let $\xi^r$ and $\xi^\pa$ denote respectively the radial and pseudoangular components of $\xi$. By Lemma \ref{lemma:EstimateOfRadialComponent}, $\xi^r$ extends continuously by $0$ over $\Lambda$. Likewise, since the uniform limit of continuous functions is continuous, it follows by Lemma \ref{lemma:UniformConvergence} that $\xi^\pa$ also extends continuously to $\Lambda$. The vector field $\xi$ thus extends continuously to $\Lambda$. The restriction of $\xi$ to $\Lambda$ has vanishing radial component and is thus tangent to $\partial\Omega=\Ein$. By continuity, the restriction of $\xi$ to $\Lambda$ is automorphic, and this completes this proof.
\end{proof} 
\section{Rigidity}
\label{sec:rigidity}
We now prove the main results of this paper, namely infinitesimal and global rigidity. The main novelty is the infinitesimal rigidity result of Theorem \ref{thm:InfinitesimalRigidity}. Following the overall strategy of \cite{Sch}, we use the Pogorelov transform to reduce the problem to one of infinitesimal rigidity for convex subsets in Euclidean space.

We recall that the key challenge lies in establishing sufficient regularity of the first order variation field $\xi$ in order for the technique developed by Vekua in \cite{Vek} to be applied. The main purpose of Section \ref{s:AutomorphicVectorFields} was to establish continuity. However, this, in itself, is insufficient since, whilst Vekua's techniques require $W^{1,\infty}$-regularity, it seems unlikely that we can hope for better than $C^{0,1-k}$-regularity, for $k>\kappa_\rho$, with $\kappa_\rho$ as in \eqref{eqn:DefinitionOfKappa}.

This problem is bypassed in Section \ref{ss:InfinitesimalRigidity} with two key observations. First, we show (c.f. Lemma \ref{lemma:LocalGuageEquivalence}) how Vekua's argument reduces to an application of the Poincar\'e--Hopf theorem and, second, using Lemma \ref{lemma:NowhereVanishingProperty}, we show that the restriction of $\xi$ to the boundary curve $\Lambda$ is generically non-vanishing, and thus has zero contribution to the Poincar\'e--Hopf degree. Since it is only along $\Lambda$ that regularity poses a problem, infinitesimal rigidity then follows upon applying these properties in the correct order.

\subsection{Pogorelov transforms}
\label{sss:PogorelovTransforms}
We will use Pogorelov transforms in the Kleinian model to map first order variations of immersed surfaces which are isometric with respect to the anti-de Sitter metric to first order variations which are isometric with respect to the Euclidean metric. In what follows, we will make considerable use of the framework described in Section \ref{sss:KleinianModelDefinitions}.

Consider first a vector field $\xi$ over $\Omega$. Over $\Omega^s$ (resp $\Omega^t$) we denote by
\begin{equation*}
\xi := \xi^r + \xi^\pa\qquad\big(\mathrm{resp.}\ \xi:=\xi^r + \xi^{\mathrm{hyp}}\big)
\end{equation*}
its decomposition into radial and pseudoangular (resp. hyperbolic) components, and we define the \emph{first Pogorelov transform} of $\xi$ by
\begin{equation*}
P_1\xi := \lambda \xi^r + \xi^\pa\qquad\big(\mathrm{resp.}\ P_1\xi:=\lambda\xi^r + \xi^{\mathrm{hyp}}\big)\ ,
\end{equation*}
where $\lambda:\Omega\rightarrow(0,\infty)$ is given by
\begin{equation*}
\lambda(\mathbf{y}) := \frac{1}{1 + y_0^2 - y_1^2 - y_2^2}\ ,
\end{equation*}
Note that, over $\Omega^s$,
\begin{equation*}
\lambda(\mathbf{y}) = \frac{1}{1-r^2}\ ,
\end{equation*}
and, over $\Omega^t$\ ,
\begin{equation*}
\lambda(\mathbf{y}) = \frac{1}{1+r^2}\ .
\end{equation*}
The difference in sign is explained by the fact that in pseudospherical polar coordinates, the radial direction is positive, whilst in hyperbolic polar coordinates, it is negative. Over $\Omega^s$ (resp. $\Omega^t$), the first Pogorelov transform is geometrically characterized by the property that, for all $\xi$,
\begin{equation*}
(P_1\xi)^\pa = \xi^\pa\qquad\big(\text{resp.}\ (P_1\xi)^{\mathrm{hyp}}=\xi^{\mathrm{hyp}}\big)\qquad\text{and}\qquad
\|(P_1\xi)^r\|_m^2 = \|\xi^r\|_a^2\ .
\end{equation*}
By the manner in which we have defined it, it is not immediately clear that $P_1$ extends smoothly over the whole of $\Omega$. This is easily addressed.
\begin{lemma}
For every smooth vector field $\xi$ over $\AdS$, $P_1\xi$ is also smooth.
\end{lemma}
\begin{proof}It suffices to prove that $P_1\xi$ extends smoothly over the light cone. By rotational symmetry about the $y_0$ axis, it also suffices to address the analogous result over $\mathrm{AdS}^{1,1}$. Let $\partial_r$ and $\partial_t$ denote respectively the radial and pseudoangular vector fields of pseudospherical coordinates in $\Omega^s$. By \eqref{eqn:defn:PseudoSphericalCoordinates},
\begin{equation*}
\partial_r = \frac{y_0}{r}\partial_{y_0} + \frac{y_1}{r}\partial_{y_1}\qquad\text{and}\qquad
\partial_t = y_1\partial_{y_0} + y_0\partial_{y_1}\ .
\end{equation*}
In terms of the basis $(\partial_{y_0},\partial_{y_1})$, the first Pogorelov transform is given at every point $(y_0,y_1)$ of $\Omega^s$ by the matrix
\begin{equation*}
P_1 = \frac{-1}{r}
\begin{pmatrix}y_0/r\hfill& y_1\hfill\\ y_1/r\hfill& y_0\hfill\end{pmatrix}
\cdot\begin{pmatrix}1/(1-r^2)\hfill& 0\hfill\\ 0\hfill& 1\hfill\end{pmatrix}
\cdot\begin{pmatrix}y_0\hfill& -y_1\hfill\\ -y_1/r\hfill& y_0/r\hfill\end{pmatrix}\ .
\end{equation*}
Since $r^2=y_1^2-y_0^2$,
\begin{equation*}
P_1 = \frac{1}{1+y_0^2-y_1^2}\begin{pmatrix}1-y_1^2\hfill& y_0y_1\hfill\\ -y_0y_1\hfill& 1+ y_0^2\hfill\end{pmatrix}\ .
\end{equation*}
We likewise obtain the same formula over $\Omega^t$, and smoothness of $P_1\xi$ follows.
\end{proof}

The main property of the first Pogorelov transform is the following result, which is proven exactly as in \cite[Lemma 1.10]{Sch}.
\begin{lemma}
Let $\Sigma\subseteq\Omega$ be an immersed surface. A smooth vector field $\xi$ is a first order isometric variation of $\Sigma$ with respect to $\langle\cdot,\cdot\rangle_a$ if and only if $P_1\xi$ is a first order isometric variation of $\Sigma$ with respect to $\langle\cdot,\cdot\rangle_m$.
\end{lemma}

The \emph{second Pogorelov transform} $P_2$ is more straightforward, and is given by multiplication by the constant matrix
\begin{equation*}
P_2 := \begin{pmatrix}-1\hfill& & \\ &1\hfill& \\ & & 1\hfill\end{pmatrix}\ .
\end{equation*}
The main property of the second Pogorelov transform is the following result, which is likewise proven exactly as in \cite[Lemma 1.10]{Sch}.
\begin{lemma}
Let $\Sigma\subseteq\Omega$ be an immersed surface. A smooth vector field $\xi$ is a first order isometric variation of $\Sigma$ with respect to $\langle\cdot,\cdot\rangle_m$ if and only if $P_2\xi$ is a first order isometric variation of $\Sigma$ with respect to $\langle\cdot,\cdot\rangle_e$.
\end{lemma}

\subsection{Infinitesimal rigidity}
\label{ss:InfinitesimalRigidity}
We now prove infinitesimal rigidity.
\begin{theorem}
\label{thm:InfinitesimalRigidity}
If $\xi:\partial M\rightarrow\opT\AdS$ is a $(\rho,\alpha)$-automorphic first order isometric variation, then $\xi$ is the restriction of a Killing vector field of $\AdS$.
\end{theorem}

The proof of Theorem \ref{thm:InfinitesimalRigidity} will require three more technical results. The first is a local, topological reformulation of Vekua's analysis of Euclidian first order isometric variation fields. Recall first that, if $U\subseteq\mathbf{R}^2$ is a neighbourhood of the origin, and if $\xi:U\rightarrow TU$ is a continuous vector field with an isolated zero at this point, then we define its \emph{winding number} at this point to be the number of times the image of a small oriented circle about this point passes around the origin. This quantity is unchanged both by orientation-preserving homeomorphisms of the domain and by continuous orientation-preserving gauge transformations of the codomain, and is thus also well-defined for isolated zeroes of continuous tangent vector fields over $\partial M$.
\begin{lemma}
\label{lemma:LocalGuageEquivalence}
Let $\Sigma\subseteq\mathbb{R}^3$ be a smooth, strongly convex immersed surface, and let $\pi^e:\mathbb{R}^3\rightarrow\opT\Sigma$ denote the Euclidian orthogonal projection. If $\xi:\Sigma\rightarrow\mathbb{R}^3$ is a Euclidian first order isometric variation of $\Sigma$, then $\pi^e\circ\xi$ is everywhere locally gauge equivalent to a holomorphic function.

That is, for all $p\in\Sigma$, there exists a neighbourhood $\Omega$ of $p$, a quasiconformal homeomorphism $\phi:\mathbb{D}\rightarrow\Omega$, and a continuous orientation preserving trivialization $\psi:\opT\Omega\rightarrow\mathbb{C}$ such that the composition $\big(\psi\circ\pi^e\circ\xi\circ\phi\big)$ is holomorphic.

In particular, if $\pi^e\circ\xi$ is non-vanishing, then every zero is isolated and has positive integer-valued winding number
\end{lemma}
\begin{proof} We use the notation of \cite[Section V.2]{Vek}, and it will thus be preferable to denote the variation vector field provisionally by $X=:(\xi,\eta,\zeta)$. Upon rotating if necessary, we may suppose that, near $p$, $\Sigma$ is the graph of a function $z(x,y)$. We define the tangent frame $(e_x,e_y)$ near $p$ by
\begin{equation*}
e_x:=\partial_x + z_x\partial_z\qquad\text{and}\qquad e_y:=\partial_y+z_y\partial_z\ ,
\end{equation*}
and we denote by $(f_x,f_y)$ its dual frame with respect to the Euclidian metric. If $\pi(X)$ denotes the Euclidian orthogonal projection onto $T\Sigma$, then
\begin{equation*}
\pi(X) = u\cdot f_x + v\cdot f_y\ ,
\end{equation*}
where
\begin{equation*}
u := \xi + z_x\zeta\qquad\text{and}\qquad v:=\eta+z_y\zeta\ .
\end{equation*}
By $(2.9)$ of \cite[Section V.2]{Vek}, the function $w:=u+iv$ satisfies an equation of Beltrami type. By \cite[Theorem $3.30$]{Vek}, there therefore exists a quasiconformal local homeomorphism $W$ and a smooth holomorphic function $f$, such that, near $p$,
\begin{equation*}
w(x+iy) = f\big(W(x+iy)\big)\ ,
\end{equation*}
and the result follows.
\end{proof}

Let $N^m:\partial M\rightarrow\mathbb{R}^{2,1}$ denote the outward pointing Minkowski unit normal over $\partial M$. Note that this is not necessarily well-defined at points where $\opT\partial M$ is null. However, the following estimate shows in particular that it is well-defined in a neighbourhood of $\Lambda$.
\begin{lemma}
\label{lemma:AsymptoticsOfNormal}
For all $k>\kappa_\rho$, over $\Omega$,
\begin{equation}
\label{eqn:AsymptoticsOfNormal}
N^m = \big(1 + \opO(\delta)\big)e^m_r + \opO\big(\delta^{\frac{1-k}{2}}\big) e^m_\theta + \opO\big(\delta^{\frac{1-k}{2}}\big) e^m_t\ ,
\end{equation}
where $\delta$ is as in \eqref{eqn:DefinitioOfDelta}. In particular $N^m$ extends continuously to $\Lambda$.
\end{lemma}
\begin{remark} It follows that the boundary of $\overline{M}:=M\cup\Lambda$, viewed as a convex subset of in $\mathbb{R}^{2,1}$, has a unique supporting tangent at every point, including along $\Lambda$. Since $M$ is convex, the tangent bundle of $\partial M$ is thus continuous and homeomorphic to $\opT\mathbb{S}^2$. In particular, it has a well-defined Poincar\'e--Hopf degree equal to $2$. Theorem \ref{thm:InfinitesimalRigidity} will follow from this fact together with the topological properties of first order isometric variations established in Lemma \ref{lemma:LocalGuageEquivalence}.
\end{remark}
\begin{proof}
Indeed, let $N^a$ denote the anti de Sitter unit normal vector field over $\partial M$. For $p\in\Sigma\cap\Omega$, denote
\begin{equation*}
N^a(p) = \alpha e^a_r + \beta e^a_\theta + \gamma e^a_t\ ,
\end{equation*}
and note that
\begin{equation*}
\label{eqn:EstimateOfNormalProofI}
\alpha^2 + \beta^2 - \gamma^2  = -1\ .
\end{equation*}
By \eqref{eqn:BoundsOfRadialComponentI}, we may suppose that there exists $C>0$ such that
\begin{equation}
\label{eqn:EstimateOfNormalProofII}
\frac{1}{C} \leq \alpha \leq C\ .
\end{equation}
By \eqref{eqn:GrowthRateOfTAngle}, we may suppose furthermore that
\begin{equation}
\label{eqn:EstimateOfNormalProofIII}
\left|\gamma\right| \leq C\delta^{-\frac{k}{2}}\ ,
\end{equation}
and hence, by \eqref{eqn:EstimateOfNormalProofI}
\begin{equation}
\label{eqn:EstimateOfNormalProofIV}
\left|\beta\right| \leq C\delta^{-\frac{k}{2}}\ .
\end{equation}

Define $\hat{N}^m$ by
\begin{equation*}
\hat{N}^m := \frac{\alpha}{\delta} e^m_r + \frac{\beta}{\sqrt{\delta}}e^m_\theta + \frac{\gamma}{\sqrt{\delta}}e^m_t\ .
\end{equation*}
By \eqref{eqn:defn:MinkowskiOrthonormalFrame} and \eqref{eqn:defn:AntiDeSitterOrthonormalFrame}, $\hat{N}_m$ is a (non-unit) Minkowski normal of $\partial M$ at $p$. Using \eqref{eqn:EstimateOfNormalProofI}, we find that its Minkowski norm-squared satisfies
\begin{equation*}
\|\hat{N}^m\|_m^2 = \frac{\alpha^2}{\delta^2} + \frac{\beta^2}{\delta} - \frac{\gamma^2}{\delta} = \frac{\alpha^2}{\delta^2} - \frac{(\alpha^2+1)}{\delta}\ .
\end{equation*}
In particular, by \eqref{eqn:EstimateOfNormalProofII}, we may suppose that $\hat{N}^m$ is spacelike and that its length $\lambda$ satisfies
\begin{equation*}
\lambda = \frac{\alpha}{\delta}\cdot\big(1 + \opO(\delta)\big)\ .
\end{equation*}
The unit Minkowski normal of $\partial M$ at $p$ is thus
\begin{equation*}
N^m = \frac{1}{\lambda}\hat{N}^m = \big(1+\opO(\delta)\big)\cdot e^m_r + \opO\big(\delta^{\frac{1}{2}}\big)\cdot\beta\cdot e^m_\theta + \opO\big(\delta^{\frac{1}{2}}\big)\cdot\gamma\cdot e^m_t\ ,
\end{equation*}
and the result now follows by \eqref{eqn:EstimateOfNormalProofIII} and \eqref{eqn:EstimateOfNormalProofIV}.
\end{proof}

Finally, denote
\begin{equation*}
Y := \partial\Omega = \big\{ \big(y_0,y_1,y_2)\ \big|\ y_1^2+y_2^2 - y_0^2 = 1 \big\}\ .
\end{equation*}
We recall the second Pogorelov transform $P_2$, defined in Section \ref{sss:PogorelovTransforms}.
\begin{lemma}
\label{lemma:BundleIsomorphism}
Let $\pi^e:\mathbb{R}^3\rightarrow TY$ denote the Euclidian orthogonal projection. The composition $(\pi^e\circ P_2)$ defines a bundle isomorphism of $TY$.
\end{lemma}
\begin{proof}By rotational symmetry about the $y_0$-axis, it suffices to consider the restriction to the $(y_0,y_1)$-plane. Denote
\begin{equation*}
Y' := \big\{ (y_0,y_1)\ \big|\ y_1^2-y_0^2 = 1\big\}\ .
\end{equation*}
This curve is parametrized by
\begin{equation*}
\mathbf{y}(t) := \big(\sinh(t),\cosh(t)\big)\ .
\end{equation*}
The tangent vector of this parametrization at $t$ is
\begin{equation*}
T(t) := \big(\cosh(t),\sinh(t)\big)\ ,
\end{equation*}
and, trivially,
\begin{equation*}
P_2T(t) = \big(-\cosh(t),\sinh(t)\big)\ .
\end{equation*}
Hence,
\begin{equation*}
(\pi^e\circ P_2)\big(T(t)\big) = \frac{1}{\|T(t)\|_e^2}\langle T(t),P_2T(t)\rangle_e\cdot T(t) = -\frac{T(t)}{\cosh(2t)}\neq 0\ .
\end{equation*}
It follows that $(\pi^e\circ P_2)$ restricts to a bundle isomorphism over $T\partial\Omega$, as desired.
\end{proof}

With these prerequisites in hand, we are now ready to prove Theorem \ref{thm:InfinitesimalRigidity}.
\begin{proof}[Proof of Theorem \ref{thm:InfinitesimalRigidity}.] We first show that $\tilde{\xi}:=\pi^e(P_2P_1\xi)$ extends continuously to $\Lambda$. Indeed, over $\Omega$, we denote respectively by $\xi^r$ and $\xi^\pa$ the radial and pseudoangular components of $\xi$. By Lemma \ref{lemma:ContinuousExtension}, $\xi$ extends continuously to $\Lambda$, and its restriction to this set is an automorphic vector field tangent to $\Ein$. By Lemma \ref{lemma:UniformConvergence}, $\xi^\pa$ also extends continuously to $\Lambda$, and its restriction to this set coincides with that of $\xi$ and hence is also an automorphic vector field tangent to $\Ein$.

By Lemma \ref{lemma:EstimateOfRadialComponent}, and by definition of the first Pogorolov transform $P_1$, there exists $C_1>0$ such that, near $\Lambda$,
\begin{equation*}
\|P_1\xi^r\|_e \leq C_1\big(1+\left|\ln(\delta(\mathbf{y}))\right|\big)\ .
\end{equation*}
Note now that, for any Minkowski timelike immersed surface, the second Pogorelov transform of its Minkowski normal is colinear with its Euclidean normal. Thus, by \eqref{eqn:AsymptoticsOfNormal}, for all $k>\kappa_\rho$, there exists $C_2>0$ such that, near $\Lambda$,
\begin{equation*}
\|\pi^e(P_2P_1\xi^r)\|_e \leq C_2\delta(\mathbf{y})^{\frac{1-k}{2}}\cdot\big(1+\left|\ln(\delta(\mathbf{y}))\right|\big)\ .
\end{equation*}
Since this converges to $0$ near $\Lambda$, the vector field
\begin{equation*}
\tilde{\xi} = \pi^e\big(P_2P_1\xi^r\big) + \pi^e\big(P_2P_1\xi^\pa\big)
\end{equation*}
extends continuously to $\Lambda$, and its restriction to this set coincides with that of $\pi^e(P_2P_1\xi^\pa)$.

Now let $\Sigma$ and $\Sigma'$ denote the two boundary components of $M$. We claim that, over $\Sigma$, $\xi$ is the restriction of some Killing vector field of $\AdS$. Indeed, suppose the contrary. We will first show that, upon adding a Killing vector field to $\xi$ if necessary, we may suppose that $\tilde{\xi}$ vanishes at least $3$ distinct points of $\Sigma$. Indeed, let $E$ denote the $6$-dimensional space of Killing vector fields of $\AdS$. For all $p\in\partial \overline{M}$, define $\Phi_p:E\rightarrow T_p\partial\overline{M}$ by
\begin{equation*}
\Phi_p(\eta) := \pi^e\big(P_2P_1\eta(p)\big)\ .
\end{equation*}
By Lemmas \ref{lemma:VanishingConditionOfKillingFieldsII} and \ref{lemma:BundleIsomorphism}, for any $3$ distinct points $p_1,p_2,p_3\in\Lambda$, the linear map $\Phi:=\big(\Phi_{p_1},\Phi_{p_2},\Phi_{p_3}\big)$ has trivial kernel, and is therefore a linear isomorphism. Since $T\partial\overline{M}$ is continuous, we may perturb $(p_1,p_2,p_3)$ to $3$ distinct points of $\Sigma$. Upon replacing $\xi$ with $\xi-\eta$ if necessary, for some Killing vector field $\eta$, we may suppose that $\tilde{\xi}$ vanishes at $p_1$, $p_2$ and $p_3$, as desired.

We now show that these zeroes are stable under small perturbations. To this end, we show that $\tilde{\xi}$ does not vanish identically over $\Sigma$. Note first that, if $\tilde{\xi}$ vanishes identically over $\Sigma$, then so too does $P_2P_1\xi$. Indeed, for any first order isometric variation field $\mu$, with tangential component $\mu^{\mathrm{t}}$ and normal component $\mu^{\mathrm{n}}$,
\begin{equation*}
2\langle\mu^{\mathrm{n}},N\rangle\mathrm{II} = -\mathcal{L}_{\mu^{\mathrm{t}}}g\ ,
\end{equation*}
where $N$, $\mathrm{II}$, and $g$ denote respectively the unit normal vector field, the second fundamental form, and the induced metric. Since $\Sigma$ is locally strongly convex, $\mathrm{II}$ is nowhere vanishing, and it follows that $\mu^{\mathrm{n}}$, and therefore also $\mu$, vanishes whenever $\mu^{\mathrm{t}}$ does. Finally, since $\xi$ is not the restriction of an anti-de Sitter Killing vector field, $P_2P_1\xi$ is not the restriction of a Euclidean Killing vector field, and, in particular, does not vanish identically over $\Sigma$. It follows that $\tilde{\xi}$ does not vanish identically over $\Sigma$ either. Thus, by Lemma \ref{lemma:LocalGuageEquivalence}, the zeroes of $\tilde{\xi}$ over $\Sigma$ are stable under small perturbations, as desired.

By Lemma \ref{lemma:NowhereVanishingProperty}, upon perturbing by a Killing field if necessary, we may suppose that the continuous extension of $P_1\xi^\pa$ is \emph{nowhere vanishing} over $\Lambda$. By Lemma \ref{lemma:BundleIsomorphism}, we may suppose that the same holds for the continuous extension of $\pi^e(P_2P_1\xi^\pa)$, and therefore also for $\tilde{\xi}$. By the preceding discussion, upon making this perturbation sufficiently small, we may still suppose that $\tilde{\xi}$ vanishes at $3$ distinct points close to $p_1$, $p_2$, and $p_3$.

By Lemma \ref{lemma:LocalGuageEquivalence}, the set of zeroes of $\tilde{\xi}$ over $\Sigma$ is discrete and every zero has positive, integer winding number. Since $\tilde{\xi}$ does not vanish over $\Lambda$, it does not vanish identically over $\Sigma'$ and, by Lemma \ref{lemma:LocalGuageEquivalence}, again, the set of zeros of $\tilde{\xi}$ over $\Sigma'$ is also discrete, and every zero has positive, integer winding number. Since $\opT\partial\overline{M}$ is homeomorphic to $\opT\mathbb{S}^2$, it has total degree $2$, and $\tilde{\xi}$ therefore has at most $2$ distinct zeroes over $\Sigma\cup\Sigma'$. This is absurd, and it follows that over $\Sigma$, $\xi$ is the restriction of some Killing vector field $\eta$, say, of $\AdS$, as asserted.

By the same argument, over $\Sigma'$, $\xi$ is the restriction of some Killing vector field $\eta'$, say, of $\AdS$. Since $\eta$ and $\eta'$ coincide over $\Lambda$, they are equal. The vector field $\xi$ is thus the restriction to $\partial M$ of a Killing vector field of $\AdS$, and this completes the proof.
\end{proof}

\subsection{Global rigidity}
\label{ss:GlobalRigidity}
The rest of the proof of Theorem \ref{rigiditytheorem} now follows by a fairly standard application of the classical continuity method. We first describe the framework. Let $\Pi$ be the fundamental group of a compact surface. Let $\hat{\teich}:=\hat{\teich}_{\mathrm{rep}}(\Pi)$ denote the space of Fuchsian representations of $\Pi$ in $\PSL$, and let $\teich:=\teich_{\mathrm{rep}}(\Pi)$ denote the Teichm\"uller space of conjugacy classes of elements of $\hat{\teich}$. Let $\mathbb{D}$ denote the Poincar\'e disk, and let $h$ denote its hyperbolic metric.

Given a metric $g$ on $\mathbb{D}$ with Levi--Civita covariant derivative $\nabla$, a \emph{Codazzi field} of $g$ is defined to be a symmetric endomorphism field $A$ such that, for all $X$, $Y$,
\begin{equation*}
(\nabla_X A)Y = (\nabla_Y A)X\ .
\end{equation*}
Given a smooth, positive function $\phi$, we define a $\phi$-\emph{Labourie field} of $g$ to be a Codazzi field $A$ such that
\begin{equation*}
\det(A) = \phi\ .
\end{equation*}
Labourie fields for constant $\phi$ are discussed, for example, in \cite[Section 2.6]{FS4}. In particular, for any Fuchsian representation $\sigma:\Pi\rightarrow\PSL$, the space of $\sigma$-invariant $\phi$-Labourie fields of $g$ is a smooth manifold of real dimension $(6g-6)$, where $g$ denotes the genus of $\Pi$.

For any smooth function $u:\mathbb{D}\rightarrow\mathbb{R}$, we define the conformal metric $g[u]$ by
\begin{equation*}
g[u] := e^{2u}h\ .
\end{equation*}
Its curvature is given by
\begin{equation*}
\kappa[u] = -e^{-2u}\big(\nabla^h u + 1\big)\ .
\end{equation*}
We recall (see, for example, \cite{AvilesMcOwen1985}) that, for all $\phi\in C^{0,\alpha}(\mathbb{D})$ such that $\phi\geq\epsilon$ for some $\epsilon>0$, there exists a unique $u\in C^{2,\alpha}(\mathbb{D})$ such that $\kappa[u]=-\phi$. That is, for all such $\phi$, there exists a unique, complete conformal metric $g[u]$ with curvature $-\phi$.

By the fundamental theorem of surface theory (see, for example, \cite[Section 7]{BarGauduchonMoroianu2005}), if $\kappa[u]=-\phi<-1$, and if $A$ is a $(\phi-1)$-Labourie field of $g[u]$, then there exists a strongly future convex smooth immersion $e^+[u,A]:\mathbb{D}\rightarrow\AdS$ with induced metric $g[u]$ and shape operator $A$. Furthermore $e^+[u,A]$ is unique up to post-composition by elements of $\PSL\times\PSL$. In particular, if $u$ and $A$ are $\sigma$-invariant with respect to some Fuchsian representation $\sigma:\Pi\rightarrow\PSL$, then $e^+[u,A]$ is $(\sigma,\hat{\rho}^+)$-equivariant with respect to some quasi-Fuchsian representation $\hat{\rho}^+:=\hat{\rho}^+[u,A]:\Pi\rightarrow\PSL\times\PSL$ which we call the \emph{holonomy} of $e^+[u,A]$. The representation $\hat{\rho}^+[u,A]$ depends smoothly on both $u$ and $A$ and defines a Teichm\"uller class in $\teich\times\teich$ which we henceforth denote by $\rho^+[u,A]$. In the same way, we likewise obtain a strongly past convex smooth immersion $e^-[u,A]$, unique up to post-composition by elements of $\PSL\times\PSL$, which in turn yields a Teichm\"uller class $\rho^-[u,A]\in\teich\times\teich$.

The remainder of this section will be devoted to proving the following result.
\begin{theorem}
\label{thm:ExistenceAndUniquenessForLabourieFields}
Choose $\sigma_1,\sigma_2\in\hat{\teich}$, for each $i\in\{1,2\}$, let $\phi_i$ be a $\sigma_i$-invariant smooth function such that $\phi_i>1$, and let $u_i$ denote the unique smooth, $\sigma_i$-invariant function such that $\kappa[u_i]=-\phi_i$. There exists a unique pair $(A_1,A_2)$ such that, for each $i$, $A_i$ is a $\sigma_i$-invariant $(\phi_i-1)$-Labourie field, and
\begin{equation}
\label{eqn:EqualityOfRepresentations}
\rho^+[u_1,A_1] = \rho^-[u_2,A_2]\ .
\end{equation}
\end{theorem}

Theorem \ref{thm:ExistenceAndUniquenessForLabourieFields} will be proven using the continuity method, which we recall decomposes into a local perturbation step and a compactness step. We first prove the local perturbation result. Choose $\sigma_1,\sigma_2\in\hat{\teich}$. For each $i$, let $(\phi_{i,t})_{t\in]-\epsilon,\epsilon[}$ be a smooth family of $\sigma_i$-invariant smooth functions over $\mathbb{D}$ such that, for all $t$,
\begin{equation*}
\phi_{i,t} > 1\ .
\end{equation*}
For all $i$, and for all $t$, let $u_{i,t}$ denote the unique smooth, $\sigma_i$-invariant function such that $\kappa[u_{i,t}]=-\phi_{i,t}$.

\begin{lemma}
\label{lemma:ContinuityMethodOpenness}
Suppose that, for each $i$, there exists a $\sigma_i$-invariant $(\phi_{i,0}-1)$-Labourie field $A_i$ of $g[u_{i,0}]$ such that
\begin{equation*}
\rho^+[u_1,A_1] = \rho^-[u_2,A_2]\ .
\end{equation*}
Upon reducing $\epsilon$ if necessary, there exists a unique pair of continuous families $(A_{1,t})_{t\in]-\epsilon,\epsilon[}$ and $(A_{2,t})_{t\in]-\epsilon,\epsilon[}$ such that
\begin{enumerate}
\item $A_{1,0}=A_1$ and $A_{2,0}=A_2$;
\item for each $i$, $A_{i,t}$ is a $\sigma_i$-invariant $(\phi_{i,t}-1)$-Labourie field of $g[u_{i,t}]$; and
\item for all $t$, $\rho^+[u_{1,t},A_{1,t}]=\rho^-[u_{2,t},A_{2,t}]$.
\end{enumerate}
\end{lemma}
\begin{proof} For each $i$, let $\mathcal{L}_i$ denote the space of $\sigma_i$-invariant $(\phi_{i,0}-1)$-Labourie fields of $g[u_{i,0}]$ and recall that this is a smooth, real $(6g-6)$-dimensional manifold. Define $F_1:\mathcal{L}_1\rightarrow\teich\times\teich$ and $F_2:\mathcal{L}_2\rightarrow\teich\times\teich$ by
\begin{equation*}
F_1(A) := \rho^+[u_{1,0},A]\qquad\text{and}\qquad F_2(A) := \rho^-[u_{2,0},A]\ ,
\end{equation*}
and note that $F_1$ and $F_2$ are trivially smooth.

We claim that $(F_1,F_2)$ is transverse to the diagonal of $(\teich\times\teich)^2$ at $(A_{1,0},A_{2,0})$. Indeed, for each $i$, let $B_i$ be a tangent vector of $\mathcal{L}_i$ at $A_{i,0}$, and suppose that
\begin{equation*}
\xi:=\big(DF_1(A_{1,0})\cdot B_1,DF_2(A_{2,0})\cdot B_2\big)
\end{equation*}
is tangent to the diagonal of $(\teich\times\teich)^2$. It suffices to show that $B_1=B_2=0$. To this end, let $M$ denote the convex subset of $\AdS$ bounded by $e^+[u_{1,0},A_{1,0}]$ and $e^-[u_{2,0},A_{2,0}]$. The vector $\xi$ yields a first order isometric variation of $\partial M$ which we also denote by $\xi$. By Theorem \ref{thm:InfinitesimalRigidity}, $\xi$ is the restriction of some Killing vector field over $\AdS$. Since Killing vector fields yield trivial first order variations of the shape operator, it follows that $B_1=B_2=0$, as asserted. Transversality now follows upon counting dimensions, and the result follows upon applying the implicit function theorem.
\end{proof}

We now prove a compactness result. For each $i$, and for all $m$, let $\phi_{i,m}$ be a smooth $\sigma_i$-invariant function such that $\phi_{i,m}>1$. For each $i$, and for all $m$, let $u_{i,m}$ denote the unique smooth, $\sigma_i$-invariant function such that $\kappa[u_{i,m}]=-\phi_{i,m}$. For each $i$, and for all $m$, let $A_{i,m}$ be a $\sigma_i$-invariant $(\phi_{i,m}-1)$-Labourie field of $g[u_{i,m}]$, and suppose that
\begin{equation*}
\rho^+[u_{1,m},A_{1,m}] = \rho^-[u_{2,m},A_{2,m}] =: \rho_m\ .
\end{equation*}
Suppose that, for each $i$, $(\phi_{i,m})_{m\in\mathbb{N}}$ converges in the $C^\infty$-sense to $\phi_{i,\infty}$-say. In particular, $(u_{i,m})_{m\in\mathbb{N}}$ also converges in the $C^\infty$-sense to $u_{i,\infty}$, say.
\begin{lemma}
\label{lemma:GlobalCompactness}
With the above hypotheses, for each $i$, $(A_{i,m})_{m\in\mathbb{N}}$ subconverges to some limit $A_{i,\infty}$, say.
\end{lemma}
\begin{remark}By continuity, for each $i$, $A_{i,\infty}$ is a $\sigma_i$-invariant $(\phi_{i,\infty}-1)$-Labourie field.
\end{remark}

\noindent We first prove compactness of the family of representations $(\rho_m)_{m\in\mathbb{N}}$.
\begin{lemma}
\label{lemma:CompactnessOfRepresentations}
With the above hypotheses, $(\rho_m)_{m\in\mathbb{N}}$ is relatively compact in $\teich\times\teich$.
\end{lemma}
\begin{proof}For each $i$, for all $m$, and for all $\gamma\in\Pi$, let $l_{i,m}(\gamma)$ denote the length of the unique curve in the free homotopy class of $\sigma_i(\gamma)$ which is a geodesic with respect to the metric $g[u_{i,m}]$. For all $m$, let $\Lambda_m$ denote the limit set of $\rho_m$ in $\Ein$, let $K_m$ denote its convex hull, and let $\partial K_m^+$ (resp. $\partial K_m^-$) denote the future (resp. past) boundary of $K_m$. Recall that, for all $m$, $\partial K_m^\pm$ is intrinsically hyperbolic, and we denote by $\tau^\pm_m$ the point in Teichm\"uller space that it defines. For all $m$, let $\lambda^\pm_m(\gamma)$ denote the length of the unique geodesic in $\partial K_m^\pm$ in the free homotopy class of $\rho_m(\gamma)$. By \cite[Lemma B.1]{Pro}, for all $m$, and for all $\gamma$,
\begin{equation*}
\lambda_m^+(\gamma)\geq l_{1,m}(\gamma)\qquad\text{and}\qquad\lambda_m^-(\gamma)\geq l_{2,m}(\gamma)\ .
\end{equation*}
By \cite[Lemma 6.1]{Tam}, upon extracting a subsequence if necessary, we may suppose that $(\tau_m^\pm)_{m\in\mathbb{N}}$ subconverges in Teichm\"uller space to some limit $\tau_\infty^\pm$, say. By \cite[Lemma A.2]{BDMS}, $(\rho_m)_{m\in\mathbb{N}}$ therefore subconverges to some limit $\rho_\infty$, say, in $\teich\times\teich$, and this completes the proof.
\end{proof}

\noindent We now prove Lemma \ref{lemma:GlobalCompactness}.
\begin{proof}[Proof of Lemma \ref{lemma:GlobalCompactness}.]For all $m\in\mathbb{N}$, we denote the limit set of $\rho_m$ by $\Lambda_m$. We first show that there exists $B_1>0$ such that, for all $m$, $\Lambda_m$ is $B_1$-quasisymmetric. Indeed, by Lemma \ref{lemma:CompactnessOfRepresentations}, there exists $B_1\geq 1$ such that, for all $m\in\mathbb{N}$,
\begin{equation*}
d_{\mathrm{Lip}}(\rho_{m,r},\rho_{m,l})<\frac{1}{2}\ln(B_1)\ .
\end{equation*}
For all $m\in\mathbb{N}$, let $\phi_m:\hyp^2\rightarrow\hyp^2$ be a $\sqrt{B_1}$-bilipschitz, $(\rho_{m,r},\rho_{m,l})$ equivariant homeomorphism, let $\alpha_m:=\ideal\phi_m$ denote its extension to $\ideal\hyp^2$, and note that $\Lambda_m$ coincides with the graph of $\alpha_m$. For all $m$, since $\phi_m$ is $B_1$-quasiconformal, $\Lambda_m$ is $B_1$-quasisymmetric, as asserted.

The proof of \cite[Lemma 7.11]{BDMS}, carried out there in the case of constant curvature surfaces, adapts to the present case yielding a constant $B_2>0$ such that, for all $i$, and for all $m$,
\begin{equation*}
\|A_{i,m}\| \leq B_2\ ,
\end{equation*}
and the result now follows by \cite[Lemma 4.1]{Smith2021Weyl}.
\end{proof}

We are now ready to prove Theorem \ref{thm:ExistenceAndUniquenessForLabourieFields}
\begin{proof}[Proof of Theorem \ref{thm:ExistenceAndUniquenessForLabourieFields}.] Indeed, for each $i$, let $\phi_i'$ be a $\sigma_i$-invariant function such that $\phi_i'>1$, and let $u_i'$ denote the unique smooth, $\sigma_i$-invariant function such that $\kappa[u_i']=-\phi_i'$. For each $i$, let $(\phi_{i,t})_{t\in[0,1]}$ be a smooth family of smooth, $\sigma_i$-invariant functions such that, for all $t$, $\phi_{i,t}>1$ and
\begin{equation*}
\phi_{i,0}=\phi_i\qquad\text{and}\qquad\phi_{i,1}=\phi_i'\ .
\end{equation*}
For each $i$, and for all $t$, let $u_{i,t}$ denote the unique smooth, $\sigma_i$-invariant solution of $\kappa[u_{i,t}]=-\phi_{i,t}$. For each $i$, let $A_i$ be a $\sigma_i$-invariant $(\phi_{i, t}-1)$-Labourie field, and suppose that
\begin{equation*}
\rho^+[u_1,A_1] = \rho^-[u_2,A_2]\ .
\end{equation*}

We claim that $A_1$ and $A_2$ extend to unique continuous families $(A_{1,t})_{t\in[0,1]}$ and $(A_{2,t})_{t\in[0,1]}$ such that
\begin{enumerate}
\item for each $i$, and for all $t$, $A_{i,t}$ is a $\sigma_i$-invariant $(\phi_i-1)$-Labourie field; and
\item for all $t$,
\begin{equation*}
\rho^+[u_{1,t},A_{1,t}] = \rho^-[u_{2,t},A_{2,t}]\ .
\end{equation*}
\end{enumerate}
Indeed, consider the set $\mathcal{T}$ of all $T\in[0,1]$ for which unique extensions satisfying these properties exist over the interval $[0,T]$. By Lemma  \ref{lemma:GlobalCompactness}, this set is closed, and by Lemma \ref{lemma:ContinuityMethodOpenness}, it is open. Since it trivially contains the point $T=0$, it is non-empty, and thus, by connectedness, coincides with the whole of $[0,1]$. This proves the claim.

The result therefore holds for $(\phi_1,\phi_2)$ if and only if it holds for $(\phi_1',\phi_2')$. However, by \cite[Theorem 1.0.2]{cyclic2}, for all constant $k>0$, the result holds for
\begin{equation*}
(\phi_1,\phi_2) = (k+1,1/k+1)\ ,
\end{equation*}
and this completes the proof.
\end{proof}

We are now ready to prove Theorem \ref{rigiditytheorem}.
\begin{proof}[Proof of Theorem \ref{rigiditytheorem}.] For each $i$, let $\sigma_i:\Pi\rightarrow\PSL$ be a Fuchsian representation such that $S_i$ is conformally equivalent to $\mathbb{D}/\sigma_i(\Pi)$. For each $i$, let $u_i:\mathbb{D}\rightarrow\mathbb{R}$ be a $\sigma_i$-invariant function such that $S_i$ is isometric to $\big(\mathbb{D}/\sigma_i(\Pi),g[u_i]\big)$ and denote $\phi_i:=-\kappa[u_i]$.

By Theorem \ref{thm:ExistenceAndUniquenessForLabourieFields}, there exists a unique pair $(A_1,A_2)$ such that, for each $i$, $A_i$ is a $\sigma_i$-invariant $(\phi_i-1)$-Labourie field, and
\begin{equation*}
\rho^+[u_1,A_1]=\rho^-[u_2,A_2]=:\rho\ .
\end{equation*}
Let $\Lambda$ denote the limit set of $\rho$, let $D$ denote its invisible domain, and denote
\begin{equation*}
Y:=D/\rho(\Pi)\ .
\end{equation*}
This is a GHMC anti-de Sitter manifold into which $S_1$ (resp. $S_2$) isometrically embeds as a future convex (resp. past convex) smooth, spacelike immersion. The images of $S_1$ and $S_2$ together bound a convex subset $X$ of $Y$, and this proves existence. Since a different representation $\rho'$ necessarily yields a different pair $(A_1',A_2')$ of Labourie fields, it follows that $\rho$ is unique, and therefore so too is $X$. Finally, uniqueness of $(A_1,A_2)$ also implies uniqueness of the embeddings of $S_1$ and $S_2$ into $Y$. It follows that $X$ is indeed unique up to isometries preserving the marking, and this completes the proof.
\end{proof} 
\section{Measured laminations on the boundary of convex cores}
\label{sc:mess}

\subsection{Proof of Mess' conjecture}
In this section we prove Theorem \ref{tm:mess}, essentially following the strategy of \cite{DS}, with suitable adaptations to the anti-de Sitter case.

As in the introduction, we denote by $\mathcal{X}:=\mathcal{X}_g$ the space of marked $(2,1)$-dimensional GHMC anti-de Sitter spacetimes of genus $g$, up to isometries preserving the time orientation and the marking. We likewise denote by $\cML:=\cML_g$ the space of measured geodesic laminations of surfaces of genus $g$. We define
\begin{equation}
\label{eqn:DefinitionOfL}
\mathcal{L}:\mathcal{X}\rightarrow\cML\times\cML;X\mapsto(l_-,l_+)\ ,
\end{equation}
where, for any GHMC anti-de Sitter spacetime $X$, $l_-$ (resp. $l_+$) denotes the measured geodesic lamination of the past (resp. future) boundary component of its convex core. 

We recall that two measured geodesic laminations $(l_-,l_+)$ are said to \emph{fill} a surface whenever every connected component of the complement of the union of their supports is a bounded polygon. Equivalently, by \cite{Series2012LinesOfMinima}, $(l_-,l_+)$ fill the surface whenever there exists $\epsilon>0$ such that, the sum of the weights with respect to these laminations of any topologically non-trivial closed curve $\gamma$ is at least $\epsilon$. We denote by $\mathcal{FML}\subseteq\cML\times\cML$ the set of all pairs of measured geodesic laminations that fill the surface, and we note that this is an open subset of $\cML\times\cML$.

Theorem \ref{tm:mess} will be a consequence of the following two technical lemmas.

\begin{lemma} \label{lm:contractible}
For all $(l_-, l_+)\in\mc{FML} $, $\mathcal{L}^{-1}(l_-, l_+)$ is contractible.
\end{lemma}
\noindent The proof of this lemma, given in the next section, is similar to that of \cite{DS}, and relies on Corollary \ref{cr:III}.

Note that $\mathcal{X}$ is a real analytic manifold. Recall that any subset of a real analytic manifold is said to be \emph{real analytic} whenever it is everywhere locally modeled on a real analytic submanifold of $\R^n$.
\begin{lemma} \label{lm:analytic}
For all $(l_-, l_+)\in\mc{FML} $, $\mathcal{L}^{-1}(l_-, l_+)$ is real analytic.
\end{lemma}

These results together yield Theorem \ref{tm:mess}.
\begin{proof}[Proof of Theorem \ref{tm:mess}]
Fix $(l_-, l_+)\in \mc{FML} $. By Lemmas \ref{lm:contractible} and \ref{lm:analytic}, $\mathcal{L}^{-1}(l_-, l_+)$ is contractible and real analytic and, by \cite[Prop. 1.8]{BS3}, this set is compact. By \cite{BHa}, for every compact, $d$-dimensional, real analytic variety $V$,
\begin{equation*}
H_d(V,\mathbb{Z}_2)\neq 0\ .
\end{equation*}
In other words, the dimension of any such variety is equal to its homological dimension. Since $\mathcal{L}^{-1}(l_-, l_+)$ is contractible, it follows that $d=0$ and, since, in particular, it is connected, it thus consists of a single point. Since $(l_-,l_+)$ is arbitrary, it follows that $\mathcal{L}^{-1}$ is injective. By \cite{BS3}, this map is surjective, and this completes the proof.
\end{proof}

\subsection{Real analyticity}
\label{ss:Analyticity}

We first recall the Bers-type parametrization of $\mathcal{X}$ given by Mess in \cite{Mes}. As in Section \ref{ss:FuchsianRepresentations}, we denote by $\mathcal{T}:=\mathcal{T}_{\mathrm{rep},g}$ the Teichm\"uller space of Fuchsian representations of a genus $g$ compact surface group $\Pi:=\Pi_g$. We denote
\begin{equation*}
\Rho:\mathcal{X}\rightarrow\mathcal{T}\times\mathcal{T};X\rightarrow(\rho_l,\rho_r)\ ,
\end{equation*}
where $\rho_l$ (resp. $\rho_r$) denotes the left (resp. right) component of the holonomy of $X$ in the decomposition $SO_0(2,2)=PSL(2,\R)\times PSL(2,\R)$. In \cite{Mes}, Mess shows that $\Rho$ is a homeomorphism, and we use this homeomorphism to furnish $\mathcal{X}$ with a real analytic structure.

Given $\lambda\in\mc{ML} $, we denote respectively by $E_l^\lambda, E_r^\lambda:\mc T \to\mc T $ the left and the right earthquakes that it defines, and we recall that $E_r^\lambda=(E_l^\lambda)^{-1}$. By Thurston's Earthquake Theorem (see \cite{ThuE}), these maps are both bijections and, in fact, by \cite{Ker}, they are real analytic diffeomorphisms. By \cite{Mes} and \cite[Thm 1.5]{BS3}, given a non-Fuchsian $X\in\mathcal{X}$, with induced metrics $m_\pm$ and measured bending laminations $l_\pm$ on the future and past boundary components of its convex core,
\begin{equation}\label{eq:messearthquakes}
\rho_l=E_l^{l_+}(m_+)~,\qquad\rho_r=E_r^{l_+}(m_+)~,\qquad
\rho_l=E_r^{l_-}(m_-)~,\qquad\text{and}\qquad \rho_r=E_l^{l_-}(m_-)~,
\end{equation}
so that
\begin{equation}\label{eq:messrelations}
\rho_l \;=\; E_l^{2l_+}(\rho_r) \;=\; E_r^{2l_-}(\rho_r)~.
\end{equation}
Conversely (see \cite[Thm 1.5]{BS3}), given $(\rho_l,\rho_r)\in\mc T \times\mc T $ and $(l_-,l_+)\in\mc{ML} \times\mc{ML} $ satisfying \eqref{eq:messrelations}, the (unique) marked GHMC anti-de Sitter $(2,1)$-spacetime with holonomy data $(\rho_l,\rho_r)$ has measured bending laminations $l_-$ and $l_+$ on the past and future boundary components of its convex core. In other words, for all $(l_-,l_+)\in\mc{FML}$,
\begin{equation}\label{eq:fiber}
\mathcal{L}^{-1}(l_-,l_+) = \left\{ (\rho_l,\rho_r)\in\mc T \times\mc T  ~|~ \rho_l=E_l^{2l_+}(\rho_r) \mbox{ and } \rho_l=E_r^{2l_-}(\rho_r) \right\}~.
\end{equation}

We are now ready to proof Lemma \ref{lm:analytic}.
\begin{proof}[Proof of Lemma \ref{lm:analytic}] Let $F:=\mathcal{L}^{-1}(l_-,l_+)$ denote the fibre over $(l_-,l_+)$. By \eqref{eq:fiber}, $F$ is contained in the graph of the real-analytic map $E_r^{2l_-}$ over the factor $\rho_r$. More precisely, let $\pi_l,\pi_r:\mathcal{T}_{\mathrm{rep}}\times\mathcal{T}_{\mathrm{rep}}\rightarrow\mathcal{T}_{\mathrm{rep}}$ denote the respective projections onto the left and right components, and denote $F_r:=\pi_r(F)$. Then
\begin{equation}\label{eq:coincidence}
F_r \:= \left\{ \rho\in\mc T  ~|~ E_l^{2l_+}(\rho) = E_r^{2l_-}(\rho) \right\}~.
\end{equation}
Let $\Phi:\mc T \to\R^{6g-6}$ be a real analytic global chart (such as, for example, a system of Fenchel--Nielsen coordinates). Then $F_r$ is the zero set of the real analytic map
\begin{equation*}
\Psi(\rho):=\Phi\big(E_l^{2l_+}(\rho)\big) - \Phi\big(E_l^{2l_+}(\rho)\big)\ ,
\end{equation*}
and is thus a real analytic subset of $\mc T$. Since $F$ is the image of $F_r$ under the real analytic embedding $\rho\mapsto\big(E_l^{2l_+}(\rho),\rho\big)$, it is a real analytic subset of $\mc T \times\mc T \cong\mc X $, and this completes the proof.
\end{proof}

\subsection{Contractibility}
In order to prove Lemma \ref{lm:contractible}, we study the interrelationships between various geometric objects defined over the GHMC anti-de Sitter spacetime $X$. We refer the reader to \cite{FS4} for a thorough discussion of the following constructions.

For $X\in\mathcal{X}$, we denote by $C(X)$ its convex core, and we denote by $\Omega^-$ (resp. $\Omega^+$) the past (resp. future) component of $X\setminus C(X)$. In \cite{BBZ}, Barbot--B\'eguin--Zeghib show that $\Omega^\pm$ carries a unique foliation $(S_k^\pm)_{k\in(-\infty,-1)}$ by closed surfaces of constant intrinsic curvature $k$, with $k$ ranging over the interval $(-\infty,-1)$. This yields a smooth family of Bers-type parametrizations of $\mathcal{X}$, analogous to that constructed by Labourie for quasi-Fuchsian manifolds in \cite{Lab}, as follows. We denote by $\mathcal{T}_{\mathrm{hyp}}:=\mathcal{T}_{\mathrm{hyp},g}$ the Teichm\"uller space of marked hyperbolic metrics over a compact surface of genus $g$ and, for all $k^-,k^+\in(-\infty,-1)$, we define the smooth map $\mathcal{G}:=\mathcal{G}_{k^-,k^+}:\mathcal{X}\rightarrow\mathcal{T}_{\mathrm{hyp}}\times\mathcal{T}_{\mathrm{hyp}}$ by
\begin{equation*}
\mathcal{G}(X) := (S_{k^-}^-,S_{k^+}^+)\ ,
\end{equation*}
where here the metrics of $S_{k^-}^-$ and $S_{k^+}^+$ have been suitably rescaled in order to have curvature equal to $-1$. By Theorems \ref{rigiditytheorem} and \ref{infinitesimaltheorem1}, $\mathcal{G}$ is a smooth diffeomorphism, and thus smoothly para\-me\-trizes $\mathcal{X}$.

For all $k$, we denote respectively by $\fff_k^\pm$, $\sff_k^\pm$, and $\tff_k^\pm$ the first, second, and third fundamental form of $S_k^\pm$, and we denote by $B_k^\pm$ its shape operator. By Gauss' equation in mixed signature, for all $k$,
\begin{equation}
\det B_k^\pm \,=\, |k|-1 \,>\, 0\ .
\end{equation}
In particular, every leaf is strongly convex.

To every element $X$ of $\mathcal{X}$, we also associate two types of combinatorial objects as follows. First, for $X\in\mathcal{X}$, we denote
\begin{equation*}
\partial^-C(X) := C(X)\cap\overline{\Omega^-}\qquad\text{and}\qquad \partial^+C(X):= C(X)\cap\overline{\Omega^+}\ .
\end{equation*}
These are, respectively, the past and future boundary components of $C(X)$. Note that these sets coincide if and only if $X$ is Fuchsian. Although $\partial^\pm C(X)$ is not smoothly embedded, it carries the intrinsic structure of a hyperbolic metric, which we denote by $h^\pm$, together with a measured geodesic lamination, which we denote by $l^\pm$. In particular, when $X$ is Fuchsian, the measured geodesic laminations $l^\pm$ are both trivial.

The space $X$ also has a natural compactification $\overline{X}$, with two boundary components, which we denote by $\partial^-X$ and $\partial^+X$, and which we refer to respectively as its \emph{past} and \emph{future} singularities. When $X$ is non-Fuchsian, $\partial^\pm X$ carries the structure of a real tree $t^\pm$ furnished with an action $\rho^\pm$ of $\pi_1(X)$, where $\pi_1(X)$ here denotes the fundamental group of $X$. We recall that measured geodesic laminations and real trees are dual objects, and that, in the present case, for non-Fuchsian $X$, $(t^\pm,\rho^\pm)$ is the dual tree of the lamination $l^\pm$.

We now recall the duality of $\AdS$ discussed in the introduction (see also \cite[Section 4]{FS4}). For any $k$, the surface dual to $S_k^\pm$ is $S^\mp_{k^*}$, where
\begin{equation}
\label{eq:Kstar}
k^* := \frac{k}{\left|k\right|-1}\ .
\end{equation}
Furthermore, there is a natural diffeomorphism $e:S^\pm_k\rightarrow S^\mp_{k^*}$ and, using this diffeomorphism, the first and third fundamental forms of $S^\pm_k$ and $S^\mp_{k^*}$ are related by
\begin{equation*}
\tff^\pm_k = e^*\fff^\mp_{k^*}\qquad\text{and}\qquad\fff^\pm_k = e^*\tff^\mp_{k^*}\ .
\end{equation*}
This duality extends to the boundary components of $\Omega^\pm$, and coincides with the duality between measured geodesic laminations and real trees carrying $\pi_1(X)$ actions already discussed above.

With this framework established, we now review the interrelationships between these objects. First, we define the function $\delta:\Omega^\pm\rightarrow(0,\pi/2)$ such that, for all $x\in\Omega^\pm$, $\delta(x)$ is the length of the longest timelike geodesic from $x$ to $C(X)$. Consider now the function
\begin{equation*}
\Delta:(-\infty,-1)\rightarrow(0,\pi/2);k\mapsto\arctan\big(\sqrt{\left|k\right|-1}\big)\ .
\end{equation*}

\begin{lemma}\label{lm:trapping}
For all $X\in\mathcal{X}$, for all $k\in(-\infty,-1)$, and for all $x\in S_k^\pm$,
\begin{equation*}
\delta(x) \leq \Delta(k)\ .
\end{equation*}
In particular $(S_k^\pm)_{k\in(-\infty,-1)}$ converges in the Hausdorff sense to one of the boundary components of $C(\tilde{X})$ as $k$ tends to $-1$.
\end{lemma}
\begin{proof}
It is convenient to work in the universal cover $\tilde{X}$ of $X$. Let $\Lambda:=\ideal\tilde{X}$ denote its ideal boundary, and let $C(\tilde{X})$ denote its convex core which, we recall, is the convex hull of $\Lambda$. Choose $k\in(-\infty,-1)$, and let $\tilde{S}:=\tilde{S}^+_k$ denote the lift of $S^+_k$. Let $x_0\in\tilde{S}$ maximise $\delta$, let $\zeta$ denote the timelike geodesic through $x_0$ realizing $\delta$, and let $y_0$ denote the point at which $\zeta$ meets $\partial^+C(\tilde{X})$. Note that $\zeta$ is orthogonal to both $\tilde{S}$ and $\partial^+C(\tilde{X})$ at $x_0$ and $y_0$ respectively.

Let $P$ denote the totally geodesic plane orthogonal to $\zeta$ at $y_0$, and observe that $P$ is a support plane of $C(\tilde{X})$ at this point. Choose $\epsilon>0$, and let $P_\epsilon$ denote the parallel transport of $P$ away from $C(\tilde{X})$ by a distance $\epsilon$ along $\zeta$. Note that the ideal boundary $\ideal P_\epsilon$ of $P_\epsilon$ is disjoint from $\Lambda$. Let $P_\epsilon^*$ denote the point dual to $P_\epsilon$, and let $\Omega^+(P_\epsilon)$ denote the set of points lying in the future of $P_\epsilon$ and in the past of $P_\epsilon^*$. Define $\delta_\epsilon:\Omega^+(P_\epsilon)\rightarrow(0,\pi/2)$ such that, for all $x$, $\delta_\epsilon(x)$ is the length of the longest timelike geodesic from $x$ to $P_\epsilon$. Note that, along $\zeta$, $\delta_\epsilon$ coincides with $\delta-\epsilon$.

Since $\ideal P_\epsilon$ is disjoint from $\Lambda:=\ideal\tilde{S}$, $\delta_\epsilon$ attains a maximum at some point $x_0'$, say, of $\tilde{S}\cap\Omega^+(P_\epsilon)$. Note that, for all $\theta$, the level set $\delta_\epsilon^{-1}(\{\theta\})$ has constant intrinsic curvature equal to $\tan^2(\theta)-1$. It follows by the geometric maximum principle that
\begin{equation*}
\delta_\epsilon(x_0') \leq \Delta(k)\ .
\end{equation*}
Hence
\begin{equation*}
\delta(x_0) = \delta_\epsilon(x_0)+\epsilon \leq \delta_\epsilon(x_0') + \epsilon \leq \Delta(k) + \epsilon\ ,
\end{equation*}
and the result follows upon letting $\epsilon$ tend to zero.
\end{proof}

For all $c\in\pi_1(X)$, and for all $k$, we denote by $L_{\tff^\pm_k}(c)$ the length of the unique element of the free homotopy class of $c$ which is a geodesic with respect to the metric $\tff^\pm_k$. Likewise, for all such $c$, we denote by $i(t^\pm,c)$ the length of $c$ in the real tree $t^\pm$, and we denote by $i(l^\pm,c)$ the mass of $c$ with respect to the measured geodesic lamination $l^\pm$. Using the above defined duality, for all $c$, these quantities are related by
\begin{equation*}
i(l^\pm,c) = i(t^\mp,c)\ .
\end{equation*}
The following result is the anti-de Sitter analogue of \cite[Lemma 2.3]{DS}, and is a special case of the result \cite[Theorem 2.5]{Bel} of Belraouti.


\begin{lemma}\label{lm:convergence}
For every closed curve $c$ on $S$,
\begin{equation*}
L_{\tff_k^-}(c) \to i(l^+,c) = i(t^-,c)~,~~ L_{\tff_k^+}(c) \to i(l^-,c)=i(t^+,c)
\end{equation*}
as $k\to -1$, uniformly on compact subsets of $\mc X$.
\end{lemma}

We now recall the earthquake construction of \cite[Section 4.1]{DS}, which provides a continuous family of identifications of the space $\mathcal{T}_{\mathrm{hyp}}$ with $\mc{ML}$ in a way adapted to the degeneration $k\to-\infty$. We fix once and for all $m_0\in\mathcal{T}_{\mathrm{hyp}}$ and we define the \emph{earthquake map}
\begin{equation*}
\mr{eq}:=\mr{eq}_{m_0}:\mc{ML} \to\mathcal{T}_{\mathrm{hyp}};\lambda\mapsto E_l^\lambda(m_0)\ .
\end{equation*}
By Thurston's Earthquake Theorem \cite{ThuE}, this map is a homeomorphism. For $k<1$, we define the homeomorphism $u_k:\mathcal{T}_{\mathrm{hyp}}\rightarrow\mathcal{ML}$ by
\begin{equation*}
u_k(h) := \frac{1}{\sqrt{\left|k\right|}}\;\mr{eq}_{m_0}^{-1}\!\left(h\right)\ .
\end{equation*}
By \cite[Def.~4.1, Lemmas 4.2 and 4.3]{DS}, the following uniform convergence property holds for this family. Let $(k_m)_{m\in\mathbb{N}}$ be a sequence converging to $-\infty$. Let $T$ be a topological space, let $l:T\rightarrow\mathcal{ML}$ be a continuous family of measured geodesic laminations and, for all $m$, let $h_m:T\rightarrow\mathcal{T}_{\mathrm{hyp}}$ be a continuous family of marked hyperbolic metrics. Suppose that, for all $c\in\pi_1(X)$,
\begin{equation*}
\frac{1}{\sqrt{\left|k_m\right|}}L_{h_m(t)}(c) \rightarrow i(l(t),c)\ .
\end{equation*}
uniformly over compact subsets of $T$. Then, for all such $c$,
\begin{equation*}
i\big(u_{k_m}(h_m(t)),c\big) \rightarrow i(l(t),c)\ ,
\end{equation*}
uniformly over compact subsets of $T$. Note that, in \cite{DS}, this property is established for $\kappa=k/(k+1)$ with $k\in(-1,0)$. However, the proofs, which are based on the length estimate \cite[Lemma 4.4]{DS}, which is in turn a statement about earthquakes on surfaces, only use the fact that $(\kappa_m)_{m\in\mathbb{N}}\to-\infty$, and thus apply here with only minor modifications.

For $k<-1$, we define
\begin{equation*}
\Psi_k : \mc X  \;\to\; \mc{ML} \times\mc{ML} ~;\ X\mapsto\big(u_{k^*}(\tff_k^-(X)),u_{k^*}(\tff_k^+(X))\big)\ .
\end{equation*}
where $k^*$ is as in \eqref{eq:Kstar}. Note that $\Psi_k$ is also well-defined along the Fuchsian locus, where the foliations of $\Omega_\pm$ are by surfaces equidistant from some totally geodesic Cauchy surface.

\begin{lemma}\label{lm:embedding}
For all $k<-1$, $\Psi_k$ is a homeomorphism onto an open subset of $\mc{ML}\times\mc{ML}$.
\end{lemma}

\begin{proof}
The leaves $S_k^-$ and $S_k^+$ together bound a relatively compact collar neighbourhood of the convex core. This is a $(2,1)$-dimensional GHCC anti-de Sitter spacetime of genus $g$, which we henceforth denote by $X_k$. Its boundary components are smooth, spacelike, and strongly convex, and, by construction,
\begin{equation*}
\mc J(X_k(Y)) = \left(\tff_k^-(Y),\tff_k^+(Y)\right)\ ,
\end{equation*}
where $\mathcal{J}$ is as in the introduction.

\emph{Injectivity.} Suppose $\Psi_k(X)=\Psi_k(X')$. Since $u_{k^*}$ is injective, $\mc J(X_k)=\mc J(X'_k)$ so that, by Corollary \ref{cr:III}, $X_k=X_k'$. By \cite{Mes}, every GHCC anti-de Sitter spacetime admits a unique GHMC extension, from which it follows that $X=X'$, as desired.

\emph{Continuity.} It suffices to prove that $X\mapsto \tff_k^\pm(X)$ is continuous, which follows from the continuous dependence of the leaves of the foliation $(S_k^\pm)_{k\in(-\infty,-1)}$ on the ambient spacetime $X$. However, this follows from the following compactness argument. Given a sequence $(X_m)_{m\in\mathbb{N}}$ converging to $X_\infty$, by \ref{lm:trapping}, the leaves $(S_k^\pm(X_m))_{m\in\mathbb{N}}$ are trapped in collar neighbourhoods of fixed diameter around the convex cores. Since these surfaces have bounded geometry, the sequence $(S_k^\pm(X_m))_{m\in\mathbb{N}}$ subconverges to some limit $S_\infty$, say. By continuity, any such limit is a closed spacelike surface in $\Omega^\pm$ of constant intrinsic curvature equal to $k$. By uniqueness in \cite{BBZ}, the only possible limit point is $\Sigma^\pm_k(X_\infty)$, and continuity follows.

\emph{Openness.} Both $\mc X \cong\mc T \times\mc T $ and $\mc{ML} \times\mc{ML} $ are topological manifolds of dimension $12g-12$. Since $\Psi_k$ is continuous and injective, it follows by invariance of domain that its image is open, and this completes the proof.
\end{proof}

\begin{proof}[Proof of Lemma \ref{lm:contractible}]
Choose $(l_-,l_+)\in\mc{FML} $ and denote $F:=\psi^{-1}(l_-,l_+)$. Since the bending laminations of any Fuchsian spacetime vanish, and since $(l^-,l^+)$ fills $S$, it follows that $F$ contains no Fuchsian structure.

Under the Mess identification $\mc X \cong\mc T_{\mathrm{rep}} \times\mc T_{\mathrm{rep}}$, the map $\psi$ coincides with the map $\Phi$ of \cite[Def.~1.7]{BS3}. It is thus continuous and, by \cite[Prop.~1.8]{BS3}, it is proper. In particular $F$ is compact.

We now choose a finite family of closed curves $c_1,\dots,c_p$ on $S$ such that the function $l\mapsto (i(l,c_1),\dots,i(l,c_p))$ is a topological embedding of $\mc{ML} $ into $\R^p$. As in the proof of \cite[Lemma 4.9]{DS}, we equip $\mc{ML} \times\mc{ML} $ with the distance
\begin{equation*}
D\!\left((l_1,l_2),(l'_1,l'_2)\right) \,=\, \sum_{j=1}^p \left| i(l_1,c_j)-i(l_1',c_j)\right| + \left| i(l_2,c_j)-i(l_2',c_j)\right|\ .
\end{equation*}
Let $(k_n)_{n\in\mb N}\in(-\infty,-1)$ be any sequence converging to $-1$. By Lemma \ref{lm:convergence}, for each $j$, $L_{\tff^\pm_{K_n}(\cdot)}(c_j)\to i(l_\pm(\cdot),c_j)$ uniformly over compact subsets of $\mc X$. Thus, by the uniform convergence property described above,
\begin{equation*}
i\!\left( u_{k_m^*}\!\left(\tff^\pm_{k_m}(X)\right), c_j\right) \;\to\; i\!\left(l_\pm(X),c_j\right)\ ,
\end{equation*}
uniformly over compact subsets of $\mc X$. In other words, $\Psi_{k_n}\to\psi$ uniformly on compact subsets of $\mc X$ for the distance $D$.

By Lemma \ref{lm:embedding}, each $\Psi_{k_n}$ is an embedding of the manifold $\mc X$ into the manifold $\mc{ML} \times\mc{ML} $. It follows by Finney's theorem \cite[Thm~1]{Fin} (see \cite[Thm~4.6]{DS} for the version used here) that $F$ is a cellular subset of $\mc X$.

By Lemma \ref{lm:analytic}, $F$ is a compact real-analytic subset of $\mc X$. By Sullivan's local description \cite{Sul} of real-analytic sets and Hanner's criterion \cite{Han} (see also \cite[Lemma 4.8]{DS}), $F$ is an absolute neighbourhood retract (ANR). Since any cellular subset which is an ANR is contractible (see \cite[Lemma 4.7]{DS}), $F$ is contractible, and this completes the proof.
\end{proof}

\section{An application to landslides}
\label{sec:landslides}

Recall from \cite[Definition~1.5]{cyclic2} the \emph{landslide symmetry}: for $\theta\in(0,\pi)$ and $h_0\in\mc T $, the map $\mathcal S_{e^{i\theta},h_0}:\mc T \to\mc T $ sends $h$ to $\mathcal L^1_{e^{-i\theta}}(h_0,h_0^{\ast})$, where $h_0^{\ast}$ is the unique metric with $\mathcal L^1_{e^{i\theta}}(h_0,h_0^{\ast})=h$ (see \cite[Theorem~1.14]{cyclic}) and $\mathcal L$ is the landslide flow of \cite{cyclic}.

\begin{corollary}\label{cr:landslide}
  For all $\theta_-,\theta_+\in(0,\pi)$ and all $h_-,h_+\in\mc T $, the composition
  $$\mathcal S_{e^{i\theta_+},h_+}\circ\mathcal S_{e^{i\theta_-},h_-}:\mc T \longrightarrow\mc T $$
  has a unique fixed point.
\end{corollary}

The existence of a fixed point, and its uniqueness in the case $\theta_-+\theta_+=\pi$, are proved in \cite[Theorem~1.6]{cyclic2}, while uniqueness for $\theta_-+\theta_+\neq\pi$ is left open there (see the discussion after \cite[Theorem~1.6]{cyclic2}). It follows from Theorem~\ref{diffeotheorem}.

\begin{proof}
Let $K_\pm=-1/\cos^2(\theta_\pm/2)\in(-\infty,-1)$. By the dictionary between the landslide flow and constant curvature surfaces in GHMC anti-de Sitter $(2,1)$-spacetimes (\cite[Lemma 1.9 and Corollary~4.2]{cyclic}, and the translation carried out in \cite[Section 9]{cyclic2}), the fixed points of $\mathcal S_{e^{i\theta_+},h_+}\circ\mathcal S_{e^{i\theta_-},h_-}$ are in bijection with the GHMC anti-de Sitter $(2,1)$-spacetimes $Y$ for which the constant curvature $K_-$ (resp. $K_+$) surface in the past (resp. future) of the convex core has induced metric homothetic to $h_-$ (resp.\ $h_+$). It is therefore enough to show that there is exactly one such $Y$. However this is a direct application of Theorem \ref{rigiditytheorem}, when restricted to metrics of constant curvature.
\end{proof} 

\bibliographystyle{amsalpha}
\bibliography{biblio.bib}	
\end{document}